\documentclass[11pt]{article}

\usepackage{amsmath}
\usepackage{amsthm}
\usepackage{amssymb}
\usepackage[left=2.5cm,right=2.5cm,top=2.5cm,bottom=2.5cm]{geometry}
\usepackage{tikz}
\usetikzlibrary{positioning,shapes,calc,backgrounds,fit,matrix}
\usepackage{pgfplots}
\usetikzlibrary{plotmarks}
\usepackage{algpseudocode}
\usepackage{algorithm}
\usepackage{hyperref}

\usepackage{setspace}
\usepackage[T1]{fontenc} % gets rid of OMS/cmtt font warnings

\newtheorem{lemma}{Lemma}

\def\eps{\varepsilon}
\def\ttpi{\pi^{*}}
\pgfsetlayers{background,main}
\def\dx{\text{d}x}
\def\dy{\text{d}y}

\pgfplotsset{%
  every axis/.append style={%
    y label style={at={(0.1,1.0)},anchor=south west,rotate=-90,color=black},
    yminorgrids
  }
}
\pgfplotsset{compat=1.9}

\date{June 04, 2019}

\let\oldbibliography\thebibliography
\renewcommand{\thebibliography}[1]{%
  \oldbibliography{#1}%
  \setlength{\itemsep}{0pt}%
}

\newif\ifrosensamples
\rosensamplesfalse
\begin{document}

\title{Approximation and sampling of multivariate probability distributions in the tensor train decomposition}
% \author{\emph{[Blinded as requested by the Editor]}}
\author{Sergey Dolgov\footnote{University of Bath, United Kingdom. {\tt s.dolgov@bath.ac.uk}}, Karim Anaya-Izquierdo\footnote{University of Bath, United Kingdom. {\tt K.Anaya-Izquierdo@bath.ac.uk}}, Colin Fox\footnote{University of Otago, New Zealand. {\tt colin.fox@otago.ac.nz}}, and Robert Scheichl\footnote{University of Heidelberg, Germany. {\tt r.scheichl@uni-heidelberg.de}}}

\maketitle

\begin{abstract}
General multivariate distributions are notoriously expensive to sample from, particularly the high-dimensional posterior distributions in PDE-constrained inverse problems.
This paper develops a sampler for arbitrary continuous multivariate distributions that is based on low-rank surrogates in the tensor-train format, a methodology that has been exploited for many years for scalable, high-dimensional density function approximation in quantum physics and chemistry.
We build upon recent developments of the cross approximation algorithms in linear algebra to construct a tensor-train approximation to the target probability density function using a small number of function evaluations.
For sufficiently smooth distributions the storage required for accurate tensor-train approximations is moderate, scaling linearly with dimension.
In turn, the structure of the tensor-train surrogate allows sampling by an efficient conditional distribution method since marginal distributions are computable with linear complexity in dimension.
Expected values of non-smooth quantities of interest, with respect to the surrogate distribution, can be estimated using transformed independent uniformly-random seeds that provide Monte Carlo quadrature, or transformed points from a quasi-Monte Carlo lattice to give more efficient quasi-Monte Carlo quadrature.
Unbiased estimates may be calculated by correcting the transformed random seeds using a Metropolis--Hastings accept/reject step, while the quasi-Monte Carlo quadrature may be corrected either by a control-variate strategy, or by importance weighting.
We show that the error in the tensor-train approximation propagates linearly into the Metropolis--Hastings rejection rate and the integrated autocorrelation time of the resulting Markov chain; thus the integrated autocorrelation time may be made arbitrarily close to 1, implying that, asymptotic in sample size, the cost per effectively independent sample is one target density evaluation plus the cheap tensor-train surrogate proposal that has linear cost with dimension. These methods are demonstrated in three computed examples: fitting failure time of shock absorbers; a PDE-constrained inverse diffusion problem; and sampling from the Rosenbrock distribution. The delayed rejection adaptive Metropolis (DRAM) algorithm is used as a benchmark.
In all computed examples, the importance-weight corrected quasi-Monte Carlo quadrature performs best, and is more efficient than DRAM by orders of magnitude across a wide range of approximation accuracies and sample sizes. Indeed, all the methods developed here significantly outperform DRAM in all computed examples.
\end{abstract}


% MSC2010
% 65D15  	Algorithms for functional approximation
% 65D32  	Quadrature and cubature formulas
% 65C05  	Monte Carlo methods
% 65C40  	Computational Markov chains
% 65C60  	Computational problems in statistics
% 62F15  	Bayesian inference
% 15A69  	Multilinear algebra, tensor products
% 15A23  	Factorization of matrices

\section{Introduction}
We present an algorithm for efficient MCMC when the target distribution is a continuous multivariate distribution with known, tractable probability density function (PDF) $\pi(x)$ defined for $x$ in a region in $\mathbb{R}^d$. Beyond a fixed function-approximation phase, that has cost that scales linearly with dimension $d$, independent draws from $\pi$ cost (a fraction more than) one function evaluation per independent sample. We give a basic form of the algorithm that generates random samples from $\pi$, and variants that allow efficient quadrature using quasi-Monte Carlo constructions and/or importance weighting.

There are currently few general-purpose options for sampling from multivariate distributions with no special form, particularly if one is seeking a black-box sampler that does not require case-specific tuning. We commonly encounter such distributions as the posterior distribution in a Bayesian analysis of a nonlinear inverse problem~\cite{FN1997,FHC2013,scheichl-mlmcmc-2015} (also see the example in Section~\ref{sec:ff}), or as the marginal posterior distribution over hyperparameters in a linear-Gaussian inverse problem~\cite{FN2016}, see also~\cite{NCF2017}. This work is motivated by the desire to compute inference in those examples, though the samplers and quadrature methods we present here are applicable to arbitrary continuous distributions, that could be non-Gaussian, or multi-modal, and unnormalized; e.g., see the example in Section~\ref{sec:sae}.
In target applications, the aim of sampling is often to implement Monte Carlo integration to compute summary statistics of the posterior distribution over an unobserved quantity of interest (QoI). For applications in inverse problems, the state variable is typically high-dimensional and thus requiring computation of high-dimensional quadratures~\cite{stuart-bayes-2010}, even when the QoI is low-dimensional.

Efficient black-box samplers exist for some special classes of distributions. Most notable amongst multivariate distributions are multivariate normal (MVN) distributions, with fixed covariance or precision matrix, for which efficient, automatic sampling is available using stochastic variants of efficient algorithms for solving systems of equations in the covariance or precision matrix; methods based on \emph{direct} solvers, using Cholesky factoring, can be found in~\cite{Rue01, RueHeld05}, while more recently samplers based on accelerated \emph{iterative} solvers have been developed; see~\cite{FoxParker2017} and references therein. For non-Gaussian distributions, virtually all samplers are variants of Metropolis--Hastings (MH) MCMC with a random-walk proposal, of which there are many variants~\cite{handbook2011}. These algorithms are geometrically convergent, at best, so can be very slow for our target applications. Two black-box versions are the delayed-acceptance adaptive Metropolis  (DRAM) \cite{Haario-DRAM-2006} and the
t-walk~\cite{twalk}. Both of these algorithms require multiple evaluations of the target PDF per effectively independent output sample, with that number growing roughly linearly or worse with dimension, even for simple distributions such as MVN (see~\cite{NCF2017} for the cost of these algorithms).

Computational sampling from \emph{univariate} distributions is effectively a solved problem due to developments of the adaptive rejection sampler (ARS)~\cite{gilks1992adaptive}, such as independent doubly adaptive rejection Metropolis sampling (IA$^2$RMS)~\cite{MRL2015}. These algorithms approximate the univariate PDF using simple functions, with the approximation improving (adaptively) as the algorithm progresses to achieve efficient sampling. The ARS, that is restricted to log-convex PDFs, builds a piece-wise linear upper bound to the log of the PDF, hence bounds the PDF, to give an efficient proposal in a rejection sampler\footnote{Meyer \emph{et al.}~\cite{MeyerCaiPerron2008} used piecewise quadratic approximations to the log PDF giving piecewise Gaussian approximated PDF.}.
The IA$^2$RMS has no restriction on the PDF, and uses a sequence of simple function approximations to the PDF or log PDF,
such as piecewise-constant or piecewise-linear approximations, that converge in distribution to the PDF as the algorithm progresses.
Sampling from these approximations is easy in this univariate case, whether approximating the PDF or log PDF, using the inverse cumulative transformation method \cite{devroye-rvgen-1986,johnson1987multivariate,hormann-rvgen-2004}, with samples providing independence proposals to a Metropolis-Hastings accept/reject step that ensures the correct equilibrium distribution. Distributional convergence of the approximation implies that, asymptotic in sample size, just one PDF evaluation is required per independent sample\footnote{The Matlab package for IA$^2$RMS available at \url{http://a2rms.sourceforge.net/} is far more expensive than this minimal theoretical cost, besides not being robust.}.

The sampler developed here is inspired, to some extent, by IA$^2$RMS, in that it uses function approximation methods to approximate the multivariate PDF in a way that then allows cheap simulation from the approximation. Specifically, we use an interpolation in
tensor train (TT) representation, that may be made arbitrarily accurate, with sampling via the \emph{conditional distribution method} that is the multivariate extension of inverse cumulative transformation sampling for univariate distributions~\cite{johnson1987multivariate}. The conditional distribution method requires computing integrals of the multivariate PDF $\pi(x_1,\ldots,x_d)$, over subsets of variables $x_k,\ldots,x_d$ for $k=2,\ldots,d$, in order to obtain univariate marginal-conditional distributions. \emph{Per se}, this problem is as difficult as the original quadrature. By using the TT decomposition \cite{osel-tt-2011}, this integration can be performed efficiently, and each univariate marginal-conditional distribution can then be easily sampled using its inverse cumulative distribution function (CDF). Since the inverse cumulative transform is isoprobabilistic, the resulting samples are \emph{exact} for the interpolated probability tensor, which is however an \emph{approximation} to the original target PDF.
We provide bounds on the sampling error based on the approximation errors of the TT decomposition and discretization, and thus are able to trade accuracy for compute time.

An accurate approximation to the PDF allows the almost-exact samples to be used directly, while a less expensive approximation may be used to produce  independence proposals for a MH accept/reject step that `corrects' the distribution. The conditional distribution sampler may also be seeded with quasi-Monte Carlo points in the unit cube to implement quasi-Monte Carlo quadrature, that is corrected by a multi-level MCMC scheme, or by importance weighting. These variants are discussed in Section~\ref{sec:sampalg}. We find that the combination of quasi-Monte Carlo seed points combined with importance-weighted quadrature gives the best performance in computed examples.

The attraction of approximating the PDF in TT format is that the computational cost of the construction, the storage requirements, and the operations required for conditional distribution method sampling from the distributional approximation all scale \emph{linearly} with dimension; see Section~\ref{sec:ttdecomp}. In contrast, direct calculation or  na\"{\i}ve representations lead to \emph{exponential} cost for each of these tasks. This is a remarkable feature of the TT representation, and is why the recent introduction of low-rank hierarchical tensor methods, such as TT~\cite{ot-ttcross-2010,Os-mvk2-2011,osel-tt-2011,osel-constr-2013}, is a significant development in scientific computing for multi-dimensional problems.

Thus, the basic sampler we present here differs from IA$^2$RMS in two important aspects (beyond being able to handle multivariate distributions): we approximate the PDF and not the log PDF, and the sampler is not adaptive. The PDF is approximated because operations available on the TT representation, that have cost that scales linearly with dimension, include those required for performing the conditional distribution sampling, see Section~\ref{sec:ttdecomp}, while it is not clear how to perform sampling when the log PDF is approximated in the multivariate case. Further, current methods for TT representation do not include convenient and cheap schemes for updating a TT representation using a single new evaluation. Hence the algorithm we present consists of two steps; in a setup phase the TT approximation to $\pi(\cdot)$ is constructed, then that fixed approximation is used to generate samples. Hence, unlike the univariate samplers mentioned above, the TT approximation and samplers presented here are restricted to distributions with bounded, known support. While it is simple to define coordinate transformations $\mathbb{R}\mapsto[0,1]$ to represent a distribution on the (bounded) unit cube, efficient sampling still requires locating the appreciable support of the distribution; indeed, that is often a significant task when performing sampling. We do not consider such transformations here. Despite this restriction the method advances sample-based inference in some problems of substantial interest, as shown in the computed examples in Section~\ref{sec:ne}.

Approximation of the multivariate target distribution can be recommended for the following two cases:
First, the quantity of interest may be very poorly representable in the TT format,
and hence direct tensor product integration of the QoI, as suggested in \cite{Eigel-bayes-2017}, is not possible.
The most remarkable example is the indicator function, which occurs in the computation of the probability of an event.
If the jump of the indicator function is not aligned to the coordinate axes, the cost of its TT approximation might grow exponentially in the number of variables. Then, Monte Carlo quadrature becomes the only possibility, with the quadrature error depending on the particular distribution of the samples. When the target density function admits a TT approximation with a modest storage, the cumulative transform method can produce optimally distributed samples at a low cost.
Secondly, even when a fast growth of the TT storage prevents accurate computation of the density function, the TT-surrogate distributed
samples can still be used as proposals in the MH algorithm, or with importance weighting. Even a crude approximation to the PDF with $10\%$ error can produce the acceptance rate of $90\%$ and the integrated autocorrelation time of $1.2$, which is close enough to the best-possible practical MCMC. The relationship between approximation error and acceptance rate is formalized in Section~\ref{sec:ttmcmc}.

The paper is structured as follows: In Section~\ref{sec:cdm} we review the conditional sampling method used to sample from the multivariate TT-interpolated approximation. Some background on the TT decomposition is presented in Section~\ref{sec:ttdecomp}. A Metropolised algorithm that uses the TT surrogate for sampling from the target distribution is presented in Section~\ref{sec:sampalg}, as well as methods for unbiased quadrature that utilize a two-level algorithm, importance weighting, and quasi-Monte Carlo seed points. Several numerical examples are presented in Section~\ref{sec:ne}: Section~\ref{sec:sae} shows posterior estimation of a shock absorber failure probability; Section~\ref{sec:rf} demonstrates efficient sampling when the Rosenbrock function is the log target density, that is a synthetic `banana-shaped' PDF that presents difficulties to random-walk MCMC samplers; and Section~\ref{sec:ff} demonstrates posterior inference in a classical inverse problem in subsurface flow. In each of the numerical examples, scaling for the TT-based sampling and quadrature is shown, with comparison to DRAM \cite{Haario-DRAM-2006}, as well as (in Section~\ref{sec:ff}) to direct quasi-Monte Carlo quadrature.

\section{Conditional distribution sampling method}
\label{sec:cdm}

The \emph{conditional distribution method} \cite{devroye-rvgen-1986,johnson1987multivariate,hormann-rvgen-2004} reduces the task of generating a $d$-dimensional random vector into a sequence of $d$ univariate generation tasks.

Let $\left(X_1,\ldots,X_d\right)$ be a continuous random vector with a probability density function $\pi(x_1,\ldots,x_d)$.
To simplify the presentation, we assume in this section that $\pi$ is normalized.
The density function can be written as a product of conditional densities,
$$
\pi(x_1,\ldots,x_d) = \pi_1(x_1) \pi_2(x_2|x_1) \cdots \pi_d(x_d | x_1\ldots,x_{d-1}),
$$
where $\pi_k(x_k| x_1\ldots,x_{k-1})$ is a conditional density given by
\begin{equation}
\pi_k(x_k| x_1\ldots,x_{k-1}) = \frac{p_k(x_1,\ldots,x_k)}{p_{k-1}(x_1\ldots,x_{k-1})},
\label{eq:pi_k}
\end{equation}
in terms of the marginal densities,
\begin{equation}
 p_k = \int \pi(x_1, \ldots, x_{k-1}, x_k, x_{k+1},\ldots,x_d) \dx_{k+1} \cdots \dx_d,
\label{eq:p_k}
\end{equation}
where $k=1,\ldots,d$.
To simplify the notation we set $p_{0}=1$.
The conditional distribution method then generates  $\left(x_1,\ldots,x_d\right)\sim \pi$ by sampling from each of the univariate conditional densities in turn:
\begin{algorithmic}
\For{$k=1,2,\ldots,d$}
\State Generate $x_k\sim\pi_k(x_k| x_1\ldots,x_{k-1})$.
\EndFor
\end{algorithmic}
This follows by straightforward manipulation of the definitions of marginal and conditional distributions.

To generate the univariate samples in the algorithm above, we use the inverse cumulative transformation method. Thus, our algorithm coincides with the inverse Rosenblatt transformation~\cite{rosenblatt-1952} from the $d$-dimensional unit cube to the state-space of $\pi$. The standard conditional distribution method uses independent samples distributed uniformly in the unit cube as seeds for the transformation to produce independent draws from $\pi$. This generalizes the inverse cumulative transformation method for univariate distributions. Later, we will also use quasi-random points to implement quasi-Monte Carlo quadrature for evaluating expectations with respect to $\pi$.

When the analytic inverse of each univariate cumulative distribution function is not available, a straightforward numerical procedure is to discretize the univariate density on a grid, with approximate sampling carried out using a polynomial interpolation.
In that case, the normalization, i.e., the denominator in~\eqref{eq:pi_k}, is not necessary as normalization of the numerical approximation is evaluated, allowing sampling from an un-normalized marginal density~\eqref{eq:p_k}, directly.

The main difficulty with the conditional distribution method for multi-variate random generation is obtaining all necessary marginal densities, which requires the high-dimensional integral over $x_{k+1}\ldots x_d$ in \eqref{eq:p_k}. In general, this calculation can be extremely costly. Even a simple discretization of the argument of the marginal densities~\eqref{eq:p_k}, or the conditional-marginal densities~\eqref{eq:pi_k}, leads to exponential cost with dimension.

To overcome this cost, we \emph{precompute} an approximation of $\pi(x_1,\ldots,x_d)$ in a compressed representation that allows fast computation of integrals in \eqref{eq:p_k}, and subsequent sampling from the conditionals in \eqref{eq:pi_k}.
In the next sections, we introduce the TT decomposition and the related TT-cross algorithm \cite{ot-ttcross-2010} for building a TT approximation to $\pi$.
Moreover, we show that the separated form of the TT representation allows an efficient integration in~\eqref{eq:p_k}, with cost that scales \emph{linearly} with dimension.

\section{TT approximation of the target distribution}
\label{sec:ttdecomp}

Tensor decompositions trace back to the low-rank skeleton decompositions of matrices, which can in turn be computed by the singular value decomposition (SVD).
Any matrix $P \in \mathbb{R}^{n\times m}$ (e.g. a bi-variate discrete distribution) admits a SVD $P = U\Sigma V^\top$,
where $U,V$ are orthonormal matrices of singular vectors, and $\Sigma$ is a diagonal matrix of nonnegative singular values.
If the matrix is \emph{low-rank}, $r:=\mathrm{rank}~P<\min(m,n)$, the bottom right corner of $\Sigma$ is zero,
so we can \emph{truncate} the SVD to $U_r \Sigma_r V_r^\top$,
where $U_r,V_r$ contain only the first $r$ columns, and $\Sigma_r$ contains only the principal $r\times r$ submatrix.
However, we can also \emph{approximate} the given matrix $P$ by a truncated decomposition of lower rank; the Eckart-Young theorem \cite{gvl-matcomp-2013} ensures the optimality of the rank-$r$ SVD approximation among all possible rank-$r$ approximations.
Naturally, $U_r$ and $V_r$ contain only $(n+m)r$ elements in contrast to $nm$ elements in $P$.
This process can be extended to build low-rank decompositions of multivariate distributions, which we will describe next.

\subsection{Interpolated TT decomposition}
Throughout the paper, we approximate the target PDF by an interpolated TT decomposition \cite{osel-tt-2011},
\begin{equation}
\begin{split}
& \pi(x_1,\ldots,x_d) \approx \tilde\pi(x_1,\ldots,x_d) \\
& = \sum_{\alpha_0,\ldots,\alpha_d=1}^{r_0,\ldots,r_d}\pi^{(1)}_{\alpha_0,\alpha_1}(x_1)\pi^{(2)}_{\alpha_1,\alpha_2}(x_2) \cdots \pi^{(d)}_{\alpha_{d-1},\alpha_d}(x_d),
\end{split}
\label{eq:tt}
\end{equation}
that is a sum of products of the univariate functions $\pi^{(k)}_{\alpha_{k-1},\alpha_k}(x_k)$, $k=1,2,\ldots,d$ indexed by $\alpha_k=1,\ldots,r_k$.
The $r_k$, $k=0,\ldots,d$, are called \emph{TT ranks}, with $r_0=r_d=1$ (because $\pi$ is scalar valued) but $r_1,\ldots,r_{d-1}$ can be larger. The efficiency of this representation relies on the TT ranks being bounded by some (smallish) number $r$, as discussed later.

The TT decomposition natively represents a tensor, or $d$-dimensional array of values. The function approximation \eqref{eq:tt} is obtained by first approximating the tensor that results from discretizing the PDF $\pi(x_1,\ldots,x_d)$ by collocation on a tensor product of univariate grids. Let $x_k^{i_k} \in \mathbb{R}$, with $i_k=1,\ldots,n_k$ and $x_k^{1} < \cdots < x_k^{n_k}$, define independent univariate grids in each variable, and let $\hat \pi(i_1,i_2,\ldots,i_d) = \pi(x_1^{i_1},x_2^{i_2},\ldots,x_d^{i_d})$. The TT representation is
\begin{equation}
\begin{split}
 & \hat \pi(i_1,i_2,\ldots,i_d) \\
 & =  \sum_{\alpha_0,\ldots,\alpha_d=1}^{r_0,\ldots,r_d}\hat\pi^{(1)}_{\alpha_0,\alpha_1}(i_1)\hat\pi^{(2)}_{\alpha_1,\alpha_2}(i_2) \cdots \hat\pi^{(d)}_{\alpha_{d-1},\alpha_d}(i_d)
\end{split}
 \label{eq:ttdisc}
\end{equation}
with \emph{TT blocks} $\hat\pi^{(k)}$. Each TT block is a collection of $r_{k-1}r_k$ vectors of length $n_k$, i.e., $\hat \pi^{(k)}(i_k) = \pi^{(k)}(x_k^{i_k})$ is a three-dimensional tensor of size $r_{k-1} \times n_k \times r_k$. If we assume that all $n_k \le n$ and $r_k \le r$ for some uniform bounds $n,r \in \mathbb{N}$, the storage cost of \eqref{eq:ttdisc} can be estimated by $dnr^2$ which is linear in the number of variables. In contrast, the number of elements in the \emph{tensor} of nodal values $\hat\pi(i_1,\ldots,i_d)$ grows exponentially in $d$ and quickly becomes prohibitively large with increasing $d$.

The continuous approximation of $\pi$~\eqref{eq:tt} is given by a piecewise polynomial interpolation of nodal values, or TT blocks.
For example, in the linear case we have
$$
\pi^{(k)} = \frac{x_k - x_k^{i_k}}{x_k^{i_k+1} - x_k^{i_k}} \cdot \hat \pi^{(k)}(i_k+1) + \frac{x_k^{i_k+1} - x_k}{x_k^{i_k+1} - x_k^{i_k}} \cdot \hat \pi^{(k)}(i_k),
$$
for $x_k^{i_k} \le x_k \le x_k^{i_k+1},$ which induces the corresponding multi-linear approximation $\tilde\pi$ of $\pi$ in \eqref{eq:tt}.

If the individual terms $\pi^{(k)}_{\alpha_{k-1},\alpha_k}(x_k)$ are normalized PDFs,
the TT approximation in~\eqref{eq:tt} may be viewed as a mixture distribution.
However, the TT decomposition can be more general and may also include negative terms.
Moreover, at some locations where $\pi(x)$ is close to zero the whole approximation $\tilde\pi(x)$ may take (small) negative values.
This will be circumvented by explicitly taking absolute values in the conditional distribution sampling method, see Sec. \ref{sec:ttcd}.

The interpolated TT approximation to $\pi$ in  \eqref{eq:tt} required several choices. First a coordinate system must be chosen, then an ordering of coordinates, then a rectangular region that contains the (appreciable) support of the PDF, and then univariate grids for each coordinate within the rectangular region. Each of these choices affects the TT ranks, and hence the efficiency of the TT representation in terms of storage size versus accuracy of the approximation, that is also chosen; see later. In this sense, the sampler that we develop is not `black box'. However, as we demonstrate in the computed examples, an unsophisticated choice for each of these steps already leads to a computational method for sampling and evaluating expectations that is substantially more efficient than existing MCMC algorithms. Smart choices for each of these steps could lead to further improvements.

The rationale behind the independent discretization of all variables is the rapid convergence of tensor product Gaussian quadrature rules.
If $\pi(x)$ is analytic with respect to all variables, the error of the Gaussian quadrature converges exponentially in $n$.
A straightforward summation of $n^d$ quadrature terms would imply a cost of $\mathcal{O}(|\log\eps|^{d})$ for accuracy $\eps$.
In contrast, the TT ranks often depend logarithmically on $\eps$ under the same assumptions on $\pi(x)$ \cite{tee-tensor-2003,khor-rstruct-2006,uschmajew-approx-rate-2013},
leading to $\mathcal{O}(d|\log\eps|^{3})$ cost of the TT integration, since the integration of the TT decomposition factorizes into one-dimensional integrals over the TT blocks.
This can also be significantly cheaper than the $\mathcal{O}(\eps^{-2})$ cost of Monte Carlo quadrature.

In general, it is difficult to deduce sharp bounds for the TT ranks. Empirically, low ranks occur in the situation of ``weakly'' dependent variables.
For example, if $x_1,\ldots,x_d$ correspond to independent random quantities, the PDF factorizes into a single product of univariate densities,
which corresponds to the simplest case, $r=1$ in \eqref{eq:tt}.
Thus, a numerical algorithm that can robustly reveal the ranks is indispensable.

\subsection{TT-cross approximation}
A quasi-optimal approximation of $\hat \pi$ for a given TT rank, in the Frobenius norm, is available via the truncated singular value decomposition (SVD)~\cite{osel-tt-2011}. However, the SVD requires storage of the full tensor which is not affordable in many dimensions. A practical method needs to be able to compute the representation~\eqref{eq:tt} using \emph{only a few} evaluations of $\pi$. A workhorse algorithm of this kind is the alternating TT-cross method \cite{ot-ttcross-2010}. That builds on the skeleton decomposition of a matrix~\cite{gore-tyrt-zam-1997}. It represents an $n\times m$ matrix $P$ of rank $r$ as the \emph{cross} (in MatLab-like notation)
\begin{equation}
 P=P(:,\mathcal{J})P(\mathcal{I},\mathcal{J})^{-1}P(\mathcal{I},:)
 \label{eq:cross}
\end{equation}
of $r$ columns and rows, where $\mathcal{I}$ and $\mathcal{J}$ are two index sets of cardinality $r$ such that $P(\mathcal{I},\mathcal{J})$ (the intersection matrix) is nonsingular. If $r\ll n,m$, this decomposition requires computing only $(n+m-r)r\ll nm$ elements of the original matrix. The SVD may be used for choosing the cross~\eqref{eq:cross}, though with greater cost, as noted above.

The TT-cross approximation may now be constructed by reducing the sequence of \emph{unfolding} matrices
$\hat \pi_k=[\hat \pi(i_1,\ldots,i_k;i_{k+1},\ldots,i_d)]$, that have the first $k$ indices grouped together to index rows, and the remaining indices grouped to index columns. We begin with $\hat \pi_1$.

We start with a set $\mathcal{I}_{>1} = \{(i_2^{\alpha_1},\ldots,i_d^{\alpha_1})\}_{\alpha_1=1}^{r_1}$ of $r_1$ $(d-1)$-tuples
such that $\hat\pi(:,\mathcal{I}_{>1})$ forms a ``good'' basis for the rows of $\hat \pi_1$ (in the $i_1$ variable) and choose a set $\mathcal{I}_{<2} = \{i_1^{\alpha_1}\}_{\alpha_1=1}^{r_1}$ of $r_1$ row indices such that the  \emph{volume} (the modulus of the determinant) of the $r_1 \times r_1$ submatrix $\hat\pi(\mathcal{I}_{<2},\mathcal{I}_{>1})$ is \emph{maximized}.
This can be achieved in $\mathcal{O}(nr_1^2)$ operations using the \emph{maxvol} algorithm \cite{gostz-maxvol-2010}.
The first discrete TT block $\hat \pi^{(1)}$ is then assembled from the rectangular $n \times r_1$ matrix $\hat\pi(:,\mathcal{I}_{>1})\hat\pi(\mathcal{I}_{<2},\mathcal{I}_{>1})^{-1}$, and the reduced tensor $[\hat\pi_{>1}({\alpha_1}, i_2,\ldots,i_d)] = [\hat\pi(i_1^{\alpha_1}, i_2,\ldots,i_d)]$ is passed on to the next step of the TT cross.
In a practical algorithm, to ensure numerical stability all these operations are actually carried out using QR-decompositions of the  matrices~\cite{ot-ttcross-2010}.

In the $k$-th step, we assume that we are given the reduction $\hat\pi_{>k-1}(\alpha_{k-1},i_k,\ldots,i_d)$ from the previous step, as well as two sets $\mathcal{I}_{<k} = \{(i_1^{\alpha_{k-1}},\ldots,i_{k-1}^{\alpha_{k-1}})\}_{\alpha_{k-1}=1}^{r_{k-1}}$ and $\mathcal{I}_{>k} = \{(i_{k+1}^{\alpha_k},\ldots,i_d^{\alpha_k})\}_{\alpha_k=1}^{r_k}$ containing, resp., $r_{k-1}$ $(k-1)$-tuples and $r_k$ $(d-k)$-tuples.
The unfolding tensor $[\hat\pi_{>k-1}({\alpha_{k-1}}, i_k;~\mathcal{I}_{>k})]$ can then be seen as a $r_{k-1}n \times r_k$ rectangular matrix and the \emph{maxvol} algorithm can be applied again to produce a set of row positions $\{\alpha_{k-1}^{\alpha_k}, i_k^{\alpha_k}\}_{\alpha_k=1}^{r_k}$, which upon replacing $\alpha_{k-1}^{\alpha_k}$ with the corresponding indices from $\mathcal{I}_{<k}$ leads to the next index set $\mathcal{I}_{<k+1}=\{(i_1^{\alpha_{k}},\ldots,i_{k}^{\alpha_{k}})\}_{\alpha_{k}=1}^{r_{k}}$.
The induction is completed by taking $\hat\pi^{(d)} = \hat \pi_{>d-1}$.

This process can be also organized in the form of a binary tree, which gives rise to the so-called hierarchical Tucker cross algorithm \cite{Ballani-HTUQOut-2015}.
In total, we need $\mathcal{O}(dnr^2)$ evaluations of $\pi$ and $\mathcal{O}(dnr^3)$ additional operations for the computation of the maximum volume matrices.

The choice of the univariate grids, $x_k^{1} < \cdots < x_k^{n_k}$, and of the initial index sets $\mathcal{I}_{>k}$ can be crucial.
In this paper we found that a uniform grid in each coordinate was sufficient, with even relatively coarse grids resulting in efficient sampling algorithms; see the numerical examples for details. Given any easy to sample reference distribution (e.g. uniform or Gaussian), it seems reasonable to initialize $\mathcal{I}_{>k}$ with independent realizations of that distribution (we could also expand the \emph{grids} with reference samples, though we did not do that).
If the target function $\pi$ admits an \emph{exact} TT decomposition with TT ranks not greater than $r_1,\ldots,r_{d-1}$,
and all unfolding matrices have ranks not smaller than the TT ranks of $\pi$,
the cross iteration outlined above reconstructs $\hat\pi$ \emph{exactly} \cite{ot-ttcross-2010}.
This is still a rare exception though, since most functions have infinite exact TT ranks, even if they can be \emph{approximated} by a TT decomposition with a small error and low ranks.
Nevertheless, the cross iteration, initialized with slightly \emph{overestimated} values $r_1,\ldots,r_{d-1}$, can deliver a good approximation, if a function is regular enough \cite{Ballani-HTUQOut-2015,ds-alscross-2019}.

This might be not the case for \emph{localized} probability density functions.
For example, for a heavy-tailed function $(1+x_1^2+\cdots+x_d^2)^{-1/2}$ one might try to produce $\mathcal{I}_{>k}$ from a uniform distribution in a cube $[0,a]^d$ with a sufficiently large $a$.
However, since this function is localized in an exponentially small volume $[0,\eps]^d$, uniform index sets deliver
a poor TT decomposition, worse for larger $a$ and $d$.

In this situation it is crucial to use fine grids and refine the sets $\mathcal{I}_{<k},\mathcal{I}_{>k}$ by conducting \emph{several} TT cross iterations, going back and forth over the TT blocks and optimizing the sets by the maxvol algorithm.
For example, after computing $\hat \pi^{(d)} = \hat \pi_{>d-1}$, we ``reverse'' the algorithm and consider the unfolding matrices with indices
$\{(i_1^{\alpha_{d-1}},\ldots,i_{d-1}^{\alpha_{d-1}})\}_{\alpha_{d-1}=1}^{r_{d-1}} = \mathcal{I}_{<d}$.
Applying the maxvol algorithm to the \emph{columns} of a $r_{d-1} \times n$ matrix $\hat \pi^{(d)}$, we obtain a \emph{refined} set of points $\mathcal{I}_{>d-1} = \{i_d^{\alpha_{d-1}}\}_{\alpha_{d-1}=1}^{r_{d-1}}$.
The recursion continues from $k=d$ to $k=1$, optimizing the right sets $\mathcal{I}_{>k}$, while taking the left sets $\mathcal{I}_{<k}$ from the previous (forward) iteration.
After several iterations, both $\mathcal{I}_{<k}$ and $\mathcal{I}_{>k}$ will be optimized to the particular target function,
even if the initial index sets gave a poor approximation.

This adaptation of points goes hand in hand with the \emph{adaptation of ranks}.
If the initial ranks $r_1,\ldots,r_{d-1}$ were too large for the desired accuracy, they can be reduced.
However, we can also \emph{increase} the ranks by computing the unfolding matrix $\left[\hat\pi(\mathcal{I}_{<k}, i_k;~i_{k+1}^{\alpha_k},\ldots,i_{d}^{\alpha_k})\right]$ on some \emph{enriched} index set $\{(i_{k+1}^{\alpha_k},\ldots,i_d^{\alpha_k})\}_{\alpha_{k}=1}^{r_{k}+\rho}$, by augmenting the original index set $\mathcal{I}_{>k}$ with
an \emph{auxiliary} set $\mathcal{I}_{>k}^{aux}$
and increasing the $k$-th TT rank from $r_k$ to $r_k+\rho$.
The auxiliary set can be chosen at random~\cite{Os-mvk2-2011} or using a surrogate for the error~\cite{ds-amen-2014}. % white-dmrg1c-2005
The pseudocode of the entire TT cross method is listed in Algorithm \ref{alg:cross}.
For uniformity, we let $\mathcal{I}_{<1} = \mathcal{I}_{>d} = \emptyset$.
\begin{algorithm}[t]
\caption{TT cross algorithm for TT approximation of $\pi$.}
\label{alg:cross}
\begin{algorithmic}[1]
 \Require Initial index sets $\mathcal{I}_{>k}$, rank increasing parameter $\rho\ge 0$, stopping tolerance $\delta>0$ and/or maximum number of iterations $\mathrm{iter}_{\max}$.
 \Ensure TT blocks of an approximation $\tilde \pi(x) \approx \pi(x)$.
 \While{$\mathrm{iter}<\mathrm{iter}_{\max}$ and $\|\tilde\pi_{\mbox{iter}}-\tilde\pi_{\mbox{iter}-1}\|>\delta \|\tilde\pi_{\mbox{iter}}\|$}
    \For{$k=1,2,\ldots,d$} \Comment{Forward iteration}
      \State (Optionally) prepare enrichment set $\mathcal{I}_{>k}^{aux}$.
      \State Compute $r_{k-1}n \times r_k$ unfolding $\hat\pi(\mathcal{I}_{<k},i_k;~\mathcal{I}_{>k})$.
       \State Compute $\mathcal{I}_{<k+1}$ by \emph{maxvol} alg. and truncate.
    \EndFor
    \For{$k=d,d-1,\ldots,1$} \Comment{Backward iteration}
      \State (Optionally) prepare enrichment set $\mathcal{I}_{<k}^{aux}$.
      \State Compute $r_{k-1} \times nr_k$ unfolding $\hat\pi(\mathcal{I}_{<k}~;i_k,\mathcal{I}_{>k})$.
    \State Compute $\mathcal{I}_{>k-1}$ by \emph{maxvol} alg. and truncate.
    \EndFor
 \EndWhile
\end{algorithmic}
\end{algorithm}

Systematically using the enrichment scheme, we can even employ a different approach moving
away from truncating ranks. Instead, we start with a low-rank initial guess and increase the ranks
until the desired accuracy is met. We have found that this approach is often more accurate in
numerical experiments. The relative cost of the two approaches depends on the application.

\section{Sampling Algorithms based on TT Surrogates}
\label{sec:sampalg}

\subsection{Conditional Distribution Sampling (TT-CD)}
\label{sec:ttcd}
One of the main contributions of this paper is to show that conditional distribution method is feasible, and efficient, once a PDF has been put into TT format. This section presents those calculations.

% Because the TT approximation $\tilde\pi\approx\pi$ \emph{could be negative} at some locations, we actually use $\ttpi = |\tilde\pi|$ in the conditional sampling method, that has non-negative (un-normalized) marginal distributions $\tilde p_k$.

First, we describe the computation of the marginal PDFs $p_k$, defined in~\eqref{eq:p_k}, given $\pi$ in a TT format \eqref{eq:tt}.
Note that integrals over the variable $x_p$ appear in all conditionals~\eqref{eq:p_k} with $k<p$.
The TT format allows to compute the $r_{k-1} \times 1$ vector $P_k$ required for evaluating the marginal PDF $p_{k-1}$ by the following algorithm.
\begin{algorithmic}[1]
 \State Initialize $P_{d+1}=1$
 \For{$k=d,d-1,\ldots,2$}
    \State $(P_{k})_{\alpha_{k-1}} = \sum\limits_{\alpha_k=1}^{r_k} \left(\int\limits_{\mathbb{R}} \pi^{(k)}_{\alpha_{k-1},\alpha_k}(x_k) \dx_k \right) (P_{k+1})_{\alpha_k}$
 \EndFor
\end{algorithmic}
Since $\pi^{(k)}(x_k) \in \mathbb{R}^{r_{k-1} \times r_k}$ for each fixed $x_k$, the integral $\int \pi^{(k)}(x_k) \dx_k$ is a $r_{k-1} \times r_k$ matrix, where $\alpha_{k-1}$ is the row index, and $\alpha_k$ is the column index. Hence, we can write Line 3 as the matrix-vector product,
$$
P_{k} = \left( \int_{\mathbb{R}} \pi^{(k)}(x_k) \dx_k \right)  P_{k+1}.
$$
Assuming $n$ quadrature points for each $x_k$, and the uniform rank bound $r_k \le r$, the asymptotic complexity of this algorithm is $\mathcal{O}(dnr^2)$.

The first marginal PDF is approximated by $p_1^*(x_1) = |\pi^{(1)}(x_1) P_2|$.
We take the absolute value because the TT approximation $\tilde\pi$ (and hence, $\pi^{(1)}(x_1) P_2$) may be negative at some locations.
In the $k$-th step of the sampling procedure, the marginal PDF also requires the first $k-1$ TT blocks, restricted to the components of the sample that are already determined\footnote{Here again, we treat $\pi^{(k)}(x_k)$ as a $r_{k-1} \times r_k$ matrix, such that the product is valid.},
$$
p_k^*(x_k) = \left|\pi^{(1)}(x_1) \cdots \pi^{(k-1)}(x_{k-1}) \pi^{(k)}(x_k) P_{k+1}\right|.
$$
However, since the loop goes sequentially from $k=1$ to $k=d$, the sampled TT blocks can be accumulated in the same fashion as the integrals $P_k$.
Again, we take the absolute value to ensure positivity.
The overall method for drawing $N$ samples is written in Algorithm \ref{alg:samp}.
Note that if $\tilde\pi$ is negative at any points, the actual density $\ttpi$ at $x^\ell$, which is the product of marginal PDFs computed in each step, may slightly differ from $\tilde\pi$.
\begin{algorithm}[t]
\caption{CD sampling from a TT decomposition of a PDF}
\label{alg:samp}
\begin{algorithmic}[1]
 \Require TT blocks $\pi^{(1)},\ldots,\pi^{(d)}$ of the approximation $\tilde\pi$, uniformly distributed seeds $\{(q_1^\ell,\ldots,q_d^\ell)\}_{\ell=1}^{N} \sim \mathcal{U}(0,1)^d$.\vspace{0.25ex}
 \Ensure $\ttpi$-distributed samples $\{(x_1^\ell,\ldots,x_d^\ell)\}_{\ell=1}^{N}$, as well as PDF values $\ttpi(x^\ell) = p_1^*(x_1^\ell) \cdots p_d^*(x_d^\ell)$.
 \State Initialize $P_{d+1}=1$.
 \For{$k=d,d-1,\ldots,2$}
    \State Compute $P_{k} = \int_{\mathbb{R}} \pi^{(k)}(x_k) \dx_k \cdot P_{k+1}$.
 \EndFor
 \State Initialize $\Phi_{1}=1 \in \mathbb{R}^{N}$.
 \For{$k=1,2,\ldots,d$}
    \State\label{alg:samp:Psi} Prepare deterministic part $\Psi_k(x_k) = \pi^{(k)}(x_k) P_{k+1}$.
    \For{$\ell=1,\ldots,N$}
       \State\label{alg:samp:pk} Compute marginal PDF $p_k^*(x_k) = \left|\Phi_k(\ell,:)\Psi_k(x_k)\right|$,
       \State\label{alg:samp:Ck} marginal CDF $C_k(x_k) = \int_{-\infty}^{x_k} p_k^*(y_k) \dy_k/{\int p_k^* \dx_k}$.
       \State\label{alg:samp:xk} Sample $x_k$ component, $x_k^\ell = C_k^{-1}(q_k^\ell)$.
       \State Compute $\Phi_{k+1}(\ell,:) = \Phi_k(\ell,:) \pi^{(k)}(x_k^\ell)$.
    \EndFor
 \EndFor
\end{algorithmic}
\end{algorithm}

The error induced by taking the absolute values in Line~\ref{alg:samp:pk} of Alg.~\ref{alg:samp} is of the order of the TT approximation error.
The approximate marginal probability $\tilde p_k(x_k) = \Phi_k(\ell,:) \Psi_k(x_k)$ is produced from $\tilde\pi(x)$ by integration,
hence if $\pi(x)-\tilde\pi(x) = \mathcal{O}(\eps)$ due to the TT approximation, we also have $|\tilde p_k(x_k) - p_k(x_k)| \le C \epsilon \|p_k\|_\infty$ for the marginals, for some $C>0$ independent of $\eps$, where $\|p_k\|_\infty := \operatorname{ess}\sup_{\xi_k} p_k(\xi_k)$.
Then, for all $x_k$ that satisfy $p_k(x_k) \ge C \eps \|p_k\|_\infty$, we have
$$
\tilde p_k(x_k) \ge C \eps \|p_k\|_\infty - |p_k(x_k) - \tilde p_k(x_k)| \ge 0.
$$
Hence, $\tilde p_k$ can only be negative where $p_k$ is small and we have
 $-C \eps \|p_k\|_\infty \le \tilde p_k(x_k) \le 0$.
The error in taking the modulus in Line~\ref{alg:samp:pk} of Alg.~\ref{alg:samp} can then be estimated as follows:
$$
|p^*_k(x_k) - \tilde p_k(x_k)| \; \le \; \left\{\begin{array}{ll}2 C \eps \|p_k\|_\infty\,, & \ \text{for} \ \tilde p_k(x_k)<0, \\ 0, & \ \mbox{otherwise}.\end{array}\right.
$$

The sample-independent prefactor of the marginal PDF in Line \ref{alg:samp:Psi} requires $\mathcal{O}(dnr^2)$ operations.
The marginal PDF in Line \ref{alg:samp:pk} can then be computed with $\mathcal{O}(dNnr)$ cost.
The cost of the CDF computation in Line \ref{alg:samp:Ck} depends on the quadrature scheme used.
For a piecewise spline approximation or for the barycentric Gauss formula the cost for both $C_k$ and $C_k^{-1}$
is $\mathcal{O}(dNn)$. The complexity of computing the conditional PDF values $\Phi_{k+1}$ depends on how $\tilde\pi$ is interpolated onto $x_k^\ell$. Global Lagrange interpolation requires $\mathcal{O}(nr^2)$ cost per sample, whereas local interpolation is independent of $n$, requiring only $\mathcal{O}(r^2)$ operations.
In our numerical experiments, we have found piecewise linear interpolation on a uniform grid to be sufficient.
In summary, the total complexity is
\begin{equation}
\mathcal{O}\big(dr (nr + N(n+r))\big)\,.
\label{eq:cost}
\end{equation}

\subsection{Metropolis--Hastings correction (TT-MH)}
\label{sec:ttmcmc}

For the TT-CD sampling procedure in Alg.~\ref{alg:samp} to be fast, the TT ranks $r$ should be as small as possible.
Since the joint PDF is typically a complicated multivariate function, its TT ranks may grow fast with the increasing accuracy.
On the other hand, low accuracy is typically sufficient if we 'correct' the distribution using the Metropolis--Hastings (MH) algorithm to ensure that the samples are distributed according to the target distribution $\pi$. Thus, we first propose to use a coarse TT approximation together with TT-CD sampling as independence proposals in a MH algorithm\footnote{A more simple scheme may be to use a multiple of $\ttpi$ to bound $\pi$ and then use a rejection algorithm. However, as noted in \cite{liu1996}, the MH is more statistically efficient.}.

When the current state is $x$ and the new proposal is $x'$, the next state is determined by the stochastic iteration that first computes the
Metropolis--Hastings ratio
$$
h(x,x') = \frac{\pi(x')}{\pi(x)} \frac{\ttpi(x)}{\ttpi(x')},
$$
and the proposal is \emph{accepted} with probability
\begin{equation}
\alpha(x,x') = \min(h(x,x'), 1), \label{eq:mh-ratio}
\end{equation}
putting the new state $x=x'$, otherwise $x'$ is \emph{rejected} and the chain remains at $x$.

As efficiency indicators of this MH algorithm for estimating the expected value $\mathbb{E}_\pi g$ of some functional $g(x)$, we consider the acceptance rate and the integrated autocorrelation time.
In this section, we study how they depend on the approximation error in the PDF. Throughout we must assume that $\pi$ is absolutely continuous with respect to $\pi^*$, that guarantees reversibility with respect to $\pi$~\cite{tierney1998}, and that we can evaluate the importance ratio $w(x)=\pi(x)/\pi^*(x)$. We require that $w^*\equiv \| w \|_\infty < \infty$, which is equivalent to uniform geometric convergence (and ergodicity) of the chain~\cite{RobertsRosenthal2011}. (The essential supremum may be taken with respect to $\pi$ or $\pi^*$.)

To simplify the presentation in this subsection, we assume again (without loss of generality) that the density is normalized.
\begin{lemma}
\label{lem:acc_rate}
Suppose that $\pi(x)$ is normalized, and that the mean absolute error in the TT-CD sampling density satisfies
$$
\int | \ttpi(x) - \pi(x)| \dx \le \eps/2.
$$
Then the rejection rate is bounded by $\eps$, i.e.,
$$
\mathbb{E}\left[ 1 - \alpha(x,x')\right]  \le \eps,
$$
where the expectation is taken over the chain.
\end{lemma}
\begin{proof}
Using ergodicity of the chain,
$$
\mathbb{E}\left[1- \alpha(x,x')\right] = \int\!\!\!\int\left[1- \alpha(x,x')\right] {{\pi}}(x) \ttpi (x') \dx \dx'.
$$
Since $1-\alpha\leq |1-h|$,
\begin{align*}
  \left[ 1 - \alpha(x,x')\right] {{\pi}}(x) \ttpi (x')&\leq |{{\pi}}(x) \ttpi (x') -{{\pi}}(x') \ttpi (x)| \\
  & \leq \pi(x)| \ttpi (x') -\pi(x')|\\
  & \qquad+ \pi(x')|\ttpi (x) -\pi(x)|
\end{align*}
where the second step uses the triangle inequality. Integrating both sides with respect to $x$ and $x'$, we obtain the claim of the lemma.\hfill$\Box$
\end{proof}

This lemma indicates that the rejection rate decreases proportionally to $\eps$, where $\eps$ is the total error due to approximating $\pi$ by a low-rank TT decomposition $\tilde{\pi}$, interpolating discrete values of $\tilde \pi$ on
a grid, and taking the absolute values in Alg. \ref{alg:samp}, Line \ref{alg:samp:pk}.

Lemma~\ref{lem:acc_rate} assumed a \emph{mean} absolute error. We need the stronger statement of \emph{local} relative error, that is $w^*<\infty$, to bound the integrated autocorrelation time (IACT) \cite{wolff-mcerr-2004}, defined as
\begin{equation}
   \tau = \left(1+2\sum_{t=1}^\infty \rho_{gg}(t) \right)
\end{equation}
where $\rho_{gg}(t)$ is the autocorrelation coefficient for the chain in statistic $g$ at lag $t$.
Defined like this, $\tau\geq 1$ can be considered as a reduction factor in the efficiency of a particular MCMC chain compared to an ideal independent chain, asymptotically as the length of the chain goes to infinity.
Note that $w^*<\infty$ implies that TT-MH is uniformly ergodic, but conversely the MCMC is not even geometrically ergodic if $w^*=\infty$~\cite[Thm. 2.1]{MengersenTweedie1996}.

\begin{lemma}
\label{lem:iact}
When $w^*<\infty$, for any $g\in L^2(\pi)$,
$$
\tau \leq \frac{1+a}{1-a},
$$
where $a=1-1/w^*$.
\end{lemma}
\begin{proof}
Without loss of generality we may consider $g\in L_0^2(\pi)$, i.e., $E_{\pi}[g]=0$ (see, e.g.,~\cite{MiraGeyer1999}).
Consider the transition kernel
\[ P_{a}(x,\dy)=(1-a)\pi(\dy)+a\delta_x(\dy). \]
(This is the chain that proposes from $\pi$ and accepts with probability $(1-a)$.)
$P_{a}$ has a simple spectrum, consisting of $1$, with right eigenvector $\mathbf 1$, and $a$ for the orthogonal compliment. Hence the asymptotic variance in a CLT for the chain in $g\in L_0^2(\pi)$ induced by $P_a$ may be evaluated using the spectral measure (see, e.g.~\cite{MiraGeyer1999,HaggstromRosenthal2007}), which reads $\mathcal{E}_g(S)=\delta_a(S)$ in this case, giving IACT equal to $(1+a)/(1-a)$.
The transition kernel for the TT-MH chain is~\cite[Thm.~1 \& Lem.~3]{smithtierney1996}
\[ P(x,\dy)=\min(1/w(x),1/w(y))\pi(\dy)+\lambda(w(x))\delta_x(\dy), \]
with $\lambda$ given by~\cite[Eq (5)]{smithtierney1996}.
Since $\min(1/w(x),1/w(y))\geq 1/w^*$, $P$ dominates $P_{a}$, in the sense of Peskun ordering~\cite{tierney1998,MiraGeyer1999}, i.e., the off-diagonal terms in $P$ are greater or equal than those in $P_{a}$, and hence the IACT using $P$ is less or equal than that using $P_a$~\cite[Thm 3.4]{MiraGeyer1999}.\hfill$\Box$
\end{proof}
For discrete state spaces, the result in Lemma~\ref{lem:iact} follows directly from~\cite[Eqn. (2.1)]{mira2001}; while one could argue that this is sufficient for practical computation since computers are finite dimensional.

The TT cross method tends to introduce a more or less uniform error of magnitude $\eps$ \emph{on average}.
For regions where $\pi(x)\gg \eps$, this leads to a bounded importance ratio $w(x) \le 1 + \mathcal{O}(\eps)$.
When $\pi(x) \ll \eps$, we will typically have $\pi^*(x) = \mathcal{O}(\eps)$ and $w(x)<1$.
However, if $\pi(x) \approx \eps$ and a negative error of order $\eps$ is committed, the two may cancel, resulting in a small $\pi^*(x)$, and consequently in a large $w(x)$.
Numerical experiments demonstrate that $w^*-1$ can indeed be much larger than the $L_1$-norm error used in Lemma~\ref{lem:acc_rate} (see Fig.~\ref{fig:sa-n-eps}).
However, these cancellations (and hence the equality in $\min(1/w(x),1/w(y))\geq 1/w^*$) seem to be rare.
Moreover, the practical IACT %for coordinates ($g=x$)
tends to be much smaller than the upper bound given by Lemma~\ref{lem:iact}.

% \subsection{Improved Quadrature Points}
% \label{sec:qmc}

\subsection{QMC samples and importance weights (TT-qIW)}
\label{sec:isqmc}

Due to the Central Limit Theorem, the rate of convergence of the statistical error of a Monte Carlo estimator for $\mathbb{E}_\pi g$, as the number of samples $N \to \infty$, is limited to $\mathcal{O}(N^{-1/2})$. The IACT of the chain induced by a MH sampler, such as the TT-MH sampler in the previous section, only affects the constant in this estimate.

Thus, it is tempting to use more structured quadrature points to obtain a better convergence rate. For example, the TT approximation of $\pi$ provides the possibility to reduce the inherent multi-variate integrals to a sequence of uni-variate integrals, as we did when forming the marginal distributions in Sec.~\ref{sec:ttcd}, and use, e.g., Gauss quadrature.
Another option is to note that the TT-CD map is also well defined for other seed points, such as those taken from a quasi-Monte Carlo (QMC) rule,  %~\cite{nieder-qmc-1978,graham-QMC-2011}.
that is, $\{(q_1^\ell,\ldots,q_d^\ell)\}_{\ell=1}^{N}$ in Alg.~\ref{alg:samp} are taken from a QMC lattice in $[0,1]^d$, rather than i.i.d. samples from $\mathcal{U}(0,1)^d$. Under certain assumptions on the smoothness of the quantity of interest, the QMC quadrature can give an error that converges with order $N^{-1}$ instead of $N^{-1/2}$ when $N \to \infty$ \cite{nieder-qmc-1978, Dick-Acta-2013}.
However, both those approaches provide estimates for $\mathbb{E}_{\ttpi} g$, which are biased due to the TT-approximation, and this bias can not be 'corrected' using a MH step, as for i.i.d. seeds. On the other hand, there are no suitable convergence results for MH algorithms based on QMC proposals.

A classical way to remove the bias in the estimate is via importance re-weighting.
Writing the expectation as an integral, then multiplying and dividing by the approximate density function, gives
\begin{align}
\mathbb{E}_{\pi} g & = \frac{1}{Z} \int g(x) \pi(x) \dx = \frac{1}{Z}
%\int g(x) \pi(x) \frac{\ttpi}{\ttpi} \dx \\ & =
\int g(x) w(x) \, \ttpi (x) \dx,
\end{align}
where $Z=\int \pi(x) \dx$ is the normalization constant
and $w(x) = \pi(x)/\ttpi(x)$ is the \emph{importance weight}.
That is, the expectation of $g$ with respect to $\pi$ equals the expectation of the weighted function $g(x) w(x)$ with respect to the approximate density $\ttpi$. The normalization constant can be rewritten as
$
Z = \int w(x) \ttpi (x) \dx\,.
$
% that accommodates the possibility that $\pi(x)$ is only available up to an unknown normalization constant, as is common in Bayesian hierarchical analyses.
% In contrast, the approximate density $\ttpi$ from Alg. \ref{alg:samp} is normalized by construction.

Thus, given a set of samples $\{x^{\ell}\}_{\ell=1}^{N} \sim \ttpi$ produced using the TT-CD algorithm, either from a set of i.i.d. samples on $[0,1]^d$ or from a QMC lattice, we compute
\begin{equation}
\mathbb{E}_{\pi} g \approx \frac{1}{\tilde Z} \left( \frac{1}{N} \sum_{\ell=1}^{N} g(x^{\ell}) w(x^{\ell}) \right), \ \
\tilde Z := \frac{1}{N}\sum_{\ell=1}^N w(x^\ell).
\label{eq:is}
\end{equation}
Note that, since $x^{\ell} \sim \ttpi$, the weight $w (x^{\ell}) < \infty$ with probability $1$,
and hence the importance quadrature \eqref{eq:is} is well-defined.
The convergence depends on the distance between $|\ttpi-\pi|$ and on the choice of samples $x^{\ell}$. Most importantly, if the seeds $\{q^\ell\}$ for the TT-CD samples $\{x^{\ell}\}$ in Alg.~\ref{alg:samp} are chosen according to a randomized QMC rule, and the integrand $g(x) w(x)$ is sufficiently smooth, we can expect a rate of convergence close to $\mathcal{O}(N^{-1})$, the estimator is unbiased and under the right smoothness assumptions the convergence rate is dimension independent \cite{Dick-Acta-2013}.
% \includegraphics[width=0.45\linewidth]{\pdfpath/FF-qmc.pdf} \raisebox{1.8cm}{\Large $\rightarrow$} \begin{tikzpicture}
% \node[] at (0,0) {\includegraphics[width=0.45\linewidth]{\pdfpath/FF-Pi2.pdf}};
% \node[opacity=0.8] at (0,0) {\includegraphics[width=0.45\linewidth]{\pdfpath/FF-irtq.pdf}};
% \end{tikzpicture}

\subsection{Multilevel acceleration}
\label{sec:multi}

Following recent works on multilevel MCMC~\cite{hoang-mlmcmc-2013,scheichl-mlmcmc-2015},
we can also use the (cheap) surrogate $\ttpi$ as a type of control variate to achieve variance reduction in the estimator.

In addition to $\pi^*$, we may also have a cheap 'surrogate' $\tilde g$ for the integrand $g$. For example, in Section~\ref{sec:ff} below,
we will build a TT-surrogate $\tilde u_h(\theta)$ of the FE solution $u_h(\theta)$ of the stochastic diffusion equation, as a function of the stochastic parameters $\theta$, that allows for a cheap approximation $\tilde g(\theta) = \phi(\tilde u_h)$ of any functional $g(\theta)=\phi(u_h)$ of the PDE solution, without having to solve the PDE for each sample. Otherwise, let $\tilde g = g$.

To exploit the multilevel ideas, we observe that
\begin{align}
\mathbb{E}_{\pi} g & = \mathbb{E}_{\pi^*} \tilde g  + \Big[ \mathbb{E}_{\pi} g  - \mathbb{E}_{\pi^*} \tilde g \Big] \label{eq:ml-split1}\\
& = \mathbb{E}_{\pi^*} \tilde g  + \; \mathbb{E}_{\pi^*} \left[\frac{1}{\mathbb{E}_{\pi^*} w} g w - \tilde g\right]. \label{eq:ml-split2}
\end{align}
As in the previous section, given a set of $N_0$ samples $\{x^{\ell}\}_{\ell=1}^{N_0}\sim \ttpi$ produced using the TT-CD algorithm, the first term in \eqref{eq:ml-split1} and \eqref{eq:ml-split2} can be estimated by
\begin{equation}
 \mathbb{E}_{\ttpi} \tilde g \approx \frac{1}{N_0}\sum_{\ell=1}^{N_0} \tilde g(x^\ell)
%+ \frac{1}{N_1} \sum_{\ell=1}^{N_1} \left[g(x^\ell)\tilde w(x^\ell) - \tilde g(x^\ell)\right].
 \label{eq:ml-coarse}
\end{equation}
Since the expected value in \eqref{eq:ml-coarse} is with respect to $\ttpi$, no MH correction is necessary. Moreover, we can use, as in Section~\ref{sec:isqmc}, QMC seed points $\{q^\ell\} \subset [0,1]^d$ for the TT-CD samples $\{x^{\ell}\}$ in Alg.~\ref{alg:samp}, leading to a much faster convergence of the estimator with respect to $N_0$.

In fact, if the evaluation of $\tilde g$ is significantly faster than the evaluation of~$g$, as in the stochastic diffusion problem below, the cost of estimating the first term in  \eqref{eq:ml-split1} and \eqref{eq:ml-split2} becomes entirely negligible.

To estimate the second term in \eqref{eq:ml-split1} and \eqref{eq:ml-split2} we now proceed as in Sections \ref{sec:ttmcmc} and \ref{sec:isqmc}, respectively.

 First consider a set of i.i.d. samples $\{x^{\ell}\}_{\ell=1}^{N_1} \sim \pi^*$, computed using Alg.~\ref{alg:samp}, and let $\{x_{\text{MH}}^{\ell}\}_{\ell=1}^{N_1}$ be the Markov chain of samples distributed according to $\pi$ after Metro\-polis-Hastings 'correction' of $\{x^{\ell}\}_{\ell=1}^{N_1}$ using the acceptance probability defined in \eqref{eq:mh-ratio}. We can define the following unbiased estimator:
\begin{equation}
\mathbb{E}_{\pi} g - \mathbb{E}_{\ttpi} \tilde g \approx \frac{1}{N_1}\sum_{\ell=1}^{N_1} g(x^\ell_{\text{MH}}) - \tilde g(x^{\ell})\,.
\label{eq:qcorr}
\end{equation}
If $\ttpi \approx \pi$ and $\tilde g \approx g$ the pairs of samples $(\tilde g(x^\ell), g(x_{\text{MH}}^{\ell}))$ are strongly, positively correlated and thus the variance of $g(x^\ell_{\text{MH}}) - \tilde g(x^{\ell})$ is much smaller than the variance of $g(x^\ell_{\text{MH}})$. As a consequence, the number of samples $N_1$ necessary to achieve a prescribed statistical error can be chosen significantly smaller than in Section \ref{sec:ttmcmc}.

Alternatively, consider now the second term in \eqref{eq:ml-split2} and let $\{x^{\ell}\}_{\ell=1}^{N_1}$ be obtained via Alg.~\ref{alg:samp} from a set of $N_1$ randomised QMC seed points $\{q^\ell\}_{\ell=1}^{N_1} \subset [0,1]^d$.
Then we can define the following unbiased estimator:
\begin{equation}
\mathbb{E}_{\pi^*} \left[\frac{1}{\mathbb{E}_{\pi^*} w} g w - \tilde g\right] \approx \frac{1}{N_1} \sum_{\ell=1}^{N_1} \frac{1}{\tilde Z} g(x^\ell) w(x^\ell) - \tilde g(x^\ell).
\label{eq:qcorr2}
\end{equation}
Again, if $\ttpi \approx \pi$ and $\tilde g \approx g$ then $w\approx 1$ and the variance of $g(x^\ell) w(x^\ell) / \tilde Z - \tilde g(x^\ell)$ is small, so that the number of samples $N_1$ can be chosen significantly smaller than the number $N$ of samples in \eqref{eq:is}.
Moreover, since $\tilde Z = 1 + \frac{1}{N_1} \sum_{\ell=1}^{N_1} (w(x^\ell) - 1)$ and $\mathbb{V}_{\ttpi}[w - 1]$ is small, a small value for $N_1$ is also sufficient for the calculation of $\tilde Z$ in \eqref{eq:is}.
If $g w/ \tilde Z - \tilde g$ is sufficiently smooth, the rate of convergence of the sampling error as $N_1 \to \infty$ should again be close to $\mathcal{O}(N_1^{-1})$. However, in contrast to the estimator in \eqref{eq:is}, we do not observe that better rate of convergence for the difference estimator in \eqref{eq:qcorr2}.

It would be possible to further optimize the complexity of the estimators in \eqref{eq:ml-coarse}, \eqref{eq:qcorr} and \eqref{eq:qcorr2} by a judicious choice of the TT accuracy $\eps$, as well as the numbers of samples $N_0$ and $N_1$, There is of course also scope for full multilevel estimators as in \cite{hoang-mlmcmc-2013,scheichl-mlmcmc-2015}. In particular, the values of $N_0$ and $N_1$ can be determined by an adaptive greedy procedure \cite{Scheichl-mlqmc-lognorm-2017}, which compares empirical variances and costs of the two levels and doubles $N_\ell$ on the level that has the maximum profit. However, we will not consider this further and leave it for future works.

%%%%%%%%%%%%%%%%%%%%%%%%%%%%%%%%%%%%%%%%%%%%%%%

\section{Numerical examples}
\label{sec:ne}

\subsection{Shock absorber reliability}
\label{sec:sae}

In this section, we demonstrate our algorithm on a problem of reliability estimation of a shock absorber. The time to failure of a type of shock absorber depends on some environmental conditions (covariates) such as humidity, temperature, etc. We use data \cite{oconnor-absorber-example-2012} on the distance (in kilometers) to failure for 38 vehicle shock absorbers. Since there were no values of any covariates in this example, the values of $D$ covariates were synthetically generated from the standard normal distribution as this would correspond to the case in which the covariates have been  standardized to have mean zero and variance equal to one.
The accelerated failure time regression model \cite{meeker-reliability-book-1998} is widely used for reliability estimation with covariates. We use an accelerated failure time Weibull regression model, which was described as reasonable for this data in \cite{meeker-reliability-book-1998}, where the density of time to failure is of the form
$$
f(t|\theta_1,\theta_2) = \frac{\theta_2}{\theta_1} \left(\frac{t}{\theta_1}\right)^{\theta_2-1} \exp\left(-\left(\frac{t}{\theta_1}\right)^{\theta_2}\right)
$$
and where $\theta_1,\theta_2$ are unknown scale and shape hyperparameters, respectively. The covariates are assumed to affect the failure time distribution only through the scale parameter $\theta_1$, via a standard logarithmic link function, that is
$$
\theta_1(\beta_0,\ldots,\beta_D) = \exp\left(\beta_0 + \sum_{k=1}^{D} \beta_k x_k\right),
$$
where $x_k$ are the covariates.
The $D+2$ unknown parameters $\beta_0,\ldots,\beta_{D}$ and $\theta_2$ must be inferred from the observation data on the covariates $x_k$ and the failure times, which in this example are subject to right censoring (marked with $^+$). The set $T_f$ of failure times is given by:\vspace{1ex}

\begin{tabular}{lllll}
 6700 &6950$^+$ &7820$^+$ &8790$^+$      &9120 \\
 9660$^+$ &9820$^+$ &11310$^+$ &11690$^+$ &11850$^+$ \\
 11880$^+$ &12140$^+$ &12200 &12870$^+$  &13150 \\
 13330$^+$ &13470$^+$ &14040$^+$  &14300 &17520 \\
 17540$^+$ &17890$^+$ &18420$^+$ &18960$^+$ &18980$^+$ \\
 19410$^+$ &20100 &20100$^+$   &20150$^+$ &20320$^+$ \\
 20900 &22700 &23490$^+$  &26510 &27410$^+$ \\
 27490   &27890$^+$  &28100$^+$ \\
\end{tabular}\vspace{1ex}

To perform Bayesian inference on the unknown parameters, we use the prior specifications in~\cite{Christen-shock-2005}, namely an $s$-Normal-Gamma distribution $\pi_0(\beta_0,\ldots,\beta_D, \theta_2)$ given by
$$
\pi_0 = \frac{1}{Z}\theta_2^{\alpha-0.5} \prod_{k=0}^{D}\exp\left(-\frac{\theta_2 (\beta_k-m_k)^2}{2 \sigma_k^2}\right) \exp\left(-\gamma\theta_2\right),
$$
where $\gamma = 2.2932$, $\alpha = 6.8757$, $m_0 = \log(30796)$, $\sigma_0^2 = 0.1563$, $m_1=\cdots=m_D=0$, $\sigma_1=\cdots=\sigma_D=1$,
and $Z$ is the normalization constant.
The parameter ranges
$$
[m_0-3\sigma_0, m_0+3\sigma_0] \times [m_1-3\sigma_1, m_1+3\sigma_1]^{D} \times [0,13]
$$
are large enough to treat the probability outside as negligible.

The (unnormalized) Bayesian posterior density function is given by  a product of Weibull probabilities, evaluated at each observation in $T_f$, and the prior distribution, i.e.
$$
\pi(\beta,\theta_2) = \pi_0(\beta,\theta_2) \prod_{t \in T_f} P(t|\theta_1(\beta),\theta_2),
$$
where
$$
P(t|\theta_1,\theta_2) = \left\{\begin{array}{ll}f(t|\theta_1,\theta_2) & \mbox{if $t$ is not censored}, \\ \exp\left(-\left(\frac{t}{\theta_1}\right)^{\theta_2}\right) & \mbox{if $t$ is censored.} \end{array}\right.
$$
The formula for the censored case arises from the fact that the contribution of a censored measurement is the probability that $t$ exceeds the measured value, that is, $P(t\ge t^{+}|\theta_1,\theta_2) = \int_{t^{+}}^{\infty} f(t|\theta_1,\theta_2)dt$.
We introduce $n$ uniform discretization points in $\beta_0,\ldots,\beta_D$ and $\theta_2$ and compute the TT cross approximation of the discretized density $\pi(\beta_0,\ldots,\beta_D,\theta_2)$.

We consider two quantities of interest, the right $95\%$ mean quantile and the right $95\%$ quantile of the mean distribution, i.e.
\begin{equation}
\label{eq:shock_qois}
\begin{split}
\langle q(f)\rangle & = \frac{1}{N}\sum_{i=1}^{N} \theta_1^i \log^{1/\theta_2^i}(1/0.05), \quad \mbox{and} \\
q(\langle f \rangle) & = t \quad \text{s.t.} \ \ \frac{1}{N}\sum_{i=1}^{N} \int_{0}^{t} f(s|\theta_1^i,\theta_2^i)ds = 0.95,
\end{split}
\end{equation}
respectively. The nonlinear constraint in the computation of the second quantile is solved by Newton's method.
To estimate the quadrature error, we perform $32$ runs of each experiment, and compute an average relative error over all runs, i.e.,
\begin{equation}
 \mathcal{E}_{q} =
 \frac{1}{32}  \sum_{\iota=1}^{32}
% \frac{1}{2} \dfrac{\left|\langle q(f)\rangle_{\iota} - \frac{1}{32} \sum_{\ell=1}^{32}\langle q(f)\rangle_{\ell} \right|}{\frac{1}{32} \sum_{\ell=1}^{32}\langle q(f)\rangle_{\ell}} +
%  \frac{1}{2}
 \dfrac{\left|q(\langle f \rangle_{\iota}) - \frac{1}{32} \sum_{\ell=1}^{32} q(\langle f \rangle_{\ell}) \right|}{\frac{1}{32} \sum_{\ell=1}^{32} q(\langle f \rangle_{\ell})},
\label{eq:err_q1q2}
\end{equation}
where $\iota$ and $\ell$ enumerate different runs.

The error in the mean quantile is estimated similarly and then the average of those two error estimates is used in all our convergence studies.
We used quantiles as the quantity of interest in order to illustrate that the TT surrogate captures the tails correctly.

\subsubsection{Accuracy of TT approximation and CD sampler}

We start by analysing the TT-MH sampling procedure, as described in Section \ref{sec:ttmcmc}.
First, we consider how the errors in $\tilde\pi$ due to the tensor approximation and discretization propagate into the quality of the MCMC chain produced by the MH algorithm,
i.e., the rate of rejections and the integrated autocorrelation time.
The chain length is always set to $N=2^{20}$, and the results are averaged over $32$ runs.
We choose a relatively low dimensionality $D=2$, since it allows us to approximate $\pi$ up to a high accuracy.

In Fig.~\ref{fig:sa-n-eps}, we vary the number of grid points $n$, fixing the stopping tolerance for the TT cross algorithm at $\delta=10^{-5}$, as well as benchmarking the algorithm for different thresholds $\delta$, fixing $n=512$.
We track the relative empirical standard deviation of the TT approximation,
\begin{equation}
 \mathcal{E}_{TT} = \sqrt{\frac{1}{31} \sum_{\iota=1}^{32}\left\|\tilde\pi_{\iota} - \frac{1}{32} \sum_{\ell=1}^{32} \tilde\pi_{\ell}\right\|_F^2 / \left\|\frac{1}{32} \sum_{\ell=1}^{32} \tilde\pi_{\ell}\right\|_F^2},
 \label{eq:eps}
\end{equation}
that can be computed exactly in the TT representation, as well as an importance-weighted QMC approximation to the $L_1$-norm error used in Lemma~\ref{lem:acc_rate},
\begin{equation}
\mathcal{E}_{L_1} = \frac{1}{N} \sum_{\ell=1}^{N} \left|w(x^\ell) - 1\right| \approx \int \left|\pi(x) - \pi^*(x)\right| dx.
\label{eq:epsl1}
\end{equation}
%and the essential supremum of the importance ratio $w^*$, taken with respect to $\pi^*$.

\begin{figure*}
\centering
\caption{Shock absorber example ($D=2$): rejection rate, IACT, estimated errors and importance weights (left), numbers of evaluations of $\pi$ and maximal TT ranks for TT cross (right) plotted against the grid size $n$ in each direction (top) and against the TT tolerance $\delta$ (bottom).}
\label{fig:sa-n-eps}
\resizebox{0.40\linewidth}{!}{%
\begin{tikzpicture}%
  \begin{axis}[%
  xmode=normal,
  ymode=log,
  xlabel=$\log_2 n$,
  ylabel={error indicators},
  legend style={at={(0.01,0.01)},anchor=south west},
  x filter/.code={\pgfmathparse{log2(\pgfmathresult)}\pgfmathresult},
  ]

% QTT   (y0=8, kickrank=0.3), delta=1e-5
% n             err_l1       err_tt       max_ratio    max_lag      IACT-1       rej.rate     TT rank      #evals
% 16            5.1435e-01   1.1132e-06   5.9426e+00   2.2000e+01   1.1485e+00   3.5040e-01   6.5750e+01   1.5395e+05
% 32            1.8008e-01   4.6350e-07   2.4899e+00   9.7188e+00   2.7020e-01   1.2884e-01   1.4700e+02   9.9354e+05
% 64            4.8141e-02   9.2798e-07   1.5699e+00   4.4688e+00   6.2227e-02   3.4561e-02   1.9150e+02   2.3821e+06
% 128           1.2254e-02   3.9859e-06   2.0373e+00   3.0000e+00   1.5348e-02   8.8023e-03   1.9800e+02   3.4818e+06
% 256           3.0817e-03   2.5329e-06   1.7890e+00   2.6250e+00   4.8348e-03   2.2062e-03   2.0650e+02   4.3931e+06
% 512           7.7878e-04   1.9999e-06   2.8672e+00   2.0625e+00   1.3742e-03   5.6219e-04   2.1100e+02   4.9131e+06

  % kickrank = 0.3, y0=8, nqtt = 4
% mean(num_of_rejects)/2^lvls, mean(tauint)-1, norm(err_FP)/sqrt(31), mean(cross_evalcnt(:,2)), mean(ttranks)
%    3.5052e-01   1.1486e+00   4.1821e-07   1.1744e+05   6.1500e+01
%    1.2879e-01   2.7323e-01   1.4148e-06   1.2006e+06   1.7991e+02
%    3.4498e-02   6.4633e-02   1.3625e-06   1.6557e+06   1.6950e+02
%    8.8106e-03   1.5534e-02   1.8864e-06   3.2496e+06   2.0225e+02
%    2.2233e-03   3.8576e-03   1.9965e-06   2.9039e+06   1.7112e+02
%    5.7220e-04   2.4615e-03   2.3126e-06   5.3375e+06   2.1088e+02

  \addplot+[] coordinates{(16  , 3.5040e-01)
                          (32  , 1.2884e-01)
                          (64  , 3.4561e-02)
                          (128 , 8.8023e-03)
                          (256 , 2.2062e-03)
                          (512 , 5.6219e-04)};  \addlegendentry{rej. rate};
%                           (1024, 1.3900e-04)
  \addplot+[] coordinates{(16  , 1.1485e+00)
                          (32  , 2.7020e-01)
                          (64  , 6.2227e-02)
                          (128 , 1.5348e-02)
                          (256 , 4.8348e-03)
                          (512 , 1.3742e-03)};  \addlegendentry{$\tau-1$};
%                           (1024, 5.6628e-04)

  \addplot+[mark=triangle*] coordinates{(16  , 5.1435e-01)
                          (32  , 1.8008e-01)
                          (64  , 4.8141e-02)
                          (128 , 1.2254e-02)
                          (256 , 3.0817e-03)
                          (512 , 7.7878e-04)};  \addlegendentry{$\mathcal{E}_{L_1}$};

  \addplot+[green!50!black,mark options={green!50!black}] coordinates{(16  , 5.9426e+00)
                          (32  , 2.4899e+00)
                          (64  , 1.5699e+00)
                          (128 , 2.0373e+00)
                          (256 , 1.7890e+00)
                          (512 , 2.8672e+00)};  \addlegendentry{$w^*$};

%   \addplot+[dashed,black,mark=diamond*,mark options={black}] coordinates{
%                           (16  , 1.1132e-06)
%                           (32  , 4.6350e-07)
%                           (64  , 9.2798e-07)
%                           (128 , 3.9859e-06)
%                           (256 , 2.5329e-06)
%                           (512 , 1.9999e-06)
%                           };  \addlegendentry{$\mathcal{E}_{TT}$};

  \addplot+[no marks, domain=16:512, black] {10.^(-2.0*log10(x)+2.0)}; \addlegendentry{$C\cdot n^{-2}$};
  \end{axis}
 \end{tikzpicture}%
}
\resizebox{0.4275\linewidth}{!}{%
\begin{tikzpicture}%
  \begin{axis}[%
  xmode=normal,
  ymode=normal,
  xlabel=$\log_2 n$,
  ylabel={\#evaluations (millions)},
  legend style={at={(0.99,0.01)},anchor=south east},
  y label style={at={(-0.1,1.0)},rotate=0,color={blue}},every y tick label/.style={blue},
  x filter/.code={\pgfmathparse{log2(\pgfmathresult)}\pgfmathresult},
  ]
  \addplot+[] coordinates{(16  , 1.5395e+05*1e-6)
                          (32  , 9.9354e+05*1e-6)
                          (64  , 2.3821e+06*1e-6)
                          (128 , 3.4818e+06*1e-6)
                          (256 , 4.3931e+06*1e-6)
                          (512 , 4.9131e+06*1e-6)}; \addlegendentry{\#evals}; % new likelihood
%                           (1024, 4.1375e+06*1e-6)

  \end{axis}
  \begin{axis}[%
  xmode=normal,
  ymode=normal,
  ylabel={TT rank},
  legend style={at={(0.01,0.99)},anchor=north west},
  axis y line*=right, y label style={at={(1.1,1.0)},anchor=south east,rotate=0,color={red}},every y tick label/.style={red},
  axis x line=none,
  x filter/.code={\pgfmathparse{log2(\pgfmathresult)}\pgfmathresult},
  ]
  \addplot+[red, mark options={red},mark=+]
                coordinates{
                 (16  , 6.5750e+01)
                 (32  , 1.4700e+02)
                 (64  , 1.9150e+02)
                 (128 , 1.9800e+02)
                 (256 , 2.0650e+02)
                 (512 , 2.1100e+02)
                }; \addlegendentry{TT rank};

  \end{axis}
 \end{tikzpicture}%
}\\
\resizebox{0.40\linewidth}{!}{%
\begin{tikzpicture}%
  \begin{axis}[%
  xmode=normal,
  ymode=log,
  xlabel=$\log_{10}\delta$,
  ylabel={error indicators},
  ymin=5e-8,ymax=1.5e2,
  ytick={1e-6,1e-5,1e-4,1e-3,1e-2,1e-1,1e0,1e1,1e2},
  legend style={at={(0.99,0.01)},anchor=south east},
  x filter/.code={\pgfmathparse{log10(\pgfmathresult)}\pgfmathresult},
  ]

  \addplot+[] coordinates{(5e-1, 5.5345e-02)
                          (1e-1, 1.3639e-02)
                          (3e-2, 7.0873e-03)
                          (1e-2, 4.9726e-03)
                          (3e-3, 2.2395e-03)
                          (1e-3, 1.4705e-03)
                          (3e-4, 9.0384e-04)
                          (1e-4, 6.4492e-04)
                          (3e-5, 5.8618e-04)
                          (1e-5, 5.6219e-04)
                          };  \addlegendentry{rej. rate}; % new L

  \addplot+[] coordinates{(5e-1, 7.3728e-01)
                          (1e-1, 1.8617e-01)
                          (3e-2, 1.0140e-01)
                          (1e-2, 7.6453e-02)
                          (3e-3, 4.0097e-02)
                          (1e-3, 2.6633e-02)
                          (3e-4, 1.2218e-02)
                          (1e-4, 6.2891e-03)
                          (3e-5, 2.2470e-03)
                          (1e-5, 1.3742e-03)
                          };  \addlegendentry{$\tau-1$}; % new L

  \addplot+[mark=triangle*] coordinates{
                          (5e-1, 6.7731e-02)
                          (1e-1, 1.5379e-02)
                          (3e-2, 7.8266e-03)
                          (1e-2, 5.4749e-03)
                          (3e-3, 2.5294e-03)
                          (1e-3, 1.6797e-03)
                          (3e-4, 1.1085e-03)
                          (1e-4, 8.5394e-04)
                          (3e-5, 7.9869e-04)
                          (1e-5, 7.7878e-04)
                          };  \addlegendentry{$\mathcal{E}_{L_1}$};

  \addplot+[green!50!black,mark options={green!50!black},mark=star] coordinates{
                          (5e-1, 1.3586e+02)
                          (1e-1, 8.0511e+01)
                          (3e-2, 7.2213e+01)
                          (1e-2, 4.9966e+01)
                          (3e-3, 3.6945e+01)
                          (1e-3, 1.5575e+01)
                          (3e-4, 1.0436e+01)
                          (1e-4, 5.6322e+00)
                          (3e-5, 2.8822e+00)
                          (1e-5, 2.8672e+00)
                          };  \addlegendentry{$w^*$}; % new L

  \addplot+[dashed,black,mark=diamond*,mark options={black}] coordinates{
                          (5e-1, 3.2315e-02)
                          (1e-1, 9.1142e-03)
                          (3e-2, 3.3682e-03)
                          (1e-2, 2.6205e-03)
                          (3e-3, 8.0710e-04)
                          (1e-3, 2.8202e-04)
                          (3e-4, 8.2302e-05)
                          (1e-4, 4.3589e-05)
                          (3e-5, 1.0073e-05)
                          (1e-5, 1.9999e-06)
                          };  \addlegendentry{$\mathcal{E}_{TT}$};

  \addplot+[solid,no marks, domain=1e-5:1e0, black] {x}; \addlegendentry{$\delta$};
  \end{axis}
 \end{tikzpicture}%
}
\resizebox{0.4275\linewidth}{!}{%
\begin{tikzpicture}%
  \begin{axis}[%
  xmode=normal,
  ymode=normal,
  xlabel=$\log_{10} \delta$,
  ylabel={\#evaluations (millions)},
  legend style={at={(0.01,0.01)},anchor=south west},
  y label style={at={(-0.1,1.0)},rotate=0,color={blue}},every y tick label/.style={blue},
  x filter/.code={\pgfmathparse{log10(\pgfmathresult)}\pgfmathresult},
  ]
  \addplot+[] coordinates{(5e-1, 1.2641e+05*1e-6)
                          (1e-1, 2.1294e+05*1e-6)
                          (3e-2, 2.6405e+05*1e-6)
                          (1e-2, 3.5822e+05*1e-6)
                          (3e-3, 5.4437e+05*1e-6)
                          (1e-3, 7.9554e+05*1e-6)
                          (3e-4, 1.4392e+06*1e-6)
                          (1e-4, 2.2380e+06*1e-6)
                          (3e-5, 3.3326e+06*1e-6)
                          (1e-5, 4.9131e+06*1e-6)
                          }; \addlegendentry{\#evals};
  \end{axis}
  \begin{axis}[%
  xmode=normal,
  ymode=normal,
  ylabel={TT rank},
  legend style={at={(0.99,0.99)},anchor=north east},
  axis y line*=right, y label style={at={(1.1,1.0)},anchor=south east,rotate=0,color={red}},every y tick label/.style={red},
  axis x line=none,
  x filter/.code={\pgfmathparse{log10(\pgfmathresult)}\pgfmathresult},
  ]
  \addplot+[red, mark options={red},mark=+]
                coordinates{(5e-1, 3.2750e+01)
                            (1e-1, 4.1250e+01)
                            (3e-2, 4.3625e+01)
                            (1e-2, 4.5750e+01)
                            (3e-3, 6.6250e+01)
                            (1e-3, 7.5500e+01)
                            (3e-4, 1.0962e+02)
                            (1e-4, 1.3425e+02)
                            (3e-5, 1.5700e+02)
                            (1e-5, 2.1100e+02)
                }; \addlegendentry{TT rank};
  \end{axis}
 \end{tikzpicture}%
}
\end{figure*}
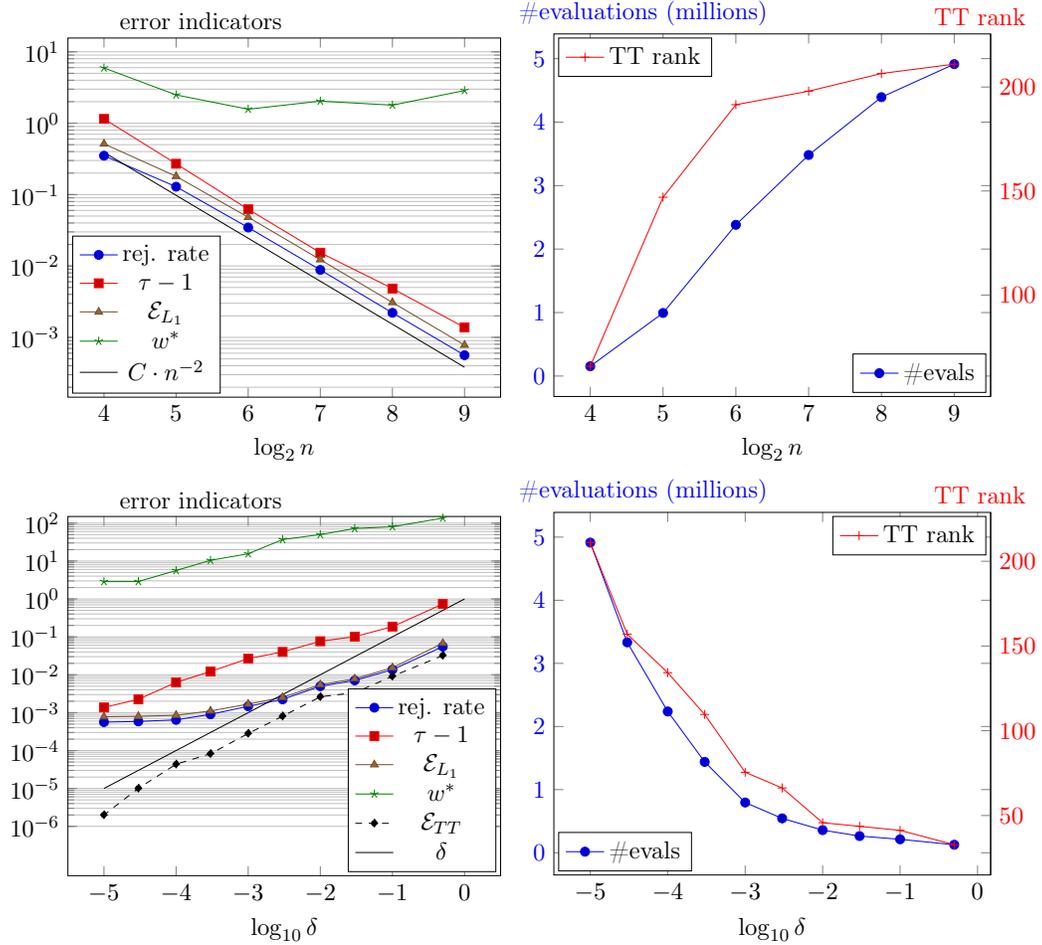

As shown in Lemma \ref{lem:acc_rate}, the rejection rate is expected to be proportional to the approximation error in $L_1$ norm, as this error goes to zero.
The TT approximation is computed on a tensor grid with $n$ vertices and uses linear interpolation to evaluate $\tilde\pi$ at intermediate values.
Thus, it can be expected that the convergence of the interpolation error, as $n \to \infty$, should be of $\mathcal{O}(n^{-2})$, provided $\pi$ is sufficiently smooth.
We can see in Fig.~\ref{fig:sa-n-eps} (top-left) that the rejection rate converges with $\mathcal{O}(n^{-2})$, suggesting that this is the case here.
Bottom-left of Fig.~\ref{fig:sa-n-eps} also suggests that the rejection rate is proportional to the TT approximation error when it is greater than the interpolation error.

The behaviour of the importance ratio and the integrated autocorrelation time (IACT) is more complicated.
The IACT $\tau$ and the essential supremum $w^*$ of the importance ratio are tracked in Fig.~\ref{fig:sa-n-eps} as well.
The TT Cross algorithm tries to reduce the average approximation error.
Pointwise relative error, however, is not guaranteed to be bounded.
Although $w^* \to 1$ as $\delta \rightarrow 0$, it is orders of magnitude larger than $\mathcal{E}_{L_1}$. Regardless, Lemma \ref{lem:iact} seems to give a too pessimistic estimate for the IACT, as the actual value $\tau-1$ is much smaller than $w^*$ and behaves similarly to the rejection rate.

The complexity of the TT cross algorithm (in terms of both the number of evaluations of $\pi$ and the computational time) grows only very mildly (sublinearly) with $\delta$ and $n$ (notice the log-polynomial scale in Fig.~\ref{fig:sa-n-eps}, right).
This makes the TT approach also well scalable for high accuracies.

\subsubsection{Convergence studies and comparison to DRAM}

Now we investigate the convergence of the quantiles and compare TT-MH with the delayed rejection adaptive Metropolis (DRAM) algorithm \cite{Haario-DRAM-2006}. The initial covariance for DRAM is chosen to be the identity matrix. In order to eliminate the effect of the burn-in period, we do not include the first $N/4$ elements of the DRAM chain in the computation of the quantiles. However, we will study the actual burn-in time empirically to have a fairer comparison of the ``set-up cost'' of the two methods.

First, in Table \ref{tab:shock-tau}, we fix $D=6$ covariates and vary the discretization grid $n$ and the TT approximation threshold $\delta$. We present the rejection rates and the IACTs for TT-MH,  with $n=12$, $16$, and $32$ grid points in each direction, using values of $\delta=0.5$ and $\delta=0.05$, as well as for DRAM. In addition, we also give the setup cost in terms of numbers of evaluations of $\pi$, i.e. the number of points needed to construct the TT approximation via the TT cross algorithm for TT-MH and the burn-in
in DRAM.
The latter is estimated as the point of stabilization of $6$ moments of $\beta$ and $\theta_2$, approximated by averaging over $2^{14}$ random initial guesses.
The coarsest TT approximation requires about  $4 \cdot 10^4$ evaluations, whereas DRAM needs a burn-in of about $5 \cdot 10^4$ steps.
\begin{table}[t]
\centering
\caption{Comparison of TT-MH and DRAM; Rejection rate, IACT, and number of function evaluations to set up TT cross and to burn in DRAM for the shock absorber ($D=6$).}
\label{tab:shock-tau}
\begin{tabular}{c|cccc|c}
                & \multicolumn{4}{c|}{TT-MH}                                     & DRAM      \\
$n$             & $12$ & $16$ & $16$ & $32$ & \\
$\delta$  & $0.5$ & $0.5$ & $0.05$ & $0.05$ \\\hline
rej. rate        & 0.61                  & 0.33                  &  0.28                  & 0.12                    &  0.5   \\
$\tau$                & 13.76                  & 4.24                  &  2.94                  & 2.15                    &  24.85  \\ \hline
$N_{setup}$ & 35158                   & 44389                   &  101564                  & 221116                    & 49200 \\
\end{tabular}
\end{table}

\begin{figure*}
\centering
\caption{Shock absorber example ($D=6$): sampling error versus chain length (left) and versus total CPU time (right) for different choices of $n$ and $\delta$ in the TT cross method.}
\label{fig:sa-err}
\resizebox{0.49\linewidth}{!}{%
\begin{tikzpicture}%
  \begin{axis}[%
  xmode=normal,
  ymode=log,
  ymax=0.5,
  xlabel=$\log_{10}N$,
  ylabel={$\mathcal{E}_q$},
  legend style={at={(0.99,0.99)},anchor=north east},
  x filter/.code={\pgfmathparse{log10(\pgfmathresult)}\pgfmathresult},
  ]

  \addplot+[blue,dashed,mark options={scale=0.5},mark=*] coordinates{
        (2^10,   1.3109e-02)
        (2^11,   8.2736e-03)
        (2^12,   6.1296e-03)
        (2^13,   3.5863e-03)
        (2^14,   3.7658e-03)
        (2^15,   2.7119e-03)
        (2^16,   1.5411e-03)
        (2^17,   1.5336e-03)
        (2^18,   9.8920e-04)
        (2^19,   1.2111e-03)
        (2^20,   7.7604e-04)
        (2^21,   5.7430e-04)
        (2^22,   3.9705e-04)
        (2^23,   3.1615e-04)
    };  \addlegendentry{\footnotesize $n,\delta=16,0.5$,MH}; % right quantile

  \addplot+[red,dashed,mark options={scale=0.5},mark=square*] coordinates{
        (2^10,   1.7423e-02)
        (2^11,   5.9006e-03)
        (2^12,   6.3667e-03)
        (2^13,   4.2981e-03)
        (2^14,   3.7757e-03)
        (2^15,   2.4400e-03)
        (2^16,   1.2709e-03)
        (2^17,   1.1289e-03)
        (2^18,   1.0004e-03)
        (2^19,   6.1937e-04)
        (2^20,   5.2299e-04)
        (2^21,   3.9008e-04)
        (2^22,   2.3691e-04)
        (2^23,   1.8923e-04)
    };  \addlegendentry{\footnotesize $n,\delta=16,0.5$,qIW}; % right quantile

%   \addplot+[mark options={scale=0.5},red,solid] coordinates{
%       (2^10,      9.4648e-03)
%       (2^11,      7.4110e-03)
%       (2^12,      5.1261e-03)
%       (2^13,      3.8522e-03)
%       (2^14,      2.1867e-03)
%       (2^15,      2.3995e-03)
%       (2^16,      1.2226e-03)
%       (2^17,      9.4754e-04)
%       (2^18,      6.8059e-04)
%       (2^19,      6.1603e-04)
%       (2^20,      3.7774e-04)
%       (2^21,      2.3543e-04)
%       (2^22,      2.1111e-04)
%       (2^23,      1.4791e-04)
%     };  \addlegendentry{\footnotesize $n,\delta=16,0.05$,MH}; % right quantile
%
%   \addplot+[mark options={scale=0.5},red,dashed] coordinates{
%       (2^10,   6.6862e-03)
%       (2^11,   4.7966e-03)
%       (2^12,   1.6680e-03)
%       (2^13,   1.1969e-03)
%       (2^14,   1.0917e-03)
%       (2^15,   8.1419e-04)
%       (2^16,   5.2519e-04)
%       (2^17,   4.6778e-04)
%       (2^18,   3.5355e-04)
%       (2^19,   2.2580e-04)
%       (2^20,   1.5776e-04)
%       (2^21,   1.2507e-04)
%       (2^22,   9.3535e-05)
%       (2^23,   5.2039e-05)
%     };  \addlegendentry{\footnotesize $n,\delta=16,0.05$,IS}; % right quantile

  \addplot+[mark options={scale=1},blue,solid,mark=diamond*] coordinates{
      (2^10,   5.5880e-03)
      (2^11,   4.2380e-03)
      (2^12,   4.0199e-03)
      (2^13,   2.4843e-03)
      (2^14,   2.5749e-03)
      (2^15,   1.5804e-03)
      (2^16,   1.4755e-03)
      (2^17,   8.5291e-04)
      (2^18,   6.4929e-04)
      (2^19,   5.1862e-04)
      (2^20,   3.6303e-04)
      (2^21,   2.2746e-04)
      (2^22,   2.0042e-04)
      (2^23,   1.5914e-04)
    };  \addlegendentry{\footnotesize $n,\delta=32,0.05$,MH}; % right quantile

  \addplot+[mark options={scale=1},red,solid,mark=triangle*] coordinates{
      (2^10,   2.1104e-03)
      (2^11,   2.0846e-03)
      (2^12,   1.0930e-03)
      (2^13,   1.7273e-03)
      (2^14,   8.6672e-04)
      (2^15,   9.8765e-04)
      (2^16,   5.9068e-04)
      (2^17,   5.1005e-04)
      (2^18,   3.4798e-04)
      (2^19,   3.7588e-04)
      (2^20,   2.4285e-04)
      (2^21,   1.6185e-04)
      (2^22,   1.2452e-04)
      (2^23,   1.2224e-04)
    };  \addlegendentry{\footnotesize $n,\delta=32,0.05$,qIW}; % right quantile

  \addplot+[orange,line width=1pt,mark options={scale=1,color=orange},mark=star] coordinates{
     (2^10,   6.6939e-02)
     (2^11,   4.9831e-02)
     (2^12,   2.5010e-02)
     (2^13,   1.0510e-02)
     (2^14,   1.0747e-02)
     (2^15,   5.4859e-03)
     (2^16,   3.9332e-03)
     (2^17,   3.6831e-03)
     (2^18,   2.1467e-03)
     (2^19,   1.5628e-03)
     (2^20,   8.7534e-04)
     (2^21,   6.4369e-04)
     (2^22,   4.9617e-04)
     (2^23,   4.5703e-04)
    };  \addlegendentry{\footnotesize DRAM};

  \addplot+[no marks, domain=1e3:2e7, black,dashed] {x^(-0.5)}; \addlegendentry{\footnotesize $N^{-0.5}$};
  \end{axis}
 \end{tikzpicture}%
}
\resizebox{0.49\linewidth}{!}{%
\begin{tikzpicture}%
  \begin{axis}[%
  xmode=normal,
  ymode=log,
  ymax=0.5,
  xlabel=$\log_{10}\mbox{CPU time}$,
  ylabel={$\mathcal{E}_q$},
  legend style={at={(0.99,0.99)},anchor=north east},
  x filter/.code={\pgfmathparse{log10(\pgfmathresult)}\pgfmathresult},
  ]

  \addplot+[blue,dashed,mark options={scale=0.5},mark=*] coordinates{
        (   9.9824e-01 + 2.0143e-02,   1.3109e-02)
        (   9.9959e-01 + 3.4696e-02,   8.2736e-03)
        (   9.4987e-01 + 6.5639e-02,   6.1296e-03)
        (   9.9922e-01 + 1.2986e-01,   3.5863e-03)
        (   9.7481e-01 + 2.6292e-01,   3.7658e-03)
        (   9.8749e-01 + 5.2631e-01,   2.7119e-03)
        (   9.8112e-01 + 1.0404e+00,   1.5411e-03)
        (   9.7601e-01 + 2.0364e+00,   1.5336e-03)
        (   1.0356e+00 + 4.1762e+00,   9.8920e-04)
        (   1.0032e+00 + 8.3475e+00,   1.2111e-03)
        (   9.9577e-01 + 1.6314e+01,   7.7604e-04)
        (   1.0025e+00 + 3.2534e+01,   5.7430e-04)
        (   9.6541e-01 + 6.4484e+01,   3.9705e-04)
        (   9.8414e-01 + 1.2986e+02,   3.1615e-04)
    };  \addlegendentry{\footnotesize $n,\delta=16,0.5$,MH}; % right quantile

  \addplot+[red,dashed,mark options={scale=0.5},mark=square*] coordinates{
        (   9.9824e-01 + 2.0143e-02,   1.7423e-02)
        (   9.9959e-01 + 3.4696e-02,   5.9006e-03)
        (   9.4987e-01 + 6.5639e-02,   6.3667e-03)
        (   9.9922e-01 + 1.2986e-01,   4.2981e-03)
        (   9.7481e-01 + 2.6292e-01,   3.7757e-03)
        (   9.8749e-01 + 5.2631e-01,   2.4400e-03)
        (   9.8112e-01 + 1.0404e+00,   1.2709e-03)
        (   9.7601e-01 + 2.0364e+00,   1.1289e-03)
        (   1.0356e+00 + 4.1762e+00,   1.0004e-03)
        (   1.0032e+00 + 8.3475e+00,   6.1937e-04)
        (   9.9577e-01 + 1.6314e+01,   5.2299e-04)
        (   1.0025e+00 + 3.2534e+01,   3.9008e-04)
        (   9.6541e-01 + 6.4484e+01,   2.3691e-04)
        (   9.8414e-01 + 1.2986e+02,   1.8923e-04)
    };  \addlegendentry{\footnotesize $n,\delta=16,0.5$,qIW}; % right quantile

  \addplot+[mark options={scale=1},blue,solid,mark=diamond*] coordinates{
      (   4.4560e+00  + 3.1086e-02,   5.5880e-03)
      (   4.1623e+00  + 5.7736e-02,   4.2380e-03)
      (   4.2826e+00  + 1.1599e-01,   4.0199e-03)
      (   4.1474e+00  + 2.2278e-01,   2.4843e-03)
      (   4.2661e+00  + 4.5029e-01,   2.5749e-03)
      (   4.2923e+00  + 9.0466e-01,   1.5804e-03)
      (   4.2332e+00  + 1.7948e+00,   1.4755e-03)
      (   3.8803e+00  + 3.4390e+00,   8.5291e-04)
      (   4.0667e+00  + 7.1232e+00,   6.4929e-04)
      (   4.2129e+00  + 1.4345e+01,   5.1862e-04)
      (   4.9182e+00  + 2.8557e+01,   3.6303e-04)
      (   3.8228e+00  + 5.4807e+01,   2.2746e-04)
      (   4.3581e+00  + 1.1613e+02,   2.0042e-04)
      (   4.1102e+00  + 2.2484e+02,   1.5914e-04)
    };  \addlegendentry{\footnotesize $n,\delta=32,0.05$,MH}; % right quantile

  \addplot+[mark options={scale=1},red,solid,mark=triangle*] coordinates{
      (   4.4560e+00 +  3.1086e-02,   2.1104e-03)
      (   4.1623e+00 +  5.7736e-02,   2.0846e-03)
      (   4.2826e+00 +  1.1599e-01,   1.0930e-03)
      (   4.1474e+00 +  2.2278e-01,   1.7273e-03)
      (   4.2661e+00 +  4.5029e-01,   8.6672e-04)
      (   4.2923e+00 +  9.0466e-01,   9.8765e-04)
      (   4.2332e+00 +  1.7948e+00,   5.9068e-04)
      (   3.8803e+00 +  3.4390e+00,   5.1005e-04)
      (   4.0667e+00 +  7.1232e+00,   3.4798e-04)
      (   4.2129e+00 +  1.4345e+01,   3.7588e-04)
      (   4.9182e+00 +  2.8557e+01,   2.4285e-04)
      (   3.8228e+00 +  5.4807e+01,   1.6185e-04)
      (   4.3581e+00 +  1.1613e+02,   1.2452e-04)
      (   4.1102e+00 +  2.2484e+02,   1.2224e-04)
    };  \addlegendentry{\footnotesize $n,\delta=32,0.05$,qIW}; % right quantile

  \addplot+[orange,line width=1pt,mark options={scale=1,color=orange},mark=star] coordinates{
  (6.2536e-01,   6.6939e-02)
  (1.0645e+00,   4.9831e-02)
  (1.8647e+00,   2.5010e-02)
  (3.5013e+00,   1.0510e-02)
  (6.6346e+00,   1.0747e-02)
  (1.2929e+01,   5.4859e-03)
  (2.5431e+01,   3.9332e-03)
  (5.0523e+01,   3.6831e-03)
  (1.0141e+02,   2.1467e-03)
  (2.0066e+02,   1.5628e-03)
  (4.0092e+02,   8.7534e-04)
  (8.0025e+02,   6.4369e-04)
  (1.6053e+03,   4.9617e-04)
  (3.1982e+03,   4.5703e-04)
  };  \addlegendentry{\footnotesize DRAM};

  \addplot+[no marks, domain=1e0:1e4, black,dashed] {1e-2*x^(-0.5)}; \addlegendentry{\footnotesize $C\cdot \mathcal{W}^{-0.5}$};
  \end{axis}
 \end{tikzpicture}%
}
\end{figure*}
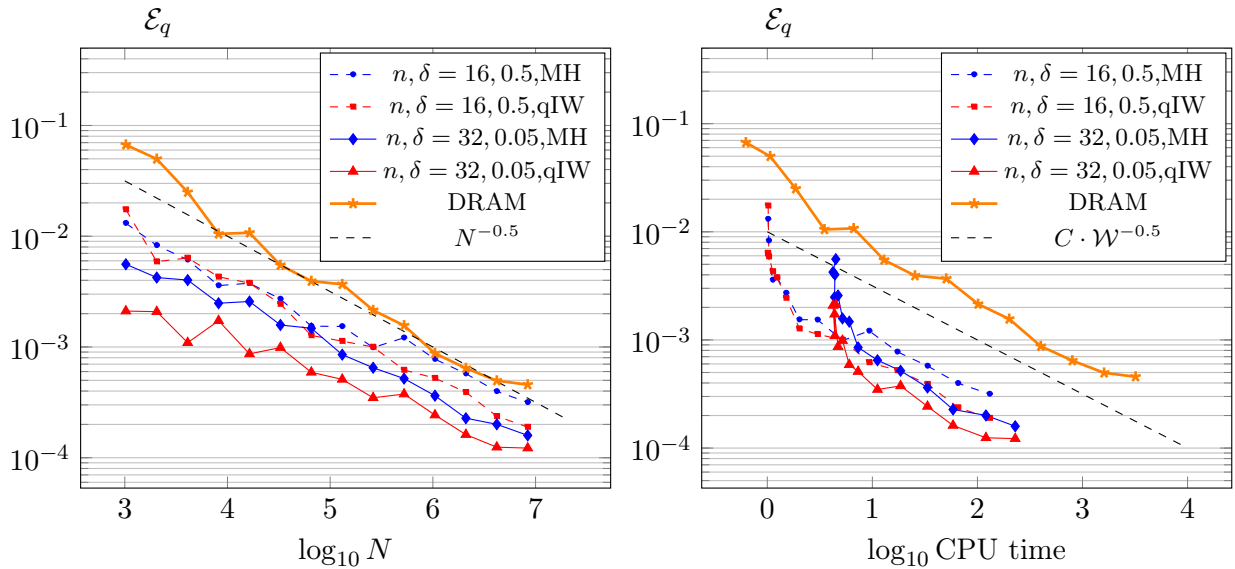

Next, in Fig. \ref{fig:sa-err} (left) we show the estimate $\mathcal{E}_q$ of the quadrature error defined in \eqref{eq:err_q1q2} for the two quantities of interest in  \eqref{eq:shock_qois}, versus the total number $N$ of samples in the MCMC chain, which is varied from $2^{10}$ to $2^{23}$.
We see that both MH methods (i.e. TT-MH and DRAM) converge with a rate of $N^{-1/2}$, as expected. To keep the set-up cost of the TT approximation low, we only consider fairly crude TT approximations (as in Tab.~\ref{tab:shock-tau}). However, all our approximations deliver a smaller sampling error for TT-MH than for DRAM when measured against the number of samples, and an even greater reduction when plotted against CPU time (Fig.~\ref{fig:sa-err}, right).
More accurate TT approximations require more evaluations of $\pi$ during the set-up in TT Cross, up to $2.5 \cdot 10^5$ for $\delta=0.05$ and $n=32$. This set-up cost is clearly visible in the vertical off-set of the curves in Fig.~\ref{fig:sa-err} (right). It exceeds the burn-in cost in DRAM.
However, TT-MH is much faster than DRAM for the same number of evaluations,
which yields a significant difference in terms of the total CPU time.

There are several reasons for this. For higher TT accuracies, the gains are mainly due to the significantly lower IACT of TT-MH, leading to a much better statistical efficiency of the MCMC chain. For low TT accuracies, the IACT of the TT-MH algorithm is still half of that for DRAM and in addition, there is some gain due to the reduced set-up cost. A further reason is the vectorization that is exploited in TT cross, where a block of $\mathcal{O}(nr^2)$ samples is evaluated in each step. In DRAM, the function needs to be evaluated point by point in order to perform the rejection.
Therefore, the number of \emph{distinct} calls to $\pi$ in TT cross is much smaller than $N$, reducing the corresponding overhead in Matlab.
In compiled languages (C, Fortran) on a single CPU, the difference may be less significant. However, parallel implementations will also benefit from the blocking, especially when each sample is expensive. More accurate TT approximations are worthwhile to compute if a highly accurate estimate of the expected value is required, since in that case the length of the MCMC chain will dominate the number of samples in the set-up phase.

In Fig. \ref{fig:sa-err}, we also present results with the TT-qIW approach described in Sec.~\ref{sec:isqmc},
where the approximate density $\ttpi$ is used as an importance weight and where the expected value and the normalizing constant are estimated via QMC quadrature.
In particular, we use a randomized rank-1 lattice rule with product weight parameters $\gamma_k = 1/k^2$.
The generating vector was taken from the file {\tt \verb+lattice-39102-1024-1048576.3600+}, available at {\tt \verb+http://web.maths.unsw.edu.au/~fkuo/+}.
Due to the non-smooth dependence of quantiles on the covariates,
the rate of convergence for TT-qIW with respect to $N$ is not improved in this example,
but in absolute terms it consistently outperforms TT-MH, leading to even bigger gains over DRAM.

Finally, we fix the TT and the MCMC parameters to $n=16$, $\delta=0.05$ and $N=2^{22}$ and vary the number of covariates $D$,
and hence the total dimensionality $d=D+2$.
In Fig. \ref{fig:sa-d}, we show the error in the quantiles, the number of evaluations of $\pi$, as well as the autocorrelation times and TT ranks.
We see that the TT ranks are almost independent of $d$,
and the TT-MH approach remains more efficient than DRAM over a wide range of dimensions.

\begin{figure*}
\centering
\caption{Shock absorber example: Error (left), number of $\pi$ evaluations during the proposal stage (middle) and IACT (right), for different numbers of covariates and $n=16$, $\delta=0.05$, $N=2^{22}$.
Numbers above points in the middle plot denote TT ranks.
}
\label{fig:sa-d}
\resizebox{0.32\linewidth}{!}{%
\begin{tikzpicture}
 \begin{axis}[%
  xmode=log,
  ymode=log,
  ymin=5e-5,ymax=3e-3,
  xlabel=$D$,
  ylabel={$\mathcal{E}_q$},
  legend style={at={(0.99,0.01)},anchor=south east},
  ]

% mean(mean(abs(Q_sp-mean(Q_sp)))./mean(Q_sp)), mean(cross_evalcnt(:,2)), mean(ttimes_cross), mean(ttimes_invcdf),  mean(tauint)
%    9.9835e-05   7.6800e+02   2.6636e-01   7.2172e+00   1.4979e+00
%    1.9708e-04   8.1160e+03   3.4616e-01   1.5574e+01   5.3302e+00
%    1.3119e-04   1.9479e+04   4.6113e-01   2.4217e+01   2.3840e+00
%    1.6805e-04   6.1585e+04   7.9878e-01   3.7488e+01   2.2517e+00
%    1.6885e-04   8.3087e+04   9.6402e-01   4.8089e+01   2.3780e+00
%    2.1433e-04   9.3551e+04   1.0835e+00   5.6677e+01   3.0751e+00
%    2.6397e-04   1.0766e+05   1.1863e+00   6.6136e+01   2.9446e+00
%    4.2550e-04   1.4510e+05   1.5280e+00   7.7490e+01   3.7976e+00
%    3.9207e-04   1.4466e+05   1.5818e+00   8.5995e+01   5.6859e+00
%    5.8234e-04   1.5131e+05   1.6635e+00   9.3520e+01   7.6432e+00
%    5.7664e-04   1.5973e+05   1.7654e+00   1.0133e+02   6.5584e+00
%    4.5829e-04   1.4154e+05   1.6402e+00   1.0423e+02   6.0049e+00
%    5.1226e-04   1.7632e+05   2.0071e+00   1.1745e+02   4.8882e+00
%    5.6835e-04   1.6482e+05   1.9586e+00   1.2003e+02   5.1203e+00
%    9.0095e-04   1.7277e+05   2.0214e+00   1.2887e+02   8.0283e+00
%    6.6968e-04   1.8404e+05   2.2745e+00   1.3803e+02   7.2548e+00
%    1.5851e-03   1.8827e+05   2.3524e+00   1.4424e+02   1.4627e+01
%    1.8896e-03   1.9548e+05   2.4031e+00   1.5220e+02   1.9639e+01

  \addplot+[] coordinates{
%                           (0 ,  9.9835e-05)
                          (1 ,  1.9708e-04)
                          (2 ,  1.3119e-04)
                          (3 ,  1.6805e-04)
                          (4 ,  1.6885e-04)
                          (5 ,  2.1433e-04)
                          (6 ,  2.6397e-04)
                          (7 ,  4.2550e-04)
                          (8 ,  3.9207e-04)
                          (9 ,  5.8234e-04)
                          (10,  5.7664e-04)
                          (11,  4.5829e-04)
                          (12,  5.1226e-04)
                          (13,  5.6835e-04)
                          (14,  9.0095e-04)
                          (15,  6.6968e-04)
%                          (16,  1.5851e-03)
%                          (17,  1.8896e-03)
                         }; \addlegendentry{TT-MH}; % n=16, eps=0.05, lvls=22, new L

  \addplot+[] coordinates{
%                           (0 ,  2.3525e-04)
                          (1 ,  2.2631e-04)
                          (2 ,  4.3846e-04)
                          (3 ,  3.8441e-04)
                          (4 ,  4.7314e-04)
                          (5 ,  5.0224e-04)
                          (6 ,  4.9617e-04)
                          (7 ,  8.1043e-04)
                          (8 ,  5.0537e-04)
                          (9 ,  6.2053e-04)
                          (10,  6.1985e-04)
                          (11,  7.9118e-04)
                          (12,  8.7507e-04)
                          (13,  9.7497e-04)
                          (14,  1.0842e-03)
                          (15,  1.3008e-03)
%                          (16,  1.5446e-03)
%                          (17,  1.4677e-03)
                         }; \addlegendentry{DRAM}; %  new L

% mean(mean(abs(Q_dram-mean(Q_dram)))./mean(Q_dram))   dram_evals   dram_ttimes, dram_tauint
%    2.3525e-04   6.9570e+06   1.2904e+03   6.0713e+00
%    2.2631e-04   7.1471e+06   7.3738e+02   9.1186e+00
%    4.3846e-04   7.2563e+06   1.5535e+03   1.2813e+01
%    3.8441e-04   7.3077e+06   1.5704e+03   1.6324e+01
%    4.7314e-04   7.3529e+06   1.5830e+03   2.0124e+01
%    5.0224e-04   7.3609e+06   1.6120e+03   2.3100e+01
%    4.9617e-04   7.3577e+06   1.6053e+03   2.4802e+01
%    8.1043e-04   7.3843e+06   1.6181e+03   2.9624e+01
%    5.0537e-04   7.3942e+06   1.6247e+03   3.3845e+01
%    6.2053e-04   7.4006e+06   1.6782e+03   3.8141e+01
%    6.1985e-04   7.3992e+06   1.6793e+03   4.1504e+01
%    7.9118e-04   7.4005e+06   1.7002e+03   4.4962e+01
%    8.7507e-04   7.4045e+06   1.7074e+03   4.9618e+01
%    9.7497e-04   7.4083e+06   1.7186e+03   5.4101e+01
%    1.0842e-03   7.4053e+06   1.7256e+03   5.8414e+01
%    1.3008e-03   7.4088e+06   1.7411e+03   6.4715e+01
%    1.5446e-03   7.4074e+06   1.7192e+03   6.9807e+01
%    1.4677e-03   7.4040e+06   1.7303e+03   7.5152e+01
 \end{axis}
\end{tikzpicture}
}
\resizebox{0.32\linewidth}{!}{%
\begin{tikzpicture}
 \begin{axis}[%
  xmode=log,
  ymode=log,
  xlabel=$D$,
  ytick={1e4,1e5,1e6},
  yticklabels={$^{\phantom{-}}10^4$,$10^5$,$10^6$},
  ylabel={$N_{proposal}$},
  legend style={at={(0.99,0.01)},anchor=south east},
  nodes near coords,point meta=explicit symbolic,
  every node near coord/.append style={anchor=south},
%   y filter/.code={\pgfmathparse{1e-6*(\pgfmathresult)}\pgfmathresult},
  ]

  \addplot+[] coordinates{
%                           (0 ,  7.6800e+02)
                          (1 ,  8.1160e+03)[14]
                          (2 ,  1.9479e+04)[16]
                          (3 ,  6.1585e+04)[18]
                          (4 ,  8.3087e+04)[18]
                          (5 ,  9.3551e+04)
                          (6 ,  1.0766e+05)[19]
                          (7 ,  1.4510e+05)
                          (8 ,  1.4466e+05)[17]
                          (9 ,  1.5131e+05)
                          (10,  1.5973e+05)[17]
                          (11,  1.4154e+05)
                          (12,  1.7632e+05)
                          (13,  1.6482e+05)
                          (14,  1.7277e+05)
                          (15,  1.8404e+05)[17]
%                           (16,  1.8827e+05)
%                           (17,  1.9548e+05)
                         }; \addlegendentry{TT-MH}; % n=16, eps=0.05, lvls=22, new L

  \addplot+[] coordinates{
%                           (0 ,  6.9570e+06 - 2^22)
                          (1 ,  7.1471e+06 - 2^22)
                          (2 ,  7.2563e+06 - 2^22)
                          (3 ,  7.3077e+06 - 2^22)
                          (4 ,  7.3529e+06 - 2^22)
                          (5 ,  7.3609e+06 - 2^22)
                          (6 ,  7.3577e+06 - 2^22)
                          (7 ,  7.3843e+06 - 2^22)
                          (8 ,  7.3942e+06 - 2^22)
                          (9 ,  7.4006e+06 - 2^22)
                          (10,  7.3992e+06 - 2^22)
                          (11,  7.4005e+06 - 2^22)
                          (12,  7.4045e+06 - 2^22)
                          (13,  7.4083e+06 - 2^22)
                          (14,  7.4053e+06 - 2^22)
                          (15,  7.4088e+06 - 2^22)
%                           (16,  7.4074e+06 - 2^22)
%                           (17,  7.4040e+06 - 2^22)
                         }; \addlegendentry{DRAM}; %  new L
 \end{axis}
\end{tikzpicture}
}
\resizebox{0.32\linewidth}{!}{%
\begin{tikzpicture}
 \begin{axis}[%
  xmode=log,
  ymode=log,
  xlabel=$D$,
  ylabel={$\tau$},
  ytick={1e1,1e2},
  yticklabels={$^{\phantom{-}}10^1$,$10^2$},
  ymax=150,
  legend style={at={(0.01,0.99)},anchor=north west},
  ]

  \addplot+[] coordinates{
%                           (0 ,  1.4979e+00)
                          (1 ,  5.3302e+00)
                          (2 ,  2.3840e+00)
                          (3 ,  2.2517e+00)
                          (4 ,  2.3780e+00)
                          (5 ,  3.0751e+00)
                          (6 ,  2.9446e+00)
                          (7 ,  3.7976e+00)
                          (8 ,  5.6859e+00)
                          (9 ,  7.6432e+00)
                          (10,  6.5584e+00)
                          (11,  6.0049e+00)
                          (12,  4.8882e+00)
                          (13,  5.1203e+00)
                          (14,  8.0283e+00)
                          (15,  7.2548e+00)
%                           (16,  1.4627e+01)
%                           (17,  1.9639e+01)
                         }; \addlegendentry{TT-MH}; % n=16, eps=0.05, lvls=22, new L

  \addplot+[] coordinates{
%                           (0 ,  6.0713e+00)
                          (1 ,  9.1186e+00)
                          (2 ,  1.2813e+01)
                          (3 ,  1.6324e+01)
                          (4 ,  2.0124e+01)
                          (5 ,  2.3100e+01)
                          (6 ,  2.4802e+01)
                          (7 ,  2.9624e+01)
                          (8 ,  3.3845e+01)
                          (9 ,  3.8141e+01)
                          (10,  4.1504e+01)
                          (11,  4.4962e+01)
                          (12,  4.9618e+01)
                          (13,  5.4101e+01)
                          (14,  5.8414e+01)
                          (15,  6.4715e+01)
%                           (16,  6.9807e+01)
%                           (17,  7.5152e+01)
                         }; \addlegendentry{DRAM}; %  new L
 \end{axis}
\end{tikzpicture}
}
\end{figure*}
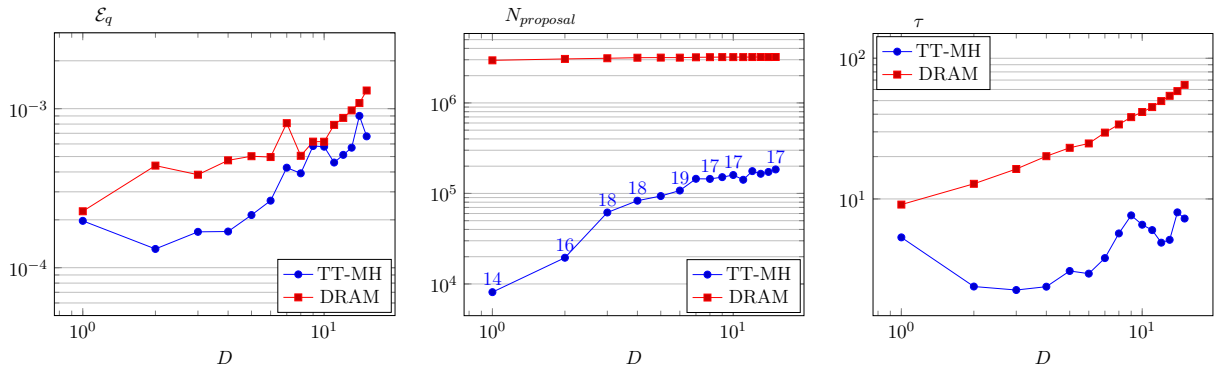

\subsection{Rosenbrock function}
\label{sec:rf}

As a benchmark example with particularly long tails (and hence potentially large autocorrelation times in MCMC), we consider the PDF induced by the Rosenbrock function $\pi(\theta) \propto \exp\left(-\frac{1}{2} r(\theta)\right)$, where
\begin{equation}
r(\theta) = \sum_{k=1}^{d-1}\left[\theta_k^2 + \left(\theta_{k+1} + 5 \cdot (\theta_{k}^2+1)\right)^2 \right].
\label{eq:rosen}
\end{equation}
The dimension $d$ can be increased arbitrarily.
The parameters for the TT approximation are chosen to be $\delta=3 \cdot 10^{-3}$ and $n=128$ for $\theta_1,\ldots,\theta_{d-2}$, $n=512$ for $\theta_{d-1}$ and $n=4096$ for $\theta_d$. Each $\theta_k$ is restricted to a finite interval $[-a_k,a_k]$, where $a_d=200$, $a_{d-1}=7$ and $a_k = 2$ otherwise.

Fig. \ref{fig:rosen-samples} shows certain projections of $N=2^{17}$ sampling points produced with TT-MH and DRAM for $d=32$.
% \begin{multicols}{2}
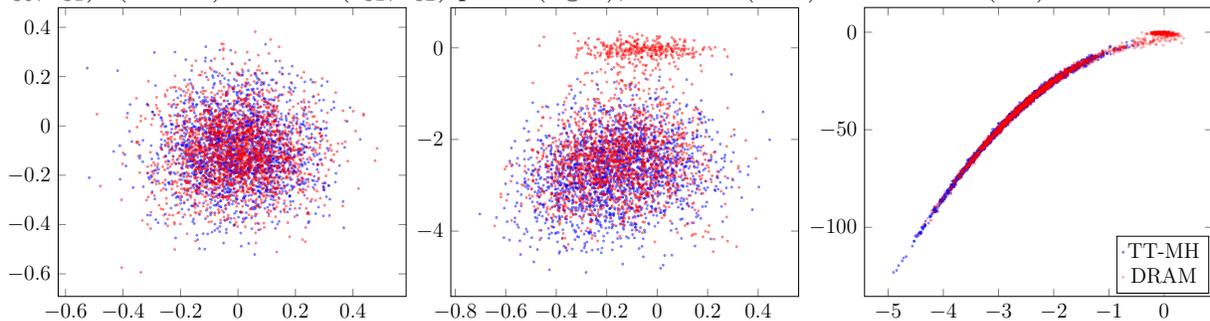
\begin{figure*}
\centering
\caption{Rosenbrock function ($d=32$): $N=2^{17}$ samples projected to the $(\theta_1,\theta_2)$- (left), the $(\theta_{30},\theta_{31})$- (middle) and the  $(\theta_{31},\theta_{32})$-plane (right); TT-MH (blue) and DRAM (red).}
\label{fig:rosen-samples}
\ifrosensamples
\resizebox{0.32\linewidth}{!}{%
\begin{tikzpicture}
\begin{axis}
\addplot[%
    scatter=true,
    only marks,
    mark=*,
    mark size=0.5,
    opacity=0.5,
    color=blue,
    scatter/use mapped color={draw=none,fill=blue},
]  table[header=false,x index=0, y index=1]{rosen-chains-d32.dat};
\addplot[%
    scatter=true,
    only marks,
    mark=*,
    mark size=0.5,
    opacity=0.5,
    color=red,
    scatter/use mapped color={draw=none,fill=red},
]  table[header=false,x index=32, y index=33]{rosen-chains-d32.dat};
\end{axis}
\end{tikzpicture}
}
\resizebox{0.32\linewidth}{!}{%
\begin{tikzpicture}
\begin{axis}
\addplot[%
    scatter=true,
    only marks,
    mark=*,
    mark size=0.5,
    opacity=0.5,
    color=blue,
    scatter/use mapped color={draw=none,fill=blue},
]  table[header=false,x index=29, y index=30]{rosen-chains-d32.dat};
\addplot[%
    scatter=true,
    only marks,
    mark=*,
    mark size=0.5,
    opacity=0.5,
    color=red,
    scatter/use mapped color={draw=none,fill=red},
]  table[header=false,x index=61, y index=62]{rosen-chains-d32.dat};
\end{axis}
\end{tikzpicture}
}\resizebox{0.32\linewidth}{!}{%
\begin{tikzpicture}
\begin{axis}[legend style={at={(0.99,0.01)},anchor=south east}]
\addplot[%
    scatter=true,
    only marks,
    mark=*,
    mark size=0.5,
    opacity=0.5,
    color=blue,
    scatter/use mapped color={draw=none,fill=blue},
]  table[header=false,x index=30, y index=31]{rosen-chains-d32.dat}; \addlegendentry{TT-MH};
\addplot[%
    scatter=true,
    only marks,
    mark=*,
    mark size=0.5,
    opacity=0.3,
    color=red,
    scatter/use mapped color={draw=none,fill=red},
]  table[header=false,x index=62, y index=63]{rosen-chains-d32.dat}; \addlegendentry{DRAM};
\end{axis}
\end{tikzpicture}
}
\else
\includegraphics[width=0.32\linewidth]{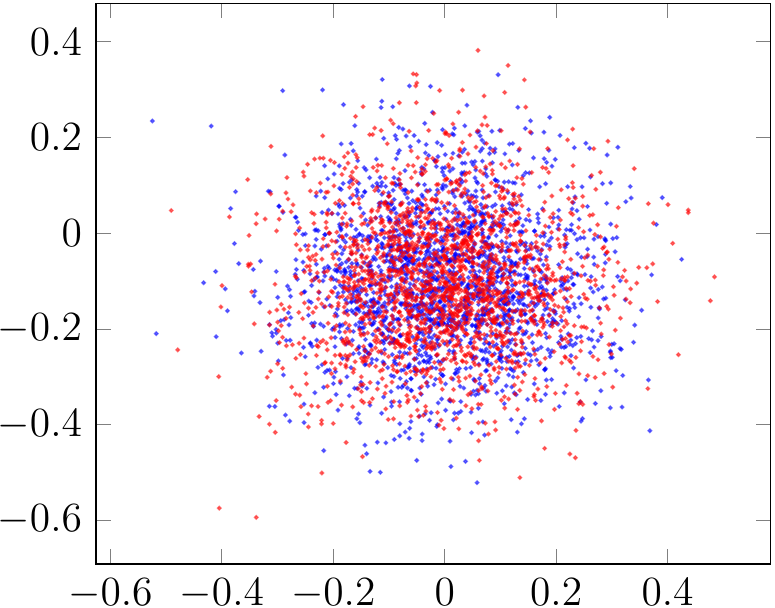}
\includegraphics[width=0.307\linewidth]{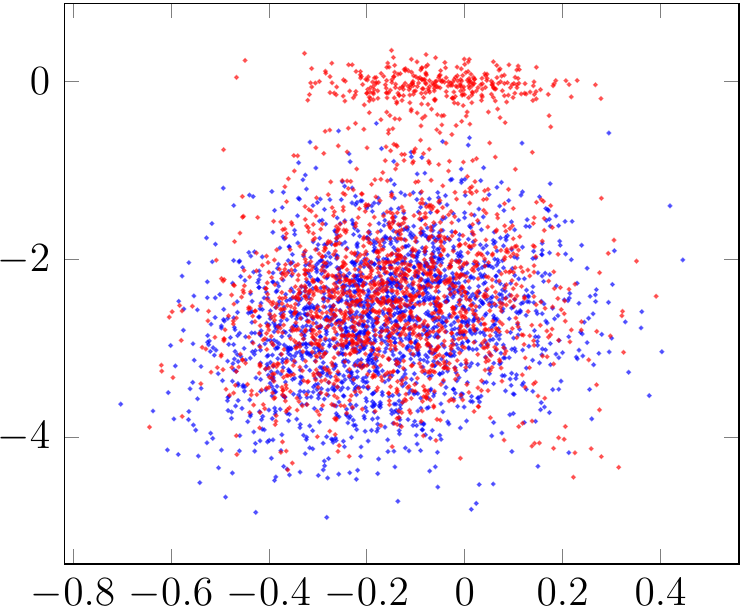}
\includegraphics[width=0.323\linewidth]{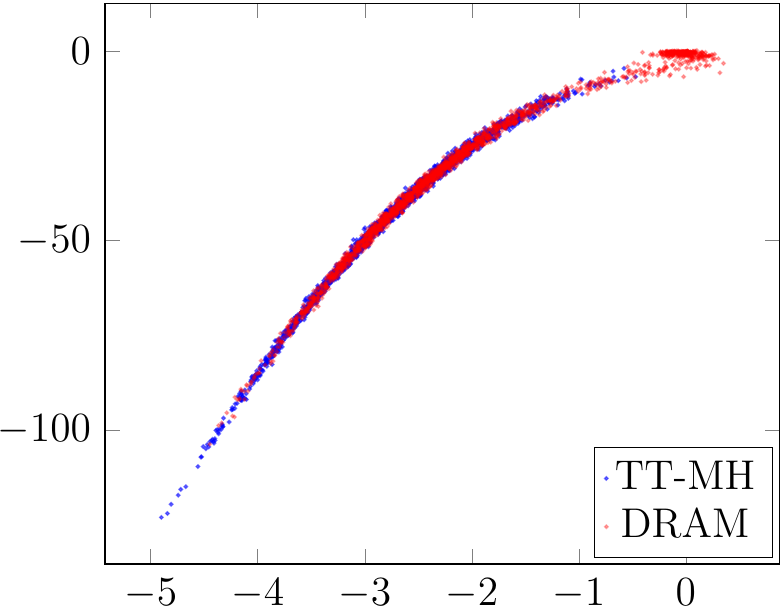}
\fi
\end{figure*}
We see that although the density function is reasonably compact and isotropic in the first variables, it is highly concentrated in the last variable.
DRAM requires a significant number of burn-in iterations, which can be seen in Fig. \ref{fig:rosen-samples} (middle and right) as the red cloud of samples that are not overlapped by blue ones.
In order to eliminate the burn-in in DRAM, we compute $2^{20}$ samples and discard the first quarter of the chain. The difference is even more significant if we look at the integrated autocorrelation times in Tab. \ref{tab:rosen-tau}. For TT-MH the IACT stays close to $1$ for all considered dimensions, while for DRAM it exceeds 100 for larger $d$.
\begin{table}[t]
\centering
\caption{Rosenbrock function example: IACT.}
\label{tab:rosen-tau}
\begin{tabular}{c|ccccc}
$d$    & 2        & 4           & 8             & 16       & 32 \\ \hline
TT-MH     & 1.096   & 1.080      & 1.100        & 1.079   & 1.084   \\
DRAM   & 61.54  & 26.63     & 45.01       & 84.02  & 169.57
\end{tabular}
\end{table}

\subsection{Inverse diffusion problem}
\label{sec:ff}

Finally, we use our new TT-CD sampler to explore the  posterior distribution arising from a Bayesian formulation of an infinite-dimensional inverse problem,
as formalized in~\cite{stuart-bayes-2010}.

Let $X$ and $V$ be two infinite-dimensional function spaces -- it is sufficient to consider separable Banach spaces --  and let $\mathcal G: X \rightarrow V$ be a (measurable and well-posed) forward map. Consider the inverse problem of finding $\kappa \in X$, an input to $\mathcal G$, given some noisy observations $y \in \mathbb{R}^{m_0}$ of some functionals of the output $u \in V$. In particular, we assume a (measurable) observation operator $Q: V \rightarrow \mathbb{R}^{m_0}$, such that
$$
y = Q(\mathcal G(\kappa)) + \eta,
$$
where $\eta \in \mathbb{R}^{m_0}$ is a mean-zero random variable that denotes the observational noise.
The inverse problem is clearly under-determined when $m_0 \ll \text{dim} (X)$ and in most mathematical models the inverse of the map $\mathcal G$ is ill-posed.

We do not consider prior modelling in any detail, and present here a stylized Bayesian formulation designed to highlight the computational structure and cost. We simply state a prior measure $\mu_0$, to model $\kappa$ in the absence of observations $y$.
%The Bayesian formalism to solve such an inverse problem assumes that the prior beliefs about $\kappa$ are given in terms of a prior measure $\mu_0$, which describes the distribution of $\kappa$ prior to observing the data $y$.
The posterior distribution $\mu^y$ over $\kappa | y$, the unknown coefficients conditioned on observed data, is given by Bayes' theorem for general measure spaces,
\begin{equation}
\label{eq:bayes}
\frac{d\mu^y}{d\mu_0} (u) = \frac{1}{Z}L(\kappa),
\end{equation}
where the left hand side is the Radon-Nikodym derivative, $L$ is the likelihood function, and $Z$ is the normalizing constant \cite{stuart-bayes-2010}. % For zero mean Gaussian noise $\eta$ with covariance matrix $\sigma_n^2 I$,
% $$
% L(\kappa) = \exp\left(-\frac{|y - Q(\mathcal{G}(\kappa))|^2}{2 \sigma_n^2}\right). %\quad \text{and} \quad Z = \mathbb{E}_{\mu_0} L(G(\kappa)).
% $$

For computing, we have to work with a finite dimensional approximation $\kappa_d \in X_d \subset X$ of the latent field $\kappa$ such that $\text{dim} (X_d) = d \in \mathbb{N}$, and define $\kappa_d$ as a deterministic function of a $d$-dimensional parameter $\theta := (\theta_1,\ldots,\theta_d)$. Typically, we require that $\kappa_d \to \kappa$ as $d \to \infty$, but we will not focus on that convergence here and instead fix $d \gg 1$.
To be able to apply the TT representation, we set $\theta_k \in [a_k,b_k]$ with $a_k < b_k$, for all $k=1,\ldots,d$, and then $\kappa_d$ maps the tensor-product domain $\Gamma_d := \prod_{k=1}^d [a_k,b_k]$ to $X_d$. We denote by $\pi_0(\theta)$ and $\pi(\theta) = \pi^y(\theta)$ the probability density functions of the pull-back measures of the prior and posterior measures $\mu_0$ and $\mu^y$ under the map $\kappa_d: \Gamma_d \to X_d$, respectively, and specify that map so that $\pi_0(\theta) = 1/|\Gamma_d|$, i.e. the prior distribution over $\theta$ is uniform.

We can then compute TT approximations of the posterior density $\pi(\theta)$ as in the previous examples by using Bayes' formula \eqref{eq:bayes}, i.e.
$$
\pi(\theta) = \frac{1}{Z} L(\kappa_d(\theta)), \quad \text{where} \quad Z = \int\limits_{\Gamma_d} L(\kappa_d(\theta)) d\pi_0(\theta)\,.
$$

Consider some quantity of interest in the form of another functional $F: V \rightarrow \mathbb{R}$ of the model output $\mathcal G(\kappa_d)$. The posterior expectation of $F$, conditioned on measured $y$, can be computed as
\begin{equation}
\mathbb{E}_{\pi}\left[ F(\mathcal G(\kappa_d))\right] = \frac{\mathbb{E}_{\pi_0}\left[L(\kappa_d) F(\mathcal G(\kappa_d))\right]}{\mathbb{E}_{\pi_0} \left[ L(\kappa_d)\right]}.
\label{eq:post_inf}
\end{equation}

\subsubsection{Stylized elliptic problem and parametrization}

As an example, we consider the forward map defined by the stochastic diffusion equation
\begin{equation}
-\nabla \cdot \big(\kappa_d(\theta) \nabla u\big) = 0 \quad \mbox{on} \quad D:= (0,1)^2,
\label{eq:pdeproblem}
\end{equation}
with Dirichlet boundary conditions $u|_{x_1=0}=1$ and $u|_{x_1=1}=0$, as well as homogeneous Neumann conditions otherwise \cite{scheichl-qmc-bayes-2017}, which depends on an unknown (parametrized) diffusion coefficient $\kappa_d \in X_d \subset L_\infty(D)$.

For this example, we take each of the parameters $\theta_k$, $k=1,\ldots,d,$ to be uniformly distributed on $[-\sqrt{3},\sqrt{3}]$. % with mean $0$ and unit variance, i.e. $\Gamma_d := [-\sqrt{3},\sqrt{3}]^d$ and $\pi_0(\theta) = 1/|\Gamma_d|$.
Then, for any $\theta \in \Gamma_d$ and $x = (x_1,x_2) \in D$,  the logarithm of the diffusion coefficient at $x$ is defined by the following expansion:
\begin{equation}
\begin{split}
\ln \kappa_d(\theta, x) & = \sum_{k=1}^{d} \theta_k \, \sqrt{\eta_k} \, \cos(2\pi \rho_1(k) x_1) \cos(2\pi \rho_2(k) x_2),\\
\rho_1(k) & = k - \tau(k)\frac{(\tau(k)+1)}{2}, \ \  \rho_2(k) = \tau(k)-\rho_1(k), \\
\tau(k) & = \left\lfloor -\frac{1}{2} + \sqrt{\frac{1}{4}+2k} \right\rfloor \ \text{and} \\
\eta_k & = k^{-(\nu+1)}/K, \quad K = \sum_{k=1}^{d} k^{-(\nu+1)}.
\end{split}
\label{eq:kle_art}
\end{equation}
The expansion is similar to the one proposed in \cite{eigel-adapt-stoch-fem-2014}, and %although there it was the deviation of the diffusion coefficient from some mean coefficient $\overline{\kappa}$ that was expanded and not $\ln \kappa_d$.
%The series expansion in \eqref{eq:kle_art}
mimics the asymptotic behaviour of the Karhunen-Lo\`eve expansion of random fields with Mat\'ern covariance function and smoothness parameter $\nu$ in two dimensions, in that the norms of the individual terms decay algebraically with the same rate. However, realizations do not have the same qualitative features and we use it purely to demonstrate the computational efficiency of our new TT samplers.

To discretize the partial differential equation (PDE) in \eqref{eq:pdeproblem} we tessellate the spatial domain $D$ with a uniform rectangular grid $T_h$ with mesh size $h$. Then, we approximate the exact solution $u \in V := H^1(D)$ that satisfies the Dirichlet boundary conditions with the continuous, piecewise bilinear finite element (FE) approximation $u_h \in V_h$ associated with $T_h$. To find $u_h$ we solve the resulting Galerkin system using a sparse direct solver.

For this example, we take the observations to be $m_0$ noisy local averages of the PDE solution over some subsets $D_{i} \subset D$, $i=1,\ldots,m_0$, i.e.,
$$
Q_{i}(\mathcal{G}(\theta)) = \frac{1}{|D_{i}|}\int_{D_i} u_h(x,\theta) dx, \quad i=1,\ldots,m_0\,.
$$
We take observation noise to be additive, distributed as i.i.d. zero-mean Gaussian noise with variance $\sigma_n^2$, giving the likelihood function,
$$
L(\theta) = \exp\left(-\frac{\left| Q(\mathcal{G}(\theta)) - y\right|^2}{2\sigma_n^2}\right),
$$
and posterior distribution
$\pi(\theta) = \frac{1}{Z} L(\theta)$, with the normalization constant $Z = \int\nolimits_{[-\sqrt{3},\sqrt{3}]^d} L(\theta) d\theta.$
%where we write short $f(\theta)$ instead of $f(\kappa_d(\theta))$, for any functional $f$ of $\kappa_d(\theta)$; in particular, for
%Here, $\mathcal{G}(\kappa_d(\theta))$, denotes the FE solution $u_h$ at the parameter value $\theta$, and $Q$ denotes the observation operator.

In our experiments, the sets $D_{i}$ are square domains with side length $2/(\sqrt{m_0}+1)$, centred at the interior vertices of a uniform Cartesian grid on $D=[0,1]^2$ with grid size $1/(\sqrt{m_0}+1)$, that form an overlapping partition of $D$. We consider an academic problem with synthetic data for these $m_0$ local averages from some ``true'' value $\theta_*$. In particular, we evaluate the observation operator at  $\theta_* =(\theta_0,,\theta_0,\ldots,\theta_0)$, for some fixed $0 \not=\theta_0 \in (-\sqrt{3},\sqrt{3})$,  and synthesize data by then adding independent normally distributed noise $\eta_* \sim \mathcal{N}(0,\sigma_n^2 I)$, such that $y = Q(\mathcal{G}(\theta_*)) + \eta_*$.

We consider two quantities of interest. The first is the average flux at $x^1=1$. This can be computed as~\cite{scheichl-mlmc-further-2013}
\begin{equation}
\label{def_flux}
\begin{split}
F(\mathcal{G}(\theta)) & = -\int_{0}^{1}\int_{0}^{1} \kappa_d(x,\theta) \nabla w_h(x) \nabla u_h(x,\theta) dx \\
& = -\int_{0}^{1} \kappa_d(x,\theta) \left.\frac{\partial u_h(x,\theta)}{\partial \mathbf{n}}\right|_{x^1=1}dx^2,
\end{split}
\end{equation}
where $w_h \in V_h$ is any FE function that satisfies the Dirichlet conditions at $x^1=0$ and $x^1=1$. This formula for the average flux is a smooth function of $\theta$, which ensures a fast convergence for QMC-based quadrature rules, with an order close to $N^{-1}$.
However, we also consider the discontinuous indicator function $\mathbb{I}_{F(\theta)>1.5}$, to estimate the probability that the average flux in \eqref{def_flux} becomes larger than $1.5$, i.e.,
$$
P_{F>1.5} = \mathrm{Prob}\left(F(\mathcal{G}(\theta))>1.5\right) = \mathbb{E}_{\pi} \left[\mathbb{I}_{F(\theta)>1.5}\right].
$$
As we shall see, the non-smoothness of $\mathbb{I}_{F(\theta)>1.5}$ reduces the order of convergence of the QMC quadrature to the basic Monte Carlo rate $N^{-1/2}$.
For the same reason, this function lacks a low-rank TT decomposition, and hence we cannot compute its expectation using a tensor product quadrature directly.
The mean field flux $F|_{\theta=0}=1$ (in the units used), and the probability $P_{F>1.5}$ are both of the order of $0.1$.

The default parameters used in the stochastic model and for function approximation are shown in Table \ref{tab:ff-par}. We will make it clear when we change any of those default parameters.
\begin{table}[tbh]
\centering
\caption{Default model and discretization parameters for the inverse diffusion example.}
\label{tab:ff-par}
\begin{tabular}{ccc|cc|ccc}
   $\nu$ & $\sigma_n^2$   & $\theta_0$ & $m_0$  & $h$       & $d$           & $\delta$  & $n$   \\\hline
  $2$ & $0.01$         & $1.5$ & $9$  &  $2^{-6}$  & $11$          & $0.1$     & $32$    \\
\end{tabular}
\end{table}

The TT approximation $\tilde\pi$ can be computed directly by the TT cross algorithm, as in the previous examples. For a TT tolerance of $\delta = 0.1$, this requires about $10^4-10^5$ evaluations of $\pi$. However, since here the computation of each value of $\pi(\theta)$ involves the numerical solution of the PDE \eqref{eq:pdeproblem} this leads to a significant set-up time.
This set-up time can be hugely reduced, by first building a TT approximation $\tilde u_h(\cdot,\theta)$ of the FE solution $u_h(\cdot,\theta)$ and then using $\tilde u_h(\cdot,\theta)$ in the TT cross algorithm for building $\tilde\pi$ instead of $u_h(\cdot,\theta)$.

It was shown in \cite{ds-alscross-2019} that a highly accurate approximation of $u_h(\cdot,\theta)$ in the TT format can be computed using a variant of the TT cross algorithm, the alternating least-squares cross (ALS-cross) algorithm, that only requires $\mathcal{O}(r)$ PDE solves, if the TT ranks to approximate $u_h(\cdot,\theta)$ up to the discretization error are bounded by $r$. Moreover, the rank grows only logarithmically with the required accuracy. We will see, below, that $r < 100$ for this model problem for $h=2^{-6}$, significantly reducing the number of PDE solves required in the set-up phase.

Since the observation operator $Q$ consists of integrals of the PDE solution over subdomains of the spatial domain $D$, when applied to a function given in TT format it can be evaluated at a cost that is smaller than $r$ PDE solves on $T_h$ without any increase in the TT rank \cite{ds-alscross-2019}.  Finally, to compute an approximation of $\pi$ via the TT cross algorithm we use the significantly cheaper TT surrogate $Q(\tilde u_h(\cdot,\theta))$ in each evaluation of $\pi(\theta)$ instead of computing the actual FE solution $u_h(\cdot,\theta)$. Since $\tilde u_h(\cdot,\theta)$ is accurate up to the FE discretization error in $V_h$ -- which in this model problem for $h=2^{-6}$ is of $\mathcal{O}(10^{-4})$ -- this has essentially no impact on the accuracy of the resulting TT approximation~$\tilde\pi$ (especially for TT accuracy $\delta=0.1$).

\subsubsection{Set-up cost and accuracy of TT approximation}

As in the shock absorber example, we test how the quality of the Markov chain produced by TT-MH depends on the error between $\tilde\pi$ and $\pi$. In Figure \ref{fig:ff-eps} (left), we show the rejection rates, IACT and error estimates \eqref{eq:eps},\eqref{eq:epsl1} for different stopping tolerances $\delta$ and grid sizes $n$. In the top plot, we fix $\delta=10^{-3}$ and vary $n$, while in the bottom plot, $n$ is fixed to $512$ and $\delta$ is varied. The other model parameters are set according to Table \ref{tab:ff-par}, and the chain length is $N=2^{16}$. The behaviour is as in the shock absorber example and as predicted in Lemma \ref{lem:acc_rate}.
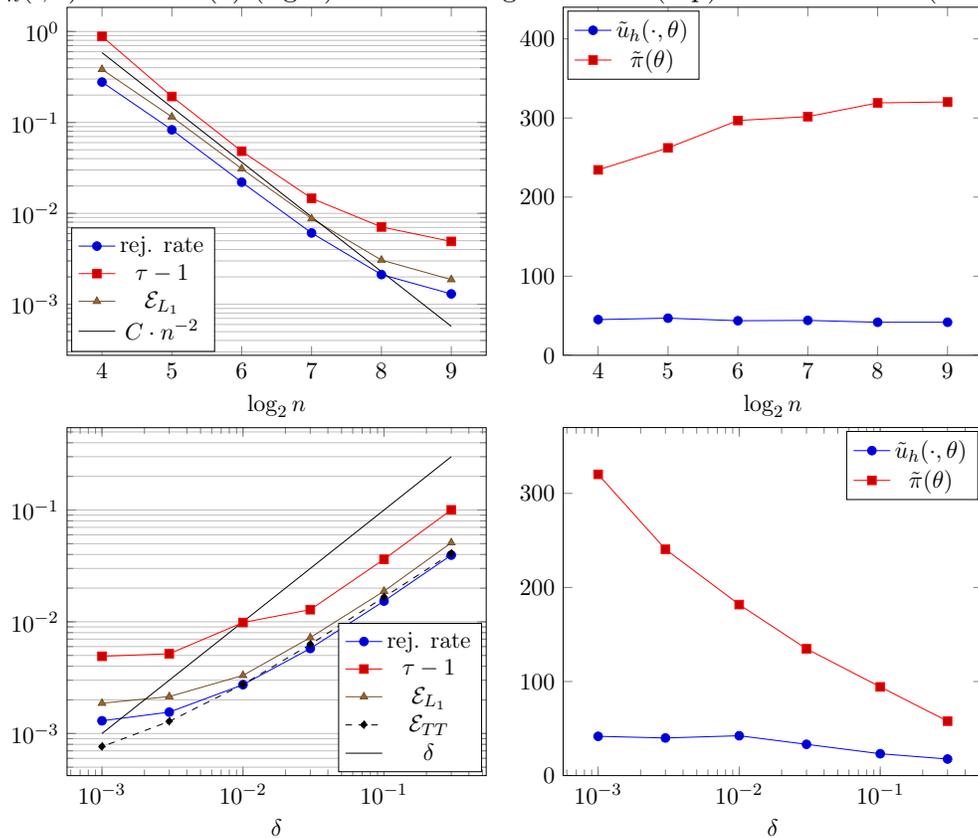
\begin{figure*}
\centering
\caption{Inverse diffusion problem: Rejection rate, IACT and errors (left), as well as maximal TT ranks for $\tilde u_h(\cdot,\theta)$ and for $\tilde\pi (\theta)$ (right) for different grid sizes $n$ (top) and values of $\delta$ (bottom).}
\label{fig:ff-eps}
\resizebox{0.4\linewidth}{!}{
\begin{tikzpicture}
  \begin{axis}[
      xmode=normal,
      ymode=log,
      xlabel={$\log_2{n}$},
      legend style={at={(0.01,0.01)},anchor=south west},
      x filter/.code={\pgfmathparse{log2(\pgfmathresult)}\pgfmathresult},
   ]
   \addplot+ coordinates{
                         (16  , 2.7798e-01)
                         (32  , 8.2954e-02)
                         (64  , 2.2023e-02)
                         (128 , 6.0844e-03)
                         (256 , 2.1219e-03)
                         (512 , 1.2999e-03)
                         }; \addlegendentry{rej. rate}; % tt_invcdf1
   \addplot+ coordinates{
                         (16  , 8.8435e-01)
                         (32  , 1.9279e-01)
                         (64  , 4.8222e-02)
                         (128 , 1.4633e-02)
                         (256 , 7.0839e-03)
                         (512 , 4.9131e-03)
                         }; \addlegendentry{$\tau-1$};  % tt_invcdf1

  \addplot+[mark=triangle*] coordinates{
                          (16  , 3.8368e-01)
                          (32  , 1.1514e-01)
                          (64  , 3.0987e-02)
                          (128 , 8.7409e-03)
                          (256 , 3.0637e-03)
                          (512 , 1.8683e-03)
                          };  \addlegendentry{$\mathcal{E}_{L_1}$};

%   \addplot+[green!50!black,mark options={green!50!black}] coordinates{
%                           (16  , 2.6548e+00)
%                           (32  , 1.4063e+00)
%                           (64  , 1.6041e+00)
%                           (128 , 1.8393e+00)
%                           (256 , 1.8699e+00)
%                           (512 , 3.4801e+00)
%                           };  \addlegendentry{$w^*$};

  \addplot+[no marks, domain=16:512, black] {(x^(-2))*1.5e2}; \addlegendentry{$C \cdot n^{-2}$};
  \end{axis}
\end{tikzpicture}
}
\resizebox{0.39\linewidth}{!}{
\begin{tikzpicture}
  \begin{axis}[
      xmode=normal,
      ymode=normal,
      xlabel={$\log_2{n}$},
      ymin=0,
      ymax=440,
%       axis y line*=left, y label style={at={(-0.1,1.0)},anchor=south east,color={blue}}, every y tick label/.style={blue},
      legend style={at={(0.01,0.99)},anchor=north west},
      x filter/.code={\pgfmathparse{log2(\pgfmathresult)}\pgfmathresult},
   ]
   \addplot+ coordinates{
                         (16  , 4.5125e+01)
                         (32  , 4.6938e+01)
                         (64  , 4.3688e+01)
                         (128 , 4.4125e+01)
                         (256 , 4.1750e+01)
                         (512 , 4.1750e+01)
                        };  \addlegendentry{$\tilde u_h(\cdot,\theta)$};

   \addplot+ coordinates{
                         (16  , 2.3438e+02)
                         (32  , 2.6219e+02)
                         (64  , 2.9681e+02)
                         (128 , 3.0169e+02)
                         (256 , 3.1906e+02)
                         (512 , 3.2025e+02)
                        };  \addlegendentry{$\tilde\pi (\theta)$};
  \end{axis}
\end{tikzpicture}
}\\
\resizebox{0.4\linewidth}{!}{
\begin{tikzpicture}
  \begin{axis}[
      xmode=log,
      ymode=log,
      xlabel={$\delta$},
      legend style={at={(0.99,0.01)},anchor=south east},
   ]
   \addplot+ coordinates{
                         (0.3   , 3.9500e-02)
                         (0.1   , 1.5313e-02)
                         (0.03  , 5.7669e-03)
                         (0.01  , 2.7342e-03)
                         (0.003 , 1.5545e-03)
                         (0.001 , 1.2999e-03)
                         }; \addlegendentry{rej. rate}; % invcdf1
   \addplot+ coordinates{
                         (0.3   , 1.0038e-01)
                         (0.1   , 3.6247e-02)
                         (0.03  , 1.2833e-02)
                         (0.01  , 9.8507e-03)
                         (0.003 , 5.1776e-03)
                         (0.001 , 4.9131e-03)
                         }; \addlegendentry{$\tau-1$}; % invcdf1

  \addplot+[mark=triangle*] coordinates{
                         (0.3   , 5.0994e-02)
                         (0.1   , 1.8794e-02)
                         (0.03  , 7.2179e-03)
                         (0.01  , 3.3122e-03)
                         (0.003 , 2.1415e-03)
                         (0.001 , 1.8683e-03)
                         };  \addlegendentry{$\mathcal{E}_{L_1}$};

%   \addplot+[green!50!black,mark options={green!50!black}] coordinates{
%                          (0.3   , 1.9417e+02)
%                          (0.1   , 1.4539e+01)
%                          (0.03  , 6.6937e+00)
%                          (0.01  , 4.5349e+00)
%                          (0.003 , 2.8637e+00)
%                          (0.001 , 3.4801e+00)
%                          };  \addlegendentry{$w^*$};

   \addplot+[dashed,black,mark=diamond*,mark options={black}] coordinates{
                         (0.3   , 4.0759e-02)
                         (0.1   , 1.6607e-02)
                         (0.03  , 6.2794e-03)
                         (0.01  , 2.7477e-03)
                         (0.003 , 1.2835e-03)
                         (0.001 , 7.6360e-04)
                         }; \addlegendentry{$\mathcal{E}_{TT}$}; % invcdf1

  \addplot+[solid,no marks, domain=1e-3:3e-1, black] {x}; \addlegendentry{$\delta$};
  \end{axis}
\end{tikzpicture}
}
\resizebox{0.39\linewidth}{!}{
\begin{tikzpicture}
  \begin{axis}[
      xmode=log,
      ymode=normal,
      xlabel={$\delta$},
%       axis y line*=left, y label style={at={(-0.1,1.0)},anchor=south east,color={blue}}, every y tick label/.style={blue},
      legend style={at={(0.99,0.99)},anchor=north east},
      ymin=0,
      ymax=370,
   ]
   \addplot+ coordinates{
                         (0.3   , 1.7625e+01)
                         (0.1   , 2.3250e+01)
                         (0.03  , 3.3250e+01)
                         (0.01  , 4.2375e+01)
                         (0.003 , 40)
                         (0.001 , 4.1750e+01)
                        };  \addlegendentry{$\tilde u_h(\cdot,\theta)$};

   \addplot+ coordinates{
                         (0.3   , 5.7875e+01)
                         (0.1   , 9.4312e+01)
                         (0.03  , 1.3481e+02)
                         (0.01  , 1.8181e+02)
                         (0.003 , 2.4062e+02)
                         (0.001 , 3.2025e+02)
                        };  \addlegendentry{$\tilde\pi (\theta)$};
  \end{axis}
\end{tikzpicture}
}
\end{figure*}

In Fig. \ref{fig:ff-eps} (right), we demonstrate the benefit of first computing a TT approximation $\tilde u_h(\cdot,\theta)$ of the FE solution $u_h(\cdot,\theta)$, as described in the previous subsection. We see that the TT ranks to approximate $u_h$ are significantly smaller than the TT ranks to approximate the density $\pi$ to the same accuracy. In both cases, the TT ranks are independent of $n$, for $n$ sufficiently large, and they show only a logarithmic dependence on the TT accuracy $\delta$. However, the growth is significantly faster for $\ttpi$ than for $\tilde u_h$. For the default parameters in Table \ref{tab:ff-par}, the ranks of $\tilde u_h(\cdot,\theta)$ and $\ttpi(\theta)$ are 26 and 82, respectively, and the numbers of function evaluations to build the TT surrogates are about $100$ and about $53000$, respectively. The advantage is that with the surrogate $\tilde u_h$ available those $53000$ evaluations of $\pi$, using $\tilde u_h$ instead of $u_h$ in the data misfit functional, are significantly cheaper and the PDE only has to be actually solved $100$ times.

\subsubsection{Convergence of the expected quantities of interest}

In this section we investigate the convergence of estimates of the expected value of the quantities of interest, and the computational complexity of the different methods. For the TT approximation of the density function $\pi$ we fix $n=32$ and $\delta=0.1$. For the TT approximation of $u_h$ we choose a TT tolerance of $10^{-4}$, which is equal to the discretization error for $h=2^{-6}$.

To compute the posterior expectations of the QoIs in \eqref{eq:post_inf} we compare two approaches that use our TT-CD sampling procedure:
\begin{description}
 \item[{[TT-MH]}] (Sec.~\ref{sec:ttmcmc}) Metropolis-Hastings with independence proposals sampled via the TT-CD sampling procedure from the approximate distribution $\ttpi$.
%  \item[{[TT-qCV]}]  (Sec.~\ref{sec:2l}) Using QMC quadrature with respect to the the approximate density $\ttpi$ in a similar way to a control variate. Bias correction uses a Metropolis-Hastings procedure targeting the correct density $\pi$, with independence proposals sampled via the TT-CD sampler.
 \item[{[TT-qIW]}]  (Sec.~\ref{sec:isqmc}) Using the approximate density $\ttpi$ as an importance weight and estimating the expected value and the normalizing constant via a randomized QMC lattice rule.
\end{description}
Moreover, we test the two-level versions of both methods described in Section~\ref{sec:multi}.

To benchmark the TT approaches, we use again DRAM with the initial covariance chosen to be the identity and discard the first $N/4$ samples. However, as a second benchmark, we also compute the posterior expectation directly by applying QMC to the two terms in the ratio estimate (\textbf{QMC-rat}), as defined in \eqref{eq:post_inf} and analysed in \cite{scheichl-qmc-bayes-2017}. The QMC method in TT-qIW is again the randomized rank-1 lattice rule with product weights $\gamma_k = 1/k^2$ and generating vector from the file {\tt \verb+lattice-39102-1024-1048576.3600+} at {\tt \verb+http://web.maths.unsw.edu.au/~fkuo/+}.
% In the TT-qCV approach, the numbers $N_0$ and $N_1$ of samples for the two parts of the estimator are chosen adaptively, as in \cite{GilesWaterhouse-2009,Scheichl-mlqmc-lognorm-2017}, to optimize the computational efficiency for a given accuracy. As discussed in Section \ref{sec:2l}, for smooth QoIs we expect a relationship close to $N_1 \sim \eps N_0^2$, whereas for non-smooth QoIs it will be closer to $N_1 \sim \varepsilon N_0$, where $\eps$ is the accuracy of the TT approximation of $\pi$.
In order to reduce random fluctuations in the results, we average $16$ runs of each approach in each experiment. The rejection rate and the IACT for TT-MH and DRAM are shown in Table \ref{tab:ff-chain}. Notice that the autocorrelation times of DRAM for the coordinates $\theta$ and for the quantity of interest $F$ differ significantly, since the latter coordinates have a weaker influence on $F$.

\begin{figure*}
\centering
\caption{Inverse diffusion problem: Relative errors for the average flux (left) and for the probability of the flux exceeding $1.5$ (right) for different numbers of samples $N$.}
% error = mean(abs(Fl_post - mean(Fl_post,2))./mean(Fl_post,2), 2)
% time is for 1 CPU!!!!
% D.H. [coeff+solve+logPi] + ([Pi] + [invcdf] + [mcmc])/16
\label{fig:ff-err-N}
\resizebox{0.39\linewidth}{!}{
\begin{tikzpicture}
  \begin{axis}[%
  xmode=normal,
  ymode=log,
  xlabel=$\log_2 N$,
  ylabel={relative error for $\mathbb{E}_{\pi}[F]$},
  legend style={at={(0.99,0.99)},anchor=north east},
  x filter/.code={\pgfmathparse{log2(\pgfmathresult)}\pgfmathresult},
  ]
  \addplot+[] coordinates{(2^9 , 9.5709e-03)
                          (2^10, 4.7039e-03)
                          (2^11, 3.6251e-03)
                          (2^12, 2.6660e-03)
                          (2^13, 1.6596e-03)
                          (2^14, 1.0000e-03)
                          (2^15, 6.8202e-04)
                          (2^16, 6.1045e-04)
                          (2^17, 3.9070e-04)
                          (2^18, 3.2234e-04)
                          }; % \addlegendentry{TT-MH};  % with tt_invcdf1, nPy=32, eps=0.1 (err_Pi=3.3923e-02)   Metropolis

%  \addplot+[] coordinates{
%                                         (128         + 128       ,    4.1389e-03)
%                                         (256         + 256       ,    4.2016e-03)
%                                         (5.1200e+02  + 6.8538e+02,    1.8344e-03)
%                                         (1.0240e+03  + 2.5694e+03,    1.3314e-03)
%                                         (2.0480e+03  + 1.0710e+04,    6.0018e-04)
%                                         (4.0960e+03  + 4.2686e+04,    2.8323e-04)
%                                         (8.1920e+03  + 1.7243e+05,    1.0299e-04)
%                                         (1.6384e+04  + 6.7867e+05,    7.4393e-05)
%                                                                     }; % \addlegendentry{TT-qCV};  % -rr;  N1 = 3e-2*N0^2*eps

  % \addplot+[] coordinates{(2^9 ,    9.5569e-03)
  %                         (2^10,    4.3620e-03)
  %                         (2^11,    3.3179e-03)
  %                         (2^12,    2.3285e-03)
  %                         (2^13,    1.5445e-03)
  %                         (2^14,    8.6044e-04)
  %                         (2^15,    7.4067e-04)
  %                         (2^16,    5.9136e-04)
  %                         (2^17,    3.7811e-04)
  %                         (2^18,    2.8176e-04)
  %                         }; % \addlegendentry{TTrIS};  % with tt_invcdf1, nPy=32, eps=0.1 (err_Pi=3.3923e-02), IS

  \addplot+[] coordinates{(2^7 ,    2.0883e-03)
                          (2^8 ,    1.4163e-03)
                          (2^9 ,    8.9528e-04)
                          (2^10,    5.7055e-04)
                          (2^11,    2.5678e-04)
                          (2^12,    2.1606e-04)
                          (2^13,    1.3923e-04)
                          (2^14,    1.3405e-04)
                          (2^15,    6.9282e-05)
                          (2^16,    5.0129e-05)
                          (2^17,    2.3499e-05)
                          (2^18,    1.3300e-05)
                         }; % \addlegendentry{TT-qIW}; % invcdf1 without rej, nPy=32, eps=0.1, IS

  \addplot+[purple,mark=triangle*,mark options={purple}] coordinates{(2^13-6000, 2.2878e-02)
                                (2^15-6000, 6.2463e-03)
                                (2^17-6000, 1.6843e-03)
                                (2^18-6000, 1.1349e-03)
                                (2^19-6000, 6.9849e-04)
                                }; % \addlegendentry{DRAM};

  \addplot+[black] coordinates{
                          (2^11, 4.9471e-03)
                          (2^13, 1.2222e-03)
                          (2^15, 4.8390e-04)
                          (2^17, 1.5904e-04)
                          (2^19, 6.5513e-05)
                         }; % \addlegendentry{QMC-rat};

%   \addplot+[purple,dashed] coordinates{                 % ttimes_dramrun
%                                 (2^13-6000, 1.1055e-02) % 1.4046e+03/16
%                                 (2^15-6000, 4.2002e-03) % 5.3359e+03/16
%                                 (2^17-6000, 2.0129e-03) % 2.1250e+04/16
%                                 (2^18-6000, 1.4949e-03) % 4.2615e+04/16
%                                 (2^19-6000, 1.1149e-03) % 8.5116e+04/16
%                                 }; % \addlegendentry{DRAMpre};

  \addplot+[green!50!black,mark options={green!50!black}] coordinates{
                                (2^13-6000, 1.4777e-02) % 343.9808
                                (2^15-6000, 7.0841e-03) % 1388.18
                                (2^17-6000, 2.8177e-03) % 5469.89
                                (2^18-6000, 1.7960e-03)
                                % (2^19-6000, )
                                }; % \addlegendentry{MALA};

  \addplot+[no marks, domain=2^11:2^15,solid,black] {(x^(-0.5))*5.0}  node[pos=0.5,anchor=west] {$0.5$};
  \addplot+[no marks, domain=2^11:2^15,solid,black] {(x^(-1.0))*5.0*2^5.5} node[pos=0.5,anchor=east] {$1.0$};

  \addplot+[no marks,dashed,black, domain=2^6:2^20] {1.1991e-04} node[pos=0.0,anchor=north west] {discr. error};
  \end{axis}
 \end{tikzpicture}
}
\resizebox{0.39\linewidth}{!}{
\begin{tikzpicture}
  \begin{axis}[%
  xmode=normal,
  ymode=log,
  xlabel=$\log_2 N$,
  ylabel={relative error for $P_{F>1.5}$},
  legend style={at={(0.01,0.01)},anchor=south west},
  x filter/.code={\pgfmathparse{log2(\pgfmathresult)}\pgfmathresult},
  ]
  \addplot+[] coordinates{(2^9 , 1.3600e-01)
                          (2^10, 1.0773e-01)
                          (2^11, 6.8493e-02)
                          (2^12, 4.8459e-02)
                          (2^13, 2.2761e-02)
                          (2^14, 1.3097e-02)
                          (2^15, 1.1295e-02)
                          (2^16, 9.3384e-03)
                          (2^17, 7.7678e-03)
                          (2^18, 4.9315e-03)
                          };  \addlegendentry{TT-MH};  % with tt_invcdf1, nPy=32, eps=0.1 (err_Pi=3.3923e-02)

%   \addplot+[] coordinates{
%                                         (128         + 128       ,  2.1000e-01)
%                                         (256         + 256       ,  1.7069e-01)
%                                         (5.1200e+02  + 6.8538e+02,  6.7946e-02)
%                                         (1.0240e+03  + 2.5694e+03,  4.4505e-02)
%                                         (2.0480e+03  + 1.0710e+04,  1.8543e-02)
%                                         (4.0960e+03  + 4.2686e+04,  1.3959e-02)
%                                         (8.1920e+03  + 1.7243e+05,  6.8125e-03)
%                                         (1.6384e+04  + 6.7867e+05,  2.7250e-03)
%                                                                     };  \addlegendentry{TT-qCV};  % -rr;  N1 = 3e-2*N0^2*eps

 % \addplot+[] coordinates{(2^9 ,   1.4466e-01)
  %                         (2^10,   1.0552e-01)
  %                         (2^11,   6.5112e-02)
  %                         (2^12,   4.3918e-02)
  %                         (2^13,   1.6791e-02)
  %                         (2^14,   1.2327e-02)
  %                         (2^15,   1.1233e-02)
  %                         (2^16,   8.4610e-03)
  %                         (2^17,   6.8942e-03)
  %                         (2^18,   4.2000e-03)
  %                         };  \addlegendentry{TTrIS};  % with tt_invcdf1, nPy=32, eps=0.1 (
  %                        err_Pi=3.3923e-02), IS

  \addplot+[] coordinates{(2^7 ,   1.4660e-01)
                          (2^8 ,   1.4487e-01)
                          (2^9 ,   5.0458e-02)
                          (2^10,   2.5822e-02)
                          (2^11,   2.0454e-02)
                          (2^12,   8.6385e-03)
                          (2^13,   7.4490e-03)
                          (2^14,   4.0763e-03)
                          (2^15,   2.7994e-03)
                          (2^16,   1.7197e-03)
                          (2^17,   9.0654e-04)
                          (2^18,   7.2459e-04)
                         };  \addlegendentry{TT-qIW}; % invcdf1 without rej, nPy=32, eps=0.1, IS

  \addplot+[purple,mark=triangle*,mark options={purple}] coordinates{(2^13-6000, 3.2192e-01)
                                (2^15-6000, 9.4396e-02)
                                (2^17-6000, 2.4265e-02)
                                (2^18-6000, 2.2791e-02)
                                (2^19-6000, 1.4292e-02)
                                };  \addlegendentry{DRAM};

  \addplot+[black] coordinates{
                          (2^11, 8.8555e-02)
                          (2^13, 4.3787e-02)
                          (2^15, 1.6524e-02)
                          (2^17, 8.4674e-03)
                          (2^19, 7.1032e-03)
                         };  \addlegendentry{QMC-rat};

%   \addplot+[purple,dashed] coordinates{
%                                 (2^13-6000,   1.7825e-01)
%                                 (2^15-6000,   5.9119e-02)
%                                 (2^17-6000,   4.2394e-02)
%                                 (2^18-6000,   2.7235e-02)
%                                 (2^19-6000,   1.6820e-02)
%                                 };  \addlegendentry{DRAMpre};

  \addplot+[green!50!black,mark options={green!50!black}] coordinates{
                                (2^13-6000, 1.7270e-01) % 343.9808
                                (2^15-6000, 8.9108e-02) % 1388.18
                                (2^17-6000, 3.6180e-02) % 5469.89
                                (2^18-6000, 2.6522e-02)
                                % (2^19-6000, )
                                };  \addlegendentry{MALA};

  \addplot+[no marks, domain=2^11:2^15,black,solid] {(x^(-0.5))*40.0}  node[pos=0.5,anchor=west] {$0.5$};
  \addplot+[no marks, domain=2^11:2^15,black,solid] {(x^(-1.0))*40.0*2^5.5} node[pos=0.5,anchor=east] {$1.0$};

  \addplot+[no marks,dashed,black, domain=2^6:2^20] {1.5970e-03};
  \end{axis}
 \end{tikzpicture}
}
\end{figure*}
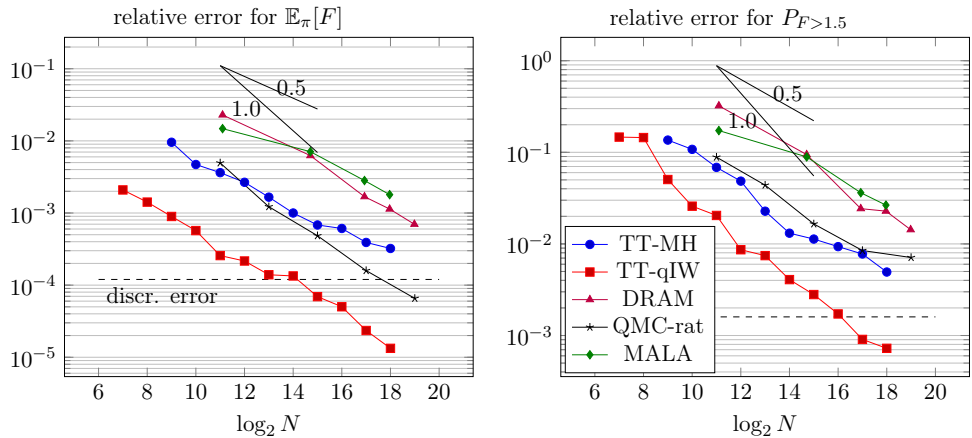

\begin{table}[t]
\centering
\caption{Inverse diffusion problem: rejection rates and IACTs for TT-MH and DRAM; $\tau_{\theta}$  and $\tau_{F}$ are the IACT for the parameter vector $\theta$ and for the QoI in \eqref{def_flux}, repectively.}
\label{tab:ff-chain}
 \begin{tabular}{c|ccc}
      & rejection rate      & $\tau_{\theta}$  & $\tau_{F}$  \\ \hline
TT-MH & 0.0853              & 1.1964           & 1.1903 \\  % invcdf1, nPy=32, eps=0.1
DRAM  & 0.73                & 84.0             & 29.7 \\
 \end{tabular}
\end{table}

In Figure \ref{fig:ff-err-N}, we present the relative errors in the quantities of interest versus the chain length~$N$ together with reference slopes.
For the expected value $\mathbb{E}_{\pi}[F]$ of the flux in Fig.~\ref{fig:ff-err-N} (left),
the QMC ratio estimator (QMC-rat) converges with a rate close to linear in $1/N$, so that it becomes competitive with the TT approaches for higher accuracies.
However, by far the most effective approach is TT-qIW,
where the TT approximation $\ttpi$ is used as an importance weight in a QMC ratio estimator.
Asymptotically, the convergence rate for TT-qIW is also $\mathcal{O}(N^{-1})$ for $\mathbb{E}_{\pi}[F]$ and the effectivity of the estimator is almost two orders of magnitude better than that of DRAM.
All the other TT-based approaches and DRAM converge, as expected, with the standard MC order $N^{-1/2}$.
For the non-smooth indicator function employed in $P_{F>1.5}$ in Fig.~\ref{fig:ff-err-N} (right),
the relative performance of the different approaches is similar, although the QMC-rat estimator now also converges with the MC rate of order $\mathcal{O}(N^{-1/2})$.
Somewhat surprisingly, the TT-qIW method seems to converge slightly better than $\mathcal{O}(N^{-1/2})$ also for $P_{F>1.5}$ and outperforms all other approaches by an order of magnitude.

The results in Fig.~\ref{fig:ff-err-N} are all computed for the same spatial resolution of the forward model.
In a practical inverse problem, for the best efficiency, all errors (due to truncation, discretization and sampling) are typically equilibrated.
Thus, it is useful to estimate the spatial discretization error.
We achieve this by computing the posterior expectations of the QoIs on three discretization grids (with TT-qIW and $N=2^{18}$) and by using these to estimate the error via Runge's rule.
The estimated error for $h=2^{-6}$ is plotted as a horizontal dashed line in Fig.~\ref{fig:ff-err-N}.
We see that with the TT-qIW method $N=2^{13}$ samples are sufficient to obtain a sampling error of the order of the discretization error for $\mathbb{E}_{\pi}[F]$,
while all other approaches require at least $N=2^{17}$ samples (up to $N > 2^{21}$ for DRAM).

In addition to DRAM, we also consider a version of the Metropolis adjusted Langevin (MALA) algorithm with adapted empirical covariance matrix as a preconditioner \cite{atchade-adaMALA-2006}.
However, the latter components of the gradient are rather small and give little information about the geometry.
This makes the MALA convergence comparable to that of DRAM.
Moreover, the computation of the gradient of $u_h(\cdot,\theta)$ (feeding into $\nabla \log\pi(\theta)$) is more expensive than the computation of the posterior alone.

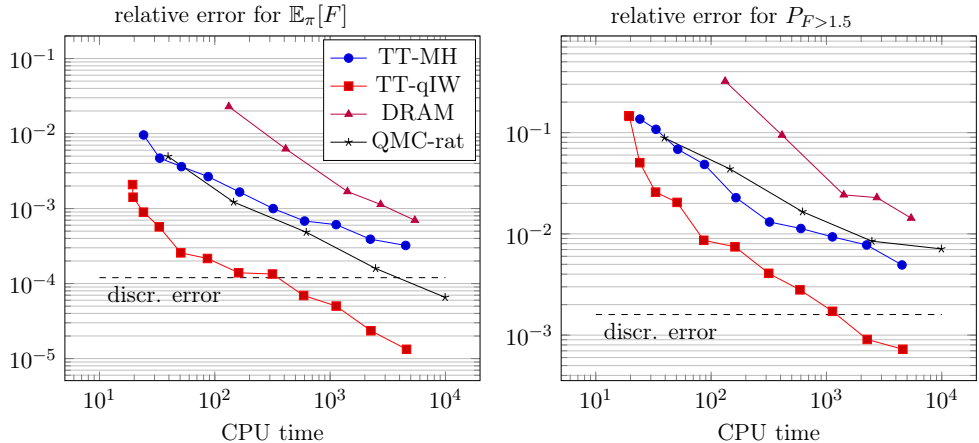
\begin{figure*}
\centering
\caption{Inverse diffusion problem: Relative errors in the mean flux (left) and in the exceedance probability (right) plotted against the total CPU times (sec.)}
\label{fig:ff-ttimes}
\resizebox{0.39\linewidth}{!}{%
\begin{tikzpicture}
  \begin{axis}[%
  xmode=log,
  ymode=log,
  ylabel={relative error for $\mathbb{E}_{\pi}[F]$},
  xlabel={CPU time},
  ymax=2e-1,
  legend style={at={(0.99,0.99)},anchor=north east},
  ]
  \addplot+[] coordinates{(1.1174e+01  + 3.6516e+00 +  1.1705e-01 +  9.1806e+00, 9.5709e-03)
                          (1.1195e+01  + 3.6814e+00 +  1.3391e-01 +  1.8255e+01, 4.7039e-03)
                          (1.1180e+01  + 3.6736e+00 +  1.7398e-01 +  3.6343e+01, 3.6251e-03)
                          (1.1146e+01  + 3.7007e+00 +  2.7548e-01 +  7.2463e+01, 2.6660e-03)
                          (1.1180e+01  + 3.6648e+00 +  4.8354e-01 +  1.4895e+02, 1.6596e-03)
                          (1.1184e+01  + 3.6509e+00 +  8.3720e-01 +  3.0519e+02, 1.0000e-03)
                          (1.1176e+01  + 3.6710e+00 +  1.4952e+00 +  5.8388e+02, 6.8202e-04)
                          (1.1232e+01  + 3.6767e+00 +  2.7833e+00 +  1.1095e+03, 6.1045e-04)
                          (1.1159e+01  + 3.6446e+00 +  5.4774e+00 +  2.2124e+03, 3.9070e-04)
                          (1.1201e+01  + 3.6698e+00 +  1.0954e+01 +  4.4924e+03, 3.2234e-04)
                          }; \addlegendentry{TT-MH};  % with tt_invcdf1, nPy=32, eps=0.1 (err_Pi=3.3923e-02)

  \addplot+[] coordinates{(1.3108e+01 +  3.7328e+00 +  1.0599e-01 +  2.4164e+00,    2.0883e-03)
                          (1.1149e+01 +  3.6634e+00 +  1.1022e-01 +  4.6348e+00,    1.4163e-03)
                          (1.1157e+01 +  3.7465e+00 +  1.1824e-01 +  9.0401e+00,    8.9528e-04)
                          (1.1164e+01 +  3.6610e+00 +  1.3293e-01 +  1.8042e+01,    5.7055e-04)
                          (1.1164e+01 +  3.6619e+00 +  1.7270e-01 +  3.5702e+01,    2.5678e-04)
                          (1.1101e+01 +  3.6627e+00 +  2.7383e-01 +  7.1311e+01,    2.1606e-04)
                          (1.1120e+01 +  3.6698e+00 +  4.8847e-01 +  1.4580e+02,    1.3923e-04)
                          (1.1140e+01 +  3.6747e+00 +  8.3033e-01 +  3.0016e+02,    1.3405e-04)
                          (1.1124e+01 +  3.6549e+00 +  1.4903e+00 +  5.7317e+02,    6.9282e-05)
                          (1.1156e+01 +  3.6635e+00 +  2.7722e+00 +  1.1105e+03,    5.0129e-05)
                          (1.1193e+01 +  3.6737e+00 +  5.4103e+00 +  2.2374e+03,    2.3499e-05)
                          (1.1142e+01 +  3.6673e+00 +  1.0694e+01 +  4.5731e+03,    1.3300e-05)
                         }; \addlegendentry{TT-qIW}; % invcdf1 without rej, nPy=32, eps=0.1, IS

  \addplot+[purple,mark=triangle*,mark options={purple}] coordinates{(1.3234e+02, 2.2878e-02)
                                (4.1180e+02, 6.2463e-03)
                                (1.4132e+03, 1.6843e-03)
                                (2.7390e+03, 1.1349e-03)
                                (5.4260e+03, 6.9849e-04)
                                }; \addlegendentry{DRAM};

%   \addplot+[green!50!black,mark options={green!50!black}] coordinates{
%                                 (343.980, 1.4777e-02) %
%                                 (1388.18, 7.0841e-03) %
%                                 (5469.89, 2.8177e-03) %
%                                 (1.0862e+04, 1.7960e-03)
%                                 % (2^19-6000, )
%                                 }; \addlegendentry{MALA};

  \addplot+[black] coordinates{
                          (3.9353e+01, 4.9471e-03)
                          (1.4544e+02, 1.2222e-03)
                          (6.2091e+02, 4.8390e-04)
                          (2.4775e+03, 1.5904e-04)
                          (9.9293e+03, 6.5513e-05)
                         }; \addlegendentry{QMC-rat};

    \addplot+[no marks,dashed,black, domain=1e1:1e4] {1.1991e-04} node[pos=0.0,anchor=north west] {discr. error};
  \end{axis}
\end{tikzpicture}
}
\resizebox{0.39\linewidth}{!}{%
\begin{tikzpicture}
  \begin{axis}[%
  xmode=log,
  ymode=log,
  ylabel={relative error for $P_{F>1.5}$},
  xlabel={CPU time},
  ymax=9e-1,
  legend style={at={(0.99,0.99)},anchor=north east},
  ]
  \addplot+[] coordinates{(1.1174e+01 +  3.6516e+00 +  1.1705e-01  + 9.1806e+00,  1.3600e-01)
                          (1.1195e+01 +  3.6814e+00 +  1.3391e-01  + 1.8255e+01,  1.0773e-01)
                          (1.1180e+01 +  3.6736e+00 +  1.7398e-01  + 3.6343e+01,  6.8493e-02)
                          (1.1146e+01 +  3.7007e+00 +  2.7548e-01  + 7.2463e+01,  4.8459e-02)
                          (1.1180e+01 +  3.6648e+00 +  4.8354e-01  + 1.4895e+02,  2.2761e-02)
                          (1.1184e+01 +  3.6509e+00 +  8.3720e-01  + 3.0519e+02,  1.3097e-02)
                          (1.1176e+01 +  3.6710e+00 +  1.4952e+00  + 5.8388e+02,  1.1295e-02)
                          (1.1232e+01 +  3.6767e+00 +  2.7833e+00  + 1.1095e+03,  9.3384e-03)
                          (1.1159e+01 +  3.6446e+00 +  5.4774e+00  + 2.2124e+03,  7.7678e-03)
                          (1.1201e+01 +  3.6698e+00 +  1.0954e+01  + 4.4924e+03,  4.9315e-03)
                          }; % \addlegendentry{TT-MH};  % with tt_invcdf1, nPy=32, eps=0.1 (err_Pi=3.3923e-02)

  \addplot+[] coordinates{(1.3108e+01 +  3.7328e+00 +  1.0599e-01 +  2.4164e+00,  1.4660e-01)
                          (1.1149e+01 +  3.6634e+00 +  1.1022e-01 +  4.6348e+00,  1.4487e-01)
                          (1.1157e+01 +  3.7465e+00 +  1.1824e-01 +  9.0401e+00,  5.0458e-02)
                          (1.1164e+01 +  3.6610e+00 +  1.3293e-01 +  1.8042e+01,  2.5822e-02)
                          (1.1164e+01 +  3.6619e+00 +  1.7270e-01 +  3.5702e+01,  2.0454e-02)
                          (1.1101e+01 +  3.6627e+00 +  2.7383e-01 +  7.1311e+01,  8.6385e-03)
                          (1.1120e+01 +  3.6698e+00 +  4.8847e-01 +  1.4580e+02,  7.4490e-03)
                          (1.1140e+01 +  3.6747e+00 +  8.3033e-01 +  3.0016e+02,  4.0763e-03)
                          (1.1124e+01 +  3.6549e+00 +  1.4903e+00 +  5.7317e+02,  2.7994e-03)
                          (1.1156e+01 +  3.6635e+00 +  2.7722e+00 +  1.1105e+03,  1.7197e-03)
                          (1.1193e+01 +  3.6737e+00 +  5.4103e+00 +  2.2374e+03,  9.0654e-04)
                          (1.1142e+01 +  3.6673e+00 +  1.0694e+01 +  4.5731e+03,  7.2459e-04)
                         }; % \addlegendentry{TT-qIW}; % invcdf1 without rej, nPy=32, eps=0.1, IS

  \addplot+[purple,mark=triangle*,mark options={purple}] coordinates{(1.3234e+02, 3.2192e-01)
                                (4.1180e+02, 9.4396e-02)
                                (1.4132e+03, 2.4265e-02)
                                (2.7390e+03, 2.2791e-02)
                                (5.4260e+03, 1.4292e-02)
                                }; % \addlegendentry{DRAM};

%   \addplot+[green!50!black,mark options={green!50!black}] coordinates{
%                                 (343.980, 1.7270e-01) %
%                                 (1388.18, 8.9108e-02) %
%                                 (5469.89, 3.6180e-02) %
%                                 (1.0862e+04, 2.6522e-02)
%                                 % (2^19-6000, )
%                                 };  % \addlegendentry{MALA};

  \addplot+[black] coordinates{(3.9353e+01, 8.8555e-02)
                               (1.4544e+02, 4.3787e-02)
                               (6.2091e+02, 1.6524e-02)
                               (2.4775e+03, 8.4674e-03)
                               (9.9293e+03, 7.1032e-03)
                         }; % \addlegendentry{QMC-rat};

  \addplot+[no marks,dashed,black, domain=1e1:1e4] {1.5970e-03}  node[pos=0.0,anchor=north west] {discr. error};
  \end{axis}
\end{tikzpicture}
}
\end{figure*}

In Fig.~\ref{fig:ff-ttimes} we compare the approaches in terms of total CPU time. The horizontal off-set for all the TT based methods is the time needed to build the TT approximation $\tilde\pi$. The error then initially drops rapidly. As soon as the number $N$ of samples is big enough, the set-up cost becomes negligible and the relative performance of all the approaches is very similar to that in Fig.~\ref{fig:ff-err-N}, since the computational time per sample is dominated by the PDE solve and all approaches that we are comparing evaluate $\pi$ for each sample.
It is possible to significantly reduce this sampling cost, if we do not evaluate the exact $\pi$ for each sample, e.g. by simply computing the expected value of the QoIs with respect to the approximate density $\ttpi$ using TT-CD and QMC quadrature. However, in that case the estimator will be biased and the amount of bias depends on the accuracy of the TT surrogate $\ttpi$. In that case, the total cost is dominated by the set-up cost (a more detailed study of the cost of the various stages of our TT approach is included in Fig. \ref{fig:ff-d} below.)

% The latter three stages add up to the \emph{proposal} time in the TT scheme.
% It is much smaller than the total time, which is dominated by solving $N$ PDEs in the Metropolis algorithm (see Fig. \ref{fig:ff-d}, right).
% Theoretically, we can speed up the TT approach significantly by removing the latter step.
% First, we can accept all proposed samples.
% This would require a higher accuracy of $\tilde \pi$, but it can still be achieved faster than the computation of the exact $\pi$ at $N$ points.
% Second, we can compute the quantities of interest from the TT approximation of the forward problem solution, instead of the actual PDE solution.
% Further analysis is needed however to estimate the errors introduced by this surrogate model.

%%%%%%%%%%%%%%%%%%%%%%%%%%%%%%%%%%%
%%%%%%% Discretisation error %%%%%%
% n = 128, TTqIW, d = 11, lvls=18, delta=3e-3
%                                                                            % C = (|u1-u2|/|u|) / (2^(-p*n1)*(1-2^(-p)))
%     Fl(2) =
%
%        1.159268785555499
%        0.085688589448499
%     Fl(3) =                                           % extrapolated error_F(lvl3) = 1.1991e-04, error_P  = 1.5970e-03
%
%        1.158851943543939
%        0.085280380096312
%      Fl(4) =                                        errors
%
%         1.158747005211523                                5.3955e-06
%         0.085201392636860                                8.3716e-04
%%%%%%%%%%%%%%%%%%%%%%%%%%%%%%%%%%%

In Fig.~\ref{fig:ff-Nneps}, we include a more detailed study of the influence of the TT parameters $n$ and $\delta$. As expected, a more accurate TT surrogate provides a better proposal/importance weight  and thus leads to a better performance, but it also leads to a higher set-up cost. So for lower accuracies, cruder approximations are better. However, the quality of the surrogate seems to be less important for Monte Carlo based approaches.
For the middle plot in Fig.~\ref{fig:ff-Nneps}, we used the importance weighting method described in Sec.~\ref{sec:isqmc} with random Monte Carlo samples (TT-rIW).
It converges with almost the same rate as TT-MH, which might be due to independence proposals.
The quality of the surrogate seems to be significantly more important for the QMC-based approaches, such as for TT-qIW (Fig.~\ref{fig:ff-Nneps}, right), since the mapped QMC samples carry the PDF approximation error.

Another thing we study in Fig.~\ref{fig:ff-Nneps} are the two-level versions of TT-MH and of importance weighting described in Section~\ref{sec:multi}.
While the variance reduction and the induced cost reduction are significant compared to the single-level quadrature in the case of i.i.d. seed points in Alg.~\ref{alg:samp} (both in TT-MH and TT-rIW),
the difference in the case of QMC seeds in TT-qIW is marginal.
This is because the rate of convergence of the QMC quadrature drops to $\mathcal{O}(N^{-1/2})$ when applied to the less smooth difference term in \eqref{eq:qcorr2}.
In contrast, the single-level QMC estimator \eqref{eq:is} converges with a noticeably higher rate.

\begin{figure*}
\centering
\caption{Inverse diffusion problem: Convergence of $\mathbb{E}_{\pi}[F]$ (solid lines) and $P_{F>1.5}$ (dashed lines) with TT-MH (left), TT-rIW (middle) and TT-qIW (right) for different choices of $n$ and $\delta$.}
\label{fig:ff-Nneps}
% y0=rmax, kickrank=0

% TT-MH
\resizebox{0.32\linewidth}{!}{%
\begin{tikzpicture}
\begin{axis}[%
  xmode=log,
  ymode=log,
  xlabel={CPU time (TT-MH)},
  xmin=9e0,xmax=2e4,
  ymin=3e-6,ymax=4e-1,
  legend style={at={(0.01,0.01)},anchor=south west},
]
\addplot+[] coordinates{
                        (   1.1092e+01 +  1.8923e+00 +  9.9443e-02 +  2.5576e+00,    1.5580e-02)
                        (   1.1078e+01 +  1.8918e+00 +  1.0097e-01 +  4.9537e+00,    1.4338e-02)
                        (   1.1108e+01 +  1.8974e+00 +  1.0406e-01 +  9.7414e+00,    1.2496e-02)
                        (   1.1118e+01 +  1.9014e+00 +  1.1395e-01 +  1.9557e+01,    6.4084e-03)
                        (   1.1149e+01 +  1.8983e+00 +  1.2924e-01 +  3.8957e+01,    5.7848e-03)
                        (   1.1065e+01 +  1.8849e+00 +  1.7099e-01 +  7.8205e+01,    3.9331e-03)
                        (   1.1092e+01 +  1.8880e+00 +  2.5307e-01 +  1.5217e+02,    1.8969e-03)
                        (   1.1087e+01 +  1.9061e+00 +  4.2599e-01 +  3.0737e+02,    1.2500e-03)
                        (   1.1078e+01 +  1.8979e+00 +  7.3324e-01 +  5.8370e+02,    1.0328e-03)
                        (   1.1048e+01 +  1.8970e+00 +  1.3848e+00 +  1.0994e+03,    8.5746e-04)
                        (   1.1081e+01 +  1.8979e+00 +  2.7622e+00 +  2.2033e+03,    5.1325e-04)
                        (   1.1112e+01 +  1.8924e+00 +  5.4399e+00 +  4.4212e+03,    3.7862e-04)
                       }; % \addlegendentry{$n,\delta=16,0.5$};
\addplot+[] coordinates{
                        (   1.1178e+01 +  3.6354e+00 +  1.0425e-01 +  2.4239e+00,   9.6214e-03)
                        (   1.1210e+01 +  3.6465e+00 +  1.0989e-01 +  4.6808e+00,   9.1990e-03)
                        (   1.1174e+01 +  3.6516e+00 +  1.1705e-01 +  9.1806e+00,   9.5709e-03)
                        (   1.1195e+01 +  3.6814e+00 +  1.3391e-01 +  1.8255e+01,   4.7039e-03)
                        (   1.1180e+01 +  3.6736e+00 +  1.7398e-01 +  3.6343e+01,   3.6251e-03)
                        (   1.1146e+01 +  3.7007e+00 +  2.7548e-01 +  7.2463e+01,   2.6660e-03)
                        (   1.1180e+01 +  3.6648e+00 +  4.8354e-01 +  1.4895e+02,   1.6596e-03)
                        (   1.1184e+01 +  3.6509e+00 +  8.3720e-01 +  3.0519e+02,   1.0000e-03)
                        (   1.1176e+01 +  3.6710e+00 +  1.4952e+00 +  5.8388e+02,   6.8202e-04)
                        (   1.1232e+01 +  3.6767e+00 +  2.7833e+00 +  1.1095e+03,   6.1045e-04)
                        (   1.1159e+01 +  3.6446e+00 +  5.4774e+00 +  2.2124e+03,   3.9070e-04)
                        (   1.1201e+01 +  3.6698e+00 +  1.0954e+01 +  4.4924e+03,   3.2234e-04)
                       }; % \addlegendentry{$n,\delta=32,0.1$};
\addplot+[mark=triangle*] coordinates{
                        (   1.2916e+01 +  1.6929e+01 +  1.7566e-01 +  2.4652e+00,   8.9385e-03)
                        (   1.1598e+01 +  1.6901e+01 +  1.8899e-01 +  4.7830e+00,   9.4374e-03)
                        (   1.1595e+01 +  1.6911e+01 +  2.2930e-01 +  9.2658e+00,   8.7064e-03)
                        (   1.1639e+01 +  1.6905e+01 +  3.0724e-01 +  1.8319e+01,   4.2334e-03)
                        (   1.1645e+01 +  1.6938e+01 +  4.8650e-01 +  3.6230e+01,   3.3129e-03)
                        (   1.1695e+01 +  1.6966e+01 +  8.8911e-01 +  7.2155e+01,   2.1345e-03)
                        (   1.1616e+01 +  1.6961e+01 +  1.6256e+00 +  1.3985e+02,   1.5281e-03)
                        (   1.1608e+01 +  1.7021e+01 +  3.5411e+00 +  2.7973e+02,   7.7447e-04)
                        (   1.1630e+01 +  1.7012e+01 +  6.9107e+00 +  5.5938e+02,   7.2758e-04)
                        (   1.1623e+01 +  1.6934e+01 +  1.4183e+01 +  1.1236e+03,   6.2144e-04)
                        (   1.1599e+01 +  1.6978e+01 +  2.8465e+01 +  2.2410e+03,   4.3162e-04)
                        (   1.1595e+01 +  1.7024e+01 +  5.6983e+01 +  4.5262e+03,   3.1168e-04)
                       }; % \addlegendentry{$n,\delta=64,0.01$};

\addplot+[mark=diamond*,green!50!black,mark options={green!50!black}] coordinates{
      (1.0918e+01 + 1.0080e+01 + 2.9669e-01 + 4.8327e+00, 2.2796e-03)
      (1.0906e+01 + 1.0081e+01 + 3.4031e-01 + 9.2710e+00, 2.1662e-03)
      (1.0956e+01 + 1.0029e+01 + 4.1132e-01 + 1.8148e+01, 1.5641e-03)
      (1.0921e+01 + 1.0144e+01 + 5.8944e-01 + 3.6238e+01, 8.6936e-04)
      (1.0920e+01 + 1.0094e+01 + 1.0583e+00 + 8.5692e+01, 6.5700e-04)
      (1.0931e+01 + 1.0094e+01 + 2.4480e+00 + 2.5666e+02, 3.3273e-04)
      (1.3078e+01 + 1.0160e+01 + 9.3746e+00 + 9.2151e+02, 1.6841e-04)
      (1.0929e+01 + 1.0103e+01 + 3.8110e+01 + 3.3849e+03, 6.8484e-05)
                       }; % 2lvl MH, n=64, delta=0.01, N1 = N0^2*0.03*rej

\pgfplotsset{cycle list shift=-4};

\addplot+[dashed] coordinates{
                        (   1.1092e+01 +  1.8923e+00 +  9.9443e-02 +  2.5576e+00,  2.2500e-01)
                        (   1.1078e+01 +  1.8918e+00 +  1.0097e-01 +  4.9537e+00,  2.1143e-01)
                        (   1.1108e+01 +  1.8974e+00 +  1.0406e-01 +  9.7414e+00,  2.0918e-01)
                        (   1.1118e+01 +  1.9014e+00 +  1.1395e-01 +  1.9557e+01,  1.1449e-01)
                        (   1.1149e+01 +  1.8983e+00 +  1.2924e-01 +  3.8957e+01,  1.1054e-01)
                        (   1.1065e+01 +  1.8849e+00 +  1.7099e-01 +  7.8205e+01,  7.4185e-02)
                        (   1.1092e+01 +  1.8880e+00 +  2.5307e-01 +  1.5217e+02,  3.1663e-02)
                        (   1.1087e+01 +  1.9061e+00 +  4.2599e-01 +  3.0737e+02,  2.1128e-02)
                        (   1.1078e+01 +  1.8979e+00 +  7.3324e-01 +  5.8370e+02,  1.4062e-02)
                        (   1.1048e+01 +  1.8970e+00 +  1.3848e+00 +  1.0994e+03,  1.0380e-02)
                        (   1.1081e+01 +  1.8979e+00 +  2.7622e+00 +  2.2033e+03,  8.6749e-03)
                        (   1.1112e+01 +  1.8924e+00 +  5.4399e+00 +  4.4212e+03,  8.3748e-03)
                       };
\addplot+[dashed] coordinates{
                        (   1.1178e+01 +  3.6354e+00 +  1.0425e-01 +  2.4239e+00,  1.6872e-01)
                        (   1.1210e+01 +  3.6465e+00 +  1.0989e-01 +  4.6808e+00,  1.7958e-01)
                        (   1.1174e+01 +  3.6516e+00 +  1.1705e-01 +  9.1806e+00,  1.3600e-01)
                        (   1.1195e+01 +  3.6814e+00 +  1.3391e-01 +  1.8255e+01,  1.0773e-01)
                        (   1.1180e+01 +  3.6736e+00 +  1.7398e-01 +  3.6343e+01,  6.8493e-02)
                        (   1.1146e+01 +  3.7007e+00 +  2.7548e-01 +  7.2463e+01,  4.8459e-02)
                        (   1.1180e+01 +  3.6648e+00 +  4.8354e-01 +  1.4895e+02,  2.2761e-02)
                        (   1.1184e+01 +  3.6509e+00 +  8.3720e-01 +  3.0519e+02,  1.3097e-02)
                        (   1.1176e+01 +  3.6710e+00 +  1.4952e+00 +  5.8388e+02,  1.1295e-02)
                        (   1.1232e+01 +  3.6767e+00 +  2.7833e+00 +  1.1095e+03,  9.3384e-03)
                        (   1.1159e+01 +  3.6446e+00 +  5.4774e+00 +  2.2124e+03,  7.7678e-03)
                        (   1.1201e+01 +  3.6698e+00 +  1.0954e+01 +  4.4924e+03,  4.9315e-03)
                       }; % 32,0.1
\addplot+[dashed,mark=triangle*] coordinates{
                        (   1.2916e+01 +  1.6929e+01 +  1.7566e-01 +  2.4652e+00,  1.6923e-01)
                        (   1.1598e+01 +  1.6901e+01 +  1.8899e-01 +  4.7830e+00,  1.5027e-01)
                        (   1.1595e+01 +  1.6911e+01 +  2.2930e-01 +  9.2658e+00,  1.3445e-01)
                        (   1.1639e+01 +  1.6905e+01 +  3.0724e-01 +  1.8319e+01,  1.0111e-01)
                        (   1.1645e+01 +  1.6938e+01 +  4.8650e-01 +  3.6230e+01,  6.2622e-02)
                        (   1.1695e+01 +  1.6966e+01 +  8.8911e-01 +  7.2155e+01,  4.0173e-02)
                        (   1.1616e+01 +  1.6961e+01 +  1.6256e+00 +  1.3985e+02,  2.0303e-02)
                        (   1.1608e+01 +  1.7021e+01 +  3.5411e+00 +  2.7973e+02,  1.1174e-02)
                        (   1.1630e+01 +  1.7012e+01 +  6.9107e+00 +  5.5938e+02,  1.1110e-02)
                        (   1.1623e+01 +  1.6934e+01 +  1.4183e+01 +  1.1236e+03,  9.0165e-03)
                        (   1.1599e+01 +  1.6978e+01 +  2.8465e+01 +  2.2410e+03,  7.9344e-03)
                        (   1.1595e+01 +  1.7024e+01 +  5.6983e+01 +  4.5262e+03,  3.8849e-03)
                       }; % 64,0.01

\addplot+[mark=diamond*,green!50!black,dashed,mark options={green!50!black}] coordinates{
      (1.0918e+01 + 1.0080e+01 + 2.9669e-01 + 4.8327e+00, 1.7919e-01)
      (1.0906e+01 + 1.0081e+01 + 3.4031e-01 + 9.2710e+00, 1.5418e-01)
      (1.0956e+01 + 1.0029e+01 + 4.1132e-01 + 1.8148e+01, 6.1111e-02)
      (1.0921e+01 + 1.0144e+01 + 5.8944e-01 + 3.6238e+01, 3.2115e-02)
      (1.0920e+01 + 1.0094e+01 + 1.0583e+00 + 8.5692e+01, 2.0153e-02)
      (1.0931e+01 + 1.0094e+01 + 2.4480e+00 + 2.5666e+02, 1.2688e-02)
      (1.3078e+01 + 1.0160e+01 + 9.3746e+00 + 9.2151e+02, 5.2103e-03)
      (1.0929e+01 + 1.0103e+01 + 3.8110e+01 + 3.3849e+03, 3.5913e-03)
                       }; % 2lvl MH, n=64,delta=0.01

\end{axis}
\end{tikzpicture}
}
% TTrIS
\resizebox{0.32\linewidth}{!}{%
\begin{tikzpicture}
\begin{axis}[%
  xmode=log,
  ymode=log,
  xlabel={CPU time (TT-rIW)},
  xmin=9e0,xmax=2e4,
  ymin=3e-6,ymax=4e-1,
  legend style={at={(0.99,0.99)},anchor=north east},
]

\addplot+[] coordinates{
   (1.1596e+01 +  1.3040e+00 +  1.0143e-01 +  2.5900e+00,   1.4433e-02)
   (1.0978e+01 +  1.3165e+00 +  1.0567e-01 +  4.9990e+00,   9.0124e-03)
   (1.0944e+01 +  1.3075e+00 +  1.0739e-01 +  9.7158e+00,   6.8487e-03)
   (1.0940e+01 +  1.3102e+00 +  1.1679e-01 +  1.9374e+01,   5.5550e-03)
   (1.0967e+01 +  1.3130e+00 +  1.3846e-01 +  3.9188e+01,   3.5438e-03)
   (1.0984e+01 +  1.3053e+00 +  1.7499e-01 +  7.7761e+01,   2.0332e-03)
   (1.0970e+01 +  1.3051e+00 +  2.5581e-01 +  1.4995e+02,   2.5539e-03)
   (1.0968e+01 +  1.2989e+00 +  4.2650e-01 +  2.8336e+02,   1.0972e-03)
   (1.0981e+01 +  1.3004e+00 +  7.3473e-01 +  5.5304e+02,   6.6201e-04)
   (1.0977e+01 +  1.3132e+00 +  1.4180e+00 +  1.1022e+03,   5.8126e-04)
   (1.0963e+01 +  1.3039e+00 +  2.7224e+00 +  2.2175e+03,   4.0715e-04)
   (1.0961e+01 +  1.3040e+00 +  5.4674e+00 +  4.3890e+03,   3.0554e-04)
                       }; % \addlegendentry{$n,\delta=16,0.5$};
\addplot+[] coordinates{
   (1.0996e+01 +  2.1736e+00 +  1.0730e-01 +  2.3506e+00,   1.3177e-02)
   (1.0866e+01 +  2.2130e+00 +  1.1102e-01 +  4.5553e+00,   8.9679e-03)
   (1.0955e+01 +  2.2005e+00 +  1.1964e-01 +  8.8989e+00,   6.4709e-03)
   (1.0960e+01 +  2.1765e+00 +  1.3463e-01 +  1.7711e+01,   4.4246e-03)
   (1.0986e+01 +  2.1812e+00 +  1.7536e-01 +  3.5426e+01,   3.2054e-03)
   (1.0979e+01 +  2.1747e+00 +  2.7064e-01 +  6.9565e+01,   2.2771e-03)
   (1.0948e+01 +  2.1681e+00 +  4.8264e-01 +  1.4201e+02,   2.1780e-03)
   (1.0995e+01 +  2.1676e+00 +  8.2630e-01 +  2.8594e+02,   9.8302e-04)
   (1.0958e+01 +  2.1719e+00 +  1.5136e+00 +  5.6467e+02,   7.6102e-04)
   (1.0936e+01 +  2.1678e+00 +  2.8306e+00 +  1.1162e+03,   5.2725e-04)
   (1.0927e+01 +  2.1749e+00 +  5.5182e+00 +  2.2249e+03,   4.3630e-04)
   (1.0955e+01 +  2.1665e+00 +  1.0978e+01 +  4.4527e+03,   2.7877e-04)
                       }; % \addlegendentry{$n,\delta=32,0.1$};
\addplot+[mark=triangle*] coordinates{
   (1.2039e+01 +  8.3904e+00 +  1.9691e-01 +  2.4520e+00,   1.3529e-02)
   (1.0922e+01 +  8.3567e+00 +  2.1610e-01 +  4.6289e+00,   8.8705e-03)
   (1.0951e+01 +  8.3983e+00 +  2.5193e-01 +  8.9818e+00,   6.5883e-03)
   (1.0957e+01 +  8.3642e+00 +  3.3455e-01 +  1.7924e+01,   4.4698e-03)
   (1.0956e+01 +  8.3615e+00 +  5.2822e-01 +  3.5484e+01,   3.1242e-03)
   (1.0962e+01 +  8.3928e+00 +  9.3271e-01 +  7.0100e+01,   2.2988e-03)
   (1.0946e+01 +  8.4007e+00 +  1.7221e+00 +  1.4023e+02,   2.0824e-03)
   (1.0924e+01 +  8.3698e+00 +  3.5282e+00 +  2.8015e+02,   9.9872e-04)
   (1.1005e+01 +  8.4021e+00 +  6.9631e+00 +  5.5927e+02,   7.5938e-04)
   (1.0920e+01 +  8.4185e+00 +  1.4502e+01 +  1.1242e+03,   4.8266e-04)
   (1.0937e+01 +  8.3951e+00 +  2.8854e+01 +  2.2436e+03,   4.6413e-04)
   (1.0908e+01 +  8.4038e+00 +  5.7335e+01 +  4.4778e+03,   2.8265e-04)
                       }; % \addlegendentry{$n,\delta=64,0.01$};

\addplot+[mark=diamond*,green!50!black,mark options={green!50!black}] coordinates{
(1.6449e+01 + 9.0068e+00 + 1.3267e+00 + 2.3397e+00, 7.7298e-04)
(1.6105e+01 + 9.0494e+00 + 2.5245e+00 + 4.5321e+00, 4.1851e-04)
(1.5845e+01 + 9.0014e+00 + 4.7859e+00 + 8.9063e+00, 2.6741e-04)
(1.5798e+01 + 9.0446e+00 + 6.8114e+00 + 1.8130e+01, 2.3847e-04)
(1.5847e+01 + 9.0457e+00 + 1.5255e+01 + 3.5327e+01, 1.1258e-04)
(1.5868e+01 + 9.0757e+00 + 2.8207e+01 + 7.0066e+01, 9.9860e-05)
(1.5966e+01 + 9.0142e+00 + 5.5131e+01 + 1.4159e+02, 7.9177e-05)
(1.8470e+01 + 9.0230e+00 + 1.0796e+02 + 2.8168e+02, 4.6215e-05)
(1.5812e+01 + 9.0422e+00 + 1.9766e+02 + 5.7011e+02, 2.9795e-05)
(1.5903e+01 + 9.0411e+00 + 4.9352e+02 + 1.1905e+03, 3.5819e-05)
(1.5865e+01 + 9.0144e+00 + 9.3106e+02 + 2.2754e+03, 2.5596e-05)
(1.5851e+01 + 9.0423e+00 + 1.8938e+03 + 4.5267e+03, 1.3862e-05)
                       }; % 2lvl rIW, n=64,delta=0.01

\pgfplotsset{cycle list shift=-4};

\addplot+[dashed] coordinates{
   (1.1596e+01 +  1.3040e+00 +  1.0143e-01 +  2.5900e+00,   2.2723e-01)
   (1.0978e+01 +  1.3165e+00 +  1.0567e-01 +  4.9990e+00,   8.1047e-02)
   (1.0944e+01 +  1.3075e+00 +  1.0739e-01 +  9.7158e+00,   8.8806e-02)
   (1.0940e+01 +  1.3102e+00 +  1.1679e-01 +  1.9374e+01,   8.0180e-02)
   (1.0967e+01 +  1.3130e+00 +  1.3846e-01 +  3.9188e+01,   7.5268e-02)
   (1.0984e+01 +  1.3053e+00 +  1.7499e-01 +  7.7761e+01,   3.1495e-02)
   (1.0970e+01 +  1.3051e+00 +  2.5581e-01 +  1.4995e+02,   3.2431e-02)
   (1.0968e+01 +  1.2989e+00 +  4.2650e-01 +  2.8336e+02,   2.0301e-02)
   (1.0981e+01 +  1.3004e+00 +  7.3473e-01 +  5.5304e+02,   1.6752e-02)
   (1.0977e+01 +  1.3132e+00 +  1.4180e+00 +  1.1022e+03,   7.8938e-03)
   (1.0963e+01 +  1.3039e+00 +  2.7224e+00 +  2.2175e+03,   8.4323e-03)
   (1.0961e+01 +  1.3040e+00 +  5.4674e+00 +  4.3890e+03,   5.2723e-03)
                       }; % 16,0.5, IS
\addplot+[dashed] coordinates{
   (1.0996e+01 +  2.1736e+00 +  1.0730e-01 +  2.3506e+00,   2.0028e-01)
   (1.0866e+01 +  2.2130e+00 +  1.1102e-01 +  4.5553e+00,   9.6390e-02)
   (1.0955e+01 +  2.2005e+00 +  1.1964e-01 +  8.8989e+00,   8.4137e-02)
   (1.0960e+01 +  2.1765e+00 +  1.3463e-01 +  1.7711e+01,   7.3274e-02)
   (1.0986e+01 +  2.1812e+00 +  1.7536e-01 +  3.5426e+01,   6.2165e-02)
   (1.0979e+01 +  2.1747e+00 +  2.7064e-01 +  6.9565e+01,   3.4422e-02)
   (1.0948e+01 +  2.1681e+00 +  4.8264e-01 +  1.4201e+02,   3.4087e-02)
   (1.0995e+01 +  2.1676e+00 +  8.2630e-01 +  2.8594e+02,   2.1253e-02)
   (1.0958e+01 +  2.1719e+00 +  1.5136e+00 +  5.6467e+02,   1.4972e-02)
   (1.0936e+01 +  2.1678e+00 +  2.8306e+00 +  1.1162e+03,   6.9168e-03)
   (1.0927e+01 +  2.1749e+00 +  5.5182e+00 +  2.2249e+03,   9.3306e-03)
   (1.0955e+01 +  2.1665e+00 +  1.0978e+01 +  4.4527e+03,   3.9741e-03)
                       }; % 32,0.1, IS
\addplot+[dashed,mark=triangle*] coordinates{
   (1.2039e+01 +  8.3904e+00 +  1.9691e-01 +  2.4520e+00,   2.1621e-01 )
   (1.0922e+01 +  8.3567e+00 +  2.1610e-01 +  4.6289e+00,   9.1258e-02 )
   (1.0951e+01 +  8.3983e+00 +  2.5193e-01 +  8.9818e+00,   8.5684e-02 )
   (1.0957e+01 +  8.3642e+00 +  3.3455e-01 +  1.7924e+01,   7.3281e-02 )
   (1.0956e+01 +  8.3615e+00 +  5.2822e-01 +  3.5484e+01,   6.4461e-02 )
   (1.0962e+01 +  8.3928e+00 +  9.3271e-01 +  7.0100e+01,   3.2548e-02 )
   (1.0946e+01 +  8.4007e+00 +  1.7221e+00 +  1.4023e+02,   3.3597e-02 )
   (1.0924e+01 +  8.3698e+00 +  3.5282e+00 +  2.8015e+02,   2.0296e-02 )
   (1.1005e+01 +  8.4021e+00 +  6.9631e+00 +  5.5927e+02,   1.3556e-02 )
   (1.0920e+01 +  8.4185e+00 +  1.4502e+01 +  1.1242e+03,   6.5204e-03 )
   (1.0937e+01 +  8.3951e+00 +  2.8854e+01 +  2.2436e+03,   9.1551e-03 )
   (1.0908e+01 +  8.4038e+00 +  5.7335e+01 +  4.4778e+03,   4.0360e-03 )
                       }; % 64,0.01, IS

\addplot+[mark=diamond*,green!50!black,dashed,mark options={green!50!black}] coordinates{
(1.6449e+01 + 9.0068e+00 + 1.3267e+00 + 2.3397e+00, 2.4298e-02)
(1.6105e+01 + 9.0494e+00 + 2.5245e+00 + 4.5321e+00, 1.6691e-02)
(1.5845e+01 + 9.0014e+00 + 4.7859e+00 + 8.9063e+00, 7.4323e-03)
(1.5798e+01 + 9.0446e+00 + 6.8114e+00 + 1.8130e+01, 4.9721e-03)
(1.5847e+01 + 9.0457e+00 + 1.5255e+01 + 3.5327e+01, 5.5693e-03)
(1.5868e+01 + 9.0757e+00 + 2.8207e+01 + 7.0066e+01, 2.6363e-03)
(1.5966e+01 + 9.0142e+00 + 5.5131e+01 + 1.4159e+02, 2.0139e-03)
(1.8470e+01 + 9.0230e+00 + 1.0796e+02 + 2.8168e+02, 1.4818e-03)
(1.5812e+01 + 9.0422e+00 + 1.9766e+02 + 5.7011e+02, 1.1530e-03)
(1.5903e+01 + 9.0411e+00 + 4.9352e+02 + 1.1905e+03, 6.2046e-04)
(1.5865e+01 + 9.0144e+00 + 9.3106e+02 + 2.2754e+03, 5.8082e-04)
(1.5851e+01 + 9.0423e+00 + 1.8938e+03 + 4.5267e+03, 3.2155e-04)
                       }; % 2lvl rIW, n=64,delta=0.01

\end{axis}
\end{tikzpicture}
}
% TTqIW
\resizebox{0.32\linewidth}{!}{%
\begin{tikzpicture}
\begin{axis}[%
  xmode=log,
  ymode=log,
  xlabel={CPU time (TT-qIW)},
  xmin=9e0,xmax=2e4,
  ymin=3e-6,ymax=4e-1,
  legend style={at={(0.99,0.99)},anchor=north east},
]

\addplot+[] coordinates{
                        (   1.1139e+01 +  1.8894e+00 +  9.9413e-02 +  2.5460e+00,     4.0220e-03)
                        (   1.1127e+01 +  1.8995e+00 +  1.0221e-01 +  4.8957e+00,     4.0692e-03)
                        (   1.1093e+01 +  1.8835e+00 +  1.0416e-01 +  9.6264e+00,     2.6611e-03)
                        (   1.1022e+01 +  1.8926e+00 +  1.1144e-01 +  1.9284e+01,     2.2489e-03)
                        (   1.1119e+01 +  1.8880e+00 +  1.2837e-01 +  3.8701e+01,     1.0558e-03)
                        (   1.1066e+01 +  2.1184e+00 +  1.7061e-01 +  7.8135e+01,     5.4976e-04)
                        (   1.1143e+01 +  1.8951e+00 +  2.4859e-01 +  1.5355e+02,     3.8209e-04)
                        (   1.1093e+01 +  1.8900e+00 +  4.2120e-01 +  2.9470e+02,     2.9300e-04)
                        (   1.1122e+01 +  1.8963e+00 +  7.1755e-01 +  5.8524e+02,     1.3049e-04)
                        (   1.1094e+01 +  1.9090e+00 +  1.3568e+00 +  1.1007e+03,     9.7475e-05)
                        (   1.1105e+01 +  1.8958e+00 +  2.7336e+00 +  2.2191e+03,     3.8946e-05)
                        (   1.1109e+01 +  1.9006e+00 +  5.1300e+00 +  4.4526e+03,     4.1083e-05)
                       }; \addlegendentry{$n,\delta=16,0.5$};
\addplot+[] coordinates{
                        (1.3108e+01 +  3.7328e+00 +  1.0599e-01 +  2.4164e+00,    2.0883e-03)
                        (1.1149e+01 +  3.6634e+00 +  1.1022e-01 +  4.6348e+00,    1.4163e-03)
                        (1.1157e+01 +  3.7465e+00 +  1.1824e-01 +  9.0401e+00,    8.9528e-04)
                        (1.1164e+01 +  3.6610e+00 +  1.3293e-01 +  1.8042e+01,    5.7055e-04)
                        (1.1164e+01 +  3.6619e+00 +  1.7270e-01 +  3.5702e+01,    2.5678e-04)
                        (1.1101e+01 +  3.6627e+00 +  2.7383e-01 +  7.1311e+01,    2.1606e-04)
                        (1.1120e+01 +  3.6698e+00 +  4.8847e-01 +  1.4580e+02,    1.3923e-04)
                        (1.1140e+01 +  3.6747e+00 +  8.3033e-01 +  3.0016e+02,    1.3405e-04)
                        (1.1124e+01 +  3.6549e+00 +  1.4903e+00 +  5.7317e+02,    6.9282e-05)
                        (1.1156e+01 +  3.6635e+00 +  2.7722e+00 +  1.1105e+03,    5.0129e-05)
                        (1.1193e+01 +  3.6737e+00 +  5.4103e+00 +  2.2374e+03,    2.3499e-05)
                        (1.1142e+01 +  3.6673e+00 +  1.0694e+01 +  4.5731e+03,    1.3300e-05)
                       }; \addlegendentry{$n,\delta=32,0.1$};
% \addplot+[mark=triangle*] coordinates{
%                         (   1.1575e+01 +  1.6980e+01 +  1.8935e-01 +  2.4379e+00,   1.0806e-03)
%                         (   1.1578e+01 +  1.6937e+01 +  2.1474e-01 +  4.7114e+00,   9.5724e-04)
%                         (   1.1603e+01 +  1.6917e+01 +  2.4934e-01 +  9.1546e+00,   2.6593e-04)
%                         (   1.1544e+01 +  1.6924e+01 +  3.2435e-01 +  1.8038e+01,   1.8621e-04) % N=2^10
%                         (   1.1540e+01 +  1.6910e+01 +  5.4302e-01 +  3.5744e+01,   9.0149e-05)
%                         (   1.1560e+01 +  1.6928e+01 +  9.4739e-01 +  7.0634e+01,   9.8062e-05)
%                         (   1.1583e+01 +  1.6943e+01 +  1.7036e+00 +  1.4028e+02,   4.8885e-05)
%                         (   1.1546e+01 +  1.6915e+01 +  3.7079e+00 +  2.8142e+02,   5.1603e-05)
%                         (   1.1569e+01 +  1.6844e+01 +  7.2830e+00 +  5.5979e+02,   3.0064e-05)
%                         (   1.1593e+01 +  1.6917e+01 +  1.4760e+01 +  1.1205e+03,   2.0078e-05)
%                         (   1.1560e+01 +  1.6930e+01 +  3.0183e+01 +  2.2467e+03,   1.3158e-05)
%                         (   1.1595e+01 +  1.6910e+01 +  5.7640e+01 +  4.4877e+03,   6.9561e-06)
%                        }; \addlegendentry{$n,\delta=64,0.01$};

\addplot+[mark=triangle*] coordinates{
                        (   1.1498e+01 +  1.8207e+01 +  2.0855e-01 +  2.4550e+00,   1.0695e-03)
                        (   1.0937e+01 +  1.8293e+01 +  2.2510e-01 +  4.6327e+00,   9.3461e-04)
                        (   1.0929e+01 +  1.8264e+01 +  2.6736e-01 +  9.0304e+00,   4.1979e-04)
                        (   1.0948e+01 +  1.8207e+01 +  3.5051e-01 +  1.7855e+01,   1.2556e-04) % N=2^10
                        (   1.1002e+01 +  1.8215e+01 +  5.6038e-01 +  3.5449e+01,   8.1101e-05)
                        (   1.0899e+01 +  1.8386e+01 +  9.8271e-01 +  7.0361e+01,   7.7591e-05)
                        (   1.0901e+01 +  1.8285e+01 +  1.8109e+00 +  1.4012e+02,   3.8038e-05)
                        (   1.0949e+01 +  1.8175e+01 +  4.1103e+00 +  2.8089e+02,   3.7592e-05)
                        (   1.0941e+01 +  1.8451e+01 +  7.9474e+00 +  5.6057e+02,   1.6753e-05)
                        (   1.2575e+01 +  1.8064e+01 +  1.4807e+01 +  1.1220e+03,   1.4232e-05)
                        (   1.2690e+01 +  1.8030e+01 +  3.0562e+01 +  2.2453e+03,   1.2002e-05)
                        (   1.0985e+01 +  1.8228e+01 +  5.8713e+01 +  4.4962e+03,   9.0773e-06)
                       }; \addlegendentry{$n,\delta=64,0.01$};

\addplot+[mark=diamond*,green!50!black] coordinates{
                        (      1.5999e+01 +  1.8274e+01 +  3.3021e-01 +  2.5326e+00,   6.1465e-04)
                        (      1.5947e+01 +  1.8224e+01 +  6.9381e-01 +  4.8696e+00,   3.5152e-04)
                        (      1.5879e+01 +  1.8287e+01 +  1.4246e+00 +  9.4671e+00,   2.1764e-04)
                        (      1.6009e+01 +  1.8350e+01 +  2.7855e+00 +  1.8853e+01,   1.4733e-04) % N=2^10
                        (      1.5927e+01 +  1.8240e+01 +  5.7703e+00 +  3.7494e+01,   1.1814e-04)
                        (      1.5906e+01 +  1.8226e+01 +  1.1837e+01 +  7.4276e+01,   6.1204e-05)
                        (      1.5957e+01 +  1.8172e+01 +  2.2959e+01 +  1.5105e+02,   5.1043e-05)
                        (      1.5939e+01 +  1.8258e+01 +  4.4318e+01 +  2.9860e+02,   3.9600e-05)
                        (      1.5904e+01 +  1.8289e+01 +  8.5183e+01 +  5.9212e+02,   1.2592e-05)
                        (      1.5936e+01 +  1.8212e+01 +  1.9183e+02 +  1.1868e+03,   9.0572e-06)
                        (      1.5932e+01 +  1.8235e+01 +  3.7807e+02 +  2.3888e+03,   8.2076e-06)
                        (      1.5982e+01 +  1.8058e+01 +  7.4154e+02 +  4.7820e+03,   5.1881e-06)
%                         (      1.7517e+01 +  1.8020e+01 +  1.4973e+03 +  9.5932e+03,   4.4400e-06)
                       }; \addlegendentry{$64,0.01$, 2L};

\pgfplotsset{cycle list shift=-4};

\addplot+[dashed] coordinates{
                        (   1.1139e+01 +  1.8894e+00 +  9.9413e-02 +  2.5460e+00,     2.2967e-01)
                        (   1.1127e+01 +  1.8995e+00 +  1.0221e-01 +  4.8957e+00,     1.7451e-01)
                        (   1.1093e+01 +  1.8835e+00 +  1.0416e-01 +  9.6264e+00,     9.2184e-02)
                        (   1.1022e+01 +  1.8926e+00 +  1.1144e-01 +  1.9284e+01,     7.0071e-02)
                        (   1.1119e+01 +  1.8880e+00 +  1.2837e-01 +  3.8701e+01,     2.5913e-02)
                        (   1.1066e+01 +  2.1184e+00 +  1.7061e-01 +  7.8135e+01,     1.5320e-02)
                        (   1.1143e+01 +  1.8951e+00 +  2.4859e-01 +  1.5355e+02,     1.1848e-02)
                        (   1.1093e+01 +  1.8900e+00 +  4.2120e-01 +  2.9470e+02,     7.9577e-03)
                        (   1.1122e+01 +  1.8963e+00 +  7.1755e-01 +  5.8524e+02,     5.1347e-03)
                        (   1.1094e+01 +  1.9090e+00 +  1.3568e+00 +  1.1007e+03,     4.3689e-03)
                        (   1.1105e+01 +  1.8958e+00 +  2.7336e+00 +  2.2191e+03,     1.8541e-03)
                        (   1.1109e+01 +  1.9006e+00 +  5.1300e+00 +  4.4526e+03,     1.0144e-03)
                       }; % 16,0.5, IS
\addplot+[dashed] coordinates{
                        (1.3108e+01 +  3.7328e+00 +  1.0599e-01 +  2.4164e+00,   1.4660e-01)
                        (1.1149e+01 +  3.6634e+00 +  1.1022e-01 +  4.6348e+00,   1.4487e-01)
                        (1.1157e+01 +  3.7465e+00 +  1.1824e-01 +  9.0401e+00,   5.0458e-02)
                        (1.1164e+01 +  3.6610e+00 +  1.3293e-01 +  1.8042e+01,   2.5822e-02)
                        (1.1164e+01 +  3.6619e+00 +  1.7270e-01 +  3.5702e+01,   2.0454e-02)
                        (1.1101e+01 +  3.6627e+00 +  2.7383e-01 +  7.1311e+01,   8.6385e-03)
                        (1.1120e+01 +  3.6698e+00 +  4.8847e-01 +  1.4580e+02,   7.4490e-03)
                        (1.1140e+01 +  3.6747e+00 +  8.3033e-01 +  3.0016e+02,   4.0763e-03)
                        (1.1124e+01 +  3.6549e+00 +  1.4903e+00 +  5.7317e+02,   2.7994e-03)
                        (1.1156e+01 +  3.6635e+00 +  2.7722e+00 +  1.1105e+03,   1.7197e-03)
                        (1.1193e+01 +  3.6737e+00 +  5.4103e+00 +  2.2374e+03,   9.0654e-04)
                        (1.1142e+01 +  3.6673e+00 +  1.0694e+01 +  4.5731e+03,   7.2459e-04)
                       }; % 32,0.1, IS
% \addplot+[dashed,mark=triangle*] coordinates{
%                         (   1.1575e+01 +  1.6980e+01 +  1.8935e-01 +  2.4379e+00,   1.2983e-01)
%                         (   1.1578e+01 +  1.6937e+01 +  2.1474e-01 +  4.7114e+00,   1.4089e-01)
%                         (   1.1603e+01 +  1.6917e+01 +  2.4934e-01 +  9.1546e+00,   4.8005e-02)
%                         (   1.1544e+01 +  1.6924e+01 +  3.2435e-01 +  1.8038e+01,   2.2668e-02) % N=2^10
%                         (   1.1540e+01 +  1.6910e+01 +  5.4302e-01 +  3.5744e+01,   1.7407e-02)
%                         (   1.1560e+01 +  1.6928e+01 +  9.4739e-01 +  7.0634e+01,   8.4013e-03)
%                         (   1.1583e+01 +  1.6943e+01 +  1.7036e+00 +  1.4028e+02,   4.5137e-03)
%                         (   1.1546e+01 +  1.6915e+01 +  3.7079e+00 +  2.8142e+02,   2.5685e-03)
%                         (   1.1569e+01 +  1.6844e+01 +  7.2830e+00 +  5.5979e+02,   1.7726e-03)
%                         (   1.1593e+01 +  1.6917e+01 +  1.4760e+01 +  1.1205e+03,   9.1439e-04)
%                         (   1.1560e+01 +  1.6930e+01 +  3.0183e+01 +  2.2467e+03,   7.4855e-04)
%                         (   1.1595e+01 +  1.6910e+01 +  5.7640e+01 +  4.4877e+03,   5.1602e-04)
%                        }; % 64,0.01, IS

\addplot+[dashed,mark=triangle*] coordinates{
                        (   1.1498e+01 +  1.8207e+01 +  2.0855e-01 +  2.4550e+00,   1.4630e-01)
                        (   1.0937e+01 +  1.8293e+01 +  2.2510e-01 +  4.6327e+00,   1.5113e-01)
                        (   1.0929e+01 +  1.8264e+01 +  2.6736e-01 +  9.0304e+00,   6.3555e-02)
                        (   1.0948e+01 +  1.8207e+01 +  3.5051e-01 +  1.7855e+01,   2.8396e-02) % N=2^10
                        (   1.1002e+01 +  1.8215e+01 +  5.6038e-01 +  3.5449e+01,   2.0650e-02)
                        (   1.0899e+01 +  1.8386e+01 +  9.8271e-01 +  7.0361e+01,   1.1737e-02)
                        (   1.0901e+01 +  1.8285e+01 +  1.8109e+00 +  1.4012e+02,   4.2961e-03)
                        (   1.0949e+01 +  1.8175e+01 +  4.1103e+00 +  2.8089e+02,   2.7076e-03)
                        (   1.0941e+01 +  1.8451e+01 +  7.9474e+00 +  5.6057e+02,   1.5543e-03)
                        (   1.2575e+01 +  1.8064e+01 +  1.4807e+01 +  1.1220e+03,   1.0456e-03)
                        (   1.2690e+01 +  1.8030e+01 +  3.0562e+01 +  2.2453e+03,   9.6637e-04)
                        (   1.0985e+01 +  1.8228e+01 +  5.8713e+01 +  4.4962e+03,   5.7682e-04)
                       }; % 64,0.01, IS

\addplot+[mark=diamond*,green!50!black,dashed,mark options={green!50!black}] coordinates{
                        (      1.5999e+01 +  1.8274e+01 +  3.3021e-01 +  2.5326e+00,   3.0726e-02)
                        (      1.5947e+01 +  1.8224e+01 +  6.9381e-01 +  4.8696e+00,   2.3797e-02)
                        (      1.5879e+01 +  1.8287e+01 +  1.4246e+00 +  9.4671e+00,   1.6271e-02)
                        (      1.6009e+01 +  1.8350e+01 +  2.7855e+00 +  1.8853e+01,   4.7887e-03) % N=2^10
                        (      1.5927e+01 +  1.8240e+01 +  5.7703e+00 +  3.7494e+01,   4.4841e-03)
                        (      1.5906e+01 +  1.8226e+01 +  1.1837e+01 +  7.4276e+01,   3.4781e-03)
                        (      1.5957e+01 +  1.8172e+01 +  2.2959e+01 +  1.5105e+02,   2.0166e-03)
                        (      1.5939e+01 +  1.8258e+01 +  4.4318e+01 +  2.9860e+02,   2.2898e-03)
                        (      1.5904e+01 +  1.8289e+01 +  8.5183e+01 +  5.9212e+02,   1.3255e-03)
                        (      1.5936e+01 +  1.8212e+01 +  1.9183e+02 +  1.1868e+03,   6.0984e-04)
                        (      1.5932e+01 +  1.8235e+01 +  3.7807e+02 +  2.3888e+03,   3.3792e-04)
                        (      1.5982e+01 +  1.8058e+01 +  7.4154e+02 +  4.7820e+03,   1.7269e-04)
%                         (      1.7517e+01 +  1.8020e+01 +  1.4973e+03 +  9.5932e+03,   1.7559e-04)
                       }; % 61,0.01,3L

\end{axis}
\end{tikzpicture}
}
\end{figure*}
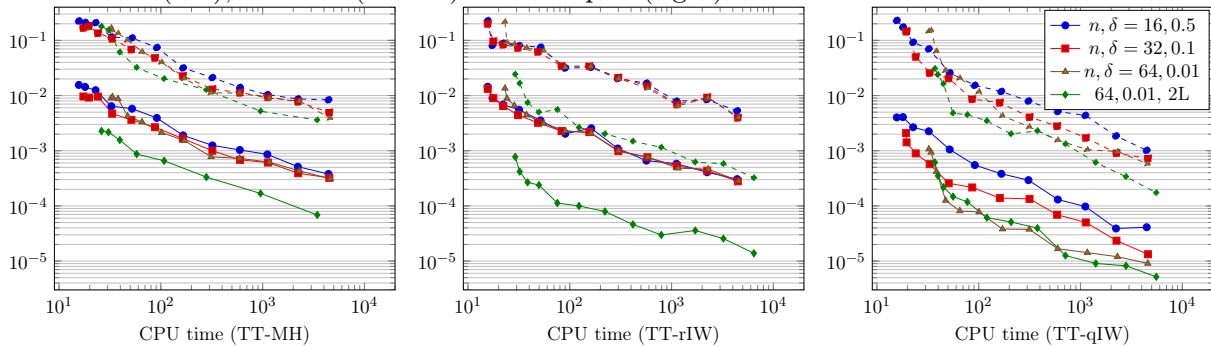

%%%%%%%%%%%%%% sigma_n=1e-3

We also benchmark the algorithms in a more challenging scenario of a smaller noise variance $\sigma_n^2=10^{-3}$.
Due to nonlinearity of the forward model, the posterior density function is concentrated
along a complicated high-dimensional manifold, for smaller $\sigma_n$.
This increases all complexity indicators: the ranks of the TT approximation, the IACT in TT-MH and in DRAM and the variances in the ratio estimators. Since the density function is more concentrated,
we choose finer parameters $n=64$ and $\delta=0.03$ for the TT approximation.
Nevertheless, in Fig. \ref{fig:ff-sn0.001} we see that even though the set-up cost is larger, the TT-based samplers are still all significantly more efficient than DRAM. Due to the stronger concentration of $\pi$, the performance of the basic ratio estimator QMC-rat is worse. On the other hand, the QMC estimator TT-qIW with TT importance weighting is again the most performant method. Note that it is the only method that reduces the quadrature error to the size of the discretization error within the considered limit of one million samples.

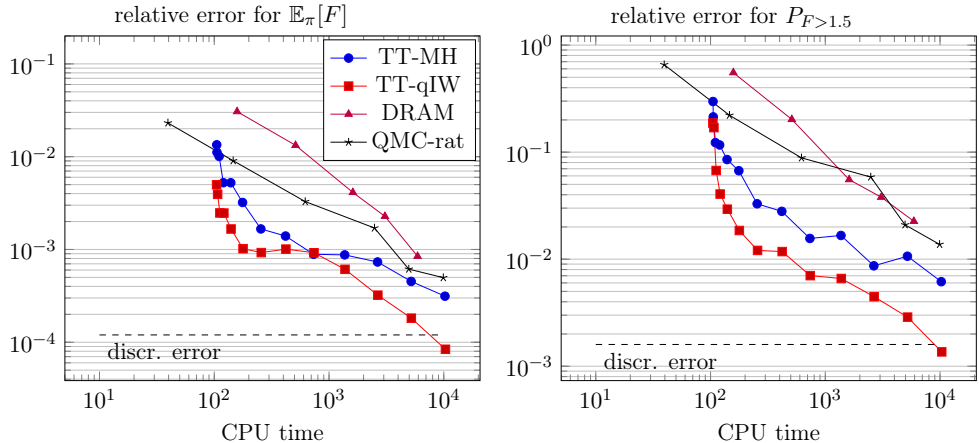
\begin{figure*}
\centering
\caption{Inverse diffusion problem: Relative errors in the mean flux (left) and in the exceedance probability (right) plotted against the total CPU times (sec.) for $\sigma_n^2=10^{-3}$.}
\label{fig:ff-sn0.001}
\resizebox{0.39\linewidth}{!}{%
\begin{tikzpicture}
  \begin{axis}[%
  xmode=log,
  ymode=log,
  ylabel={relative error for $\mathbb{E}_{\pi}[F]$},
  xlabel={CPU time},
  ymax=2e-1,
  legend style={at={(0.99,0.99)},anchor=north east},
  ]
  \addplot+[] coordinates{(   1.3330e+01 +  8.8792e+01 +  3.6088e-01 +  2.6513e+00,   1.3431e-02)
                          (   1.1594e+01 +  8.8666e+01 +  4.4891e-01 +  4.8912e+00,   1.1149e-02)
                          (   1.1592e+01 +  8.8636e+01 +  6.0994e-01 +  9.3956e+00,   1.0079e-02)
                          (   1.1598e+01 +  8.8889e+01 +  9.7644e-01 +  1.8597e+01,   5.2406e-03)
                          (   1.1665e+01 +  8.9051e+01 +  1.7170e+00 +  3.6845e+01,   5.2343e-03)
                          (   1.1552e+01 +  8.8741e+01 +  3.3137e+00 +  7.2883e+01,   3.2112e-03)
                          (   1.1566e+01 +  8.8866e+01 +  7.6022e+00 +  1.4674e+02,   1.6645e-03)
                          (   1.1568e+01 +  8.8819e+01 +  1.9704e+01 +  2.9790e+02,   1.3951e-03)
                          (   1.1570e+01 +  8.8580e+01 +  4.0292e+01 +  5.9392e+02,   8.8740e-04)
                          (   1.1561e+01 +  8.9127e+01 +  8.1178e+01 +  1.1883e+03,   8.7386e-04)
                          (   1.1570e+01 +  8.8882e+01 +  1.6370e+02 +  2.3851e+03,   7.3297e-04)
                          (   1.1616e+01 +  8.8836e+01 +  3.2759e+02 +  4.7685e+03,   4.5257e-04)
                          (   1.1579e+01 +  8.8646e+01 +  6.5778e+02 +  9.4873e+03,   3.1372e-04)
                          }; \addlegendentry{TT-MH};  % with tt_invcdf1, nPy=64, eps=0.03 (err_Pi=3.8710e-01)

% \addplot+[] coordinates{
%                                         (   1.1582e+01 +  8.8919e+01 +  7.9146e-01 +  5.5515e+00,   8.9394e-03)
%                                         (   1.1595e+01 +  8.8730e+01 +  1.1642e+00 +  1.5027e+01,   8.0453e-03)
%                                         (   1.1587e+01 +  8.8566e+01 +  2.3223e+00 +  4.8231e+01,   2.8417e-03)
%                                         (   1.1608e+01 +  8.8673e+01 +  7.3962e+00 +  1.7176e+02,   1.3915e-03)
%                                         (   1.1567e+01 +  8.8606e+01 +  2.8638e+01 +  6.4634e+02,   8.2369e-04)
%                                         (   1.1627e+01 +  8.8935e+01 +  1.5819e+02 +  2.5700e+03,   1.1749e-03)
%                                         (   1.1606e+01 +  8.8734e+01 +  6.9307e+02 +  1.0577e+04,   4.2572e-04)
%                                                                     }; \addlegendentry{TT-qCV};  % -rr;  N1 = 3e-2*N0^2*eps

  \addplot+[] coordinates{(   1.1634e+01 +  9.0110e+01 +  4.3024e-01 +  2.6397e+00,   4.9806e-03)
                          (   1.1566e+01 +  8.9982e+01 +  5.2149e-01 +  4.9376e+00,   3.9291e-03)
                          (   1.1648e+01 +  9.0148e+01 +  6.9160e-01 +  9.4356e+00,   2.4745e-03)
                          (   1.1584e+01 +  9.0007e+01 +  1.0868e+00 +  1.8594e+01,   2.4675e-03)
                          (   1.1590e+01 +  8.9825e+01 +  1.8374e+00 +  3.6858e+01,   1.6613e-03)
                          (   1.1592e+01 +  9.0399e+01 +  3.3070e+00 +  7.2963e+01,   1.0192e-03)
                          (   1.1574e+01 +  9.0318e+01 +  7.8752e+00 +  1.4709e+02,   9.2513e-04)
                          (   1.1622e+01 +  9.0287e+01 +  2.0535e+01 +  2.9860e+02,   1.0086e-03)
                          (   1.1619e+01 +  9.0118e+01 +  4.2098e+01 +  5.9589e+02,   9.2170e-04)
                          (   1.1592e+01 +  9.0020e+01 +  8.3422e+01 +  1.1932e+03,   6.1094e-04)
                          (   1.1572e+01 +  9.0350e+01 +  1.6839e+02 +  2.3937e+03,   3.2171e-04)
                          (   1.1573e+01 +  9.0237e+01 +  3.3436e+02 +  4.7783e+03,   1.8175e-04)
                          (   1.3864e+01 +  9.0075e+01 +  6.7327e+02 +  9.5649e+03,   8.4016e-05)
                         }; \addlegendentry{TT-qIW};

  \addplot+[purple,mark=triangle*,mark options={purple}] coordinates{(1.5806e+02, 3.0619e-02)
                                (5.1087e+02, 1.3259e-02)
                                (1.6139e+03, 4.1177e-03)
                                (3.0717e+03, 2.2759e-03)
                                (5.9127e+03, 8.4436e-04)
                                }; \addlegendentry{DRAM};

  \addplot+[black] coordinates{
                          (3.9494e+01, 2.3107e-02)
                          (1.4621e+02, 9.0435e-03)
                          (6.2069e+02, 3.2819e-03)
                          (2.4901e+03, 1.7034e-03)
                          (4.9564e+03, 6.1367e-04)
                          (9.9083e+03, 4.9681e-04)
                         }; \addlegendentry{QMC-rat};

  \addplot+[no marks,dashed,black, domain=1e1:1e4] {1.1991e-04} node[pos=0.0,anchor=north west] {discr. error};
  \end{axis}
\end{tikzpicture}
}
\resizebox{0.39\linewidth}{!}{%
\begin{tikzpicture}
  \begin{axis}[%
  xmode=log,
  ymode=log,
  ylabel={relative error for $P_{F>1.5}$},
  xlabel={CPU time},
  legend style={at={(0.99,0.99)},anchor=north east},
  ]
  \addplot+[] coordinates{(   1.3330e+01 +  8.8792e+01 +  3.6088e-01 +  2.6513e+00,   2.9717e-01)
                          (   1.1594e+01 +  8.8666e+01 +  4.4891e-01 +  4.8912e+00,   2.1356e-01)
                          (   1.1592e+01 +  8.8636e+01 +  6.0994e-01 +  9.3956e+00,   1.2273e-01)
                          (   1.1598e+01 +  8.8889e+01 +  9.7644e-01 +  1.8597e+01,   1.1671e-01)
                          (   1.1665e+01 +  8.9051e+01 +  1.7170e+00 +  3.6845e+01,   8.5267e-02)
                          (   1.1552e+01 +  8.8741e+01 +  3.3137e+00 +  7.2883e+01,   6.7049e-02)
                          (   1.1566e+01 +  8.8866e+01 +  7.6022e+00 +  1.4674e+02,   3.2925e-02)
                          (   1.1568e+01 +  8.8819e+01 +  1.9704e+01 +  2.9790e+02,   2.7980e-02)
                          (   1.1570e+01 +  8.8580e+01 +  4.0292e+01 +  5.9392e+02,   1.5621e-02)
                          (   1.1561e+01 +  8.9127e+01 +  8.1178e+01 +  1.1883e+03,   1.6661e-02)
                          (   1.1570e+01 +  8.8882e+01 +  1.6370e+02 +  2.3851e+03,   8.6681e-03)
                          (   1.1616e+01 +  8.8836e+01 +  3.2759e+02 +  4.7685e+03,   1.0656e-02)
                          (   1.1579e+01 +  8.8646e+01 +  6.5778e+02 +  9.4873e+03,   6.1626e-03)
                          }; % \addlegendentry{TTrM};  % with tt_invcdf1, nPy=64, eps=0.03 (err_Pi=3.8710e-01)

%   \addplot+[] coordinates{
%                                         (   1.1582e+01 +  8.8919e+01 +  7.9146e-01 +  5.5515e+00,   2.2954e-01)
%                                         (   1.1595e+01 +  8.8730e+01 +  1.1642e+00 +  1.5027e+01,   2.0481e-01)
%                                         (   1.1587e+01 +  8.8566e+01 +  2.3223e+00 +  4.8231e+01,   7.8129e-02)
%                                         (   1.1608e+01 +  8.8673e+01 +  7.3962e+00 +  1.7176e+02,   4.8589e-02)
%                                         (   1.1567e+01 +  8.8606e+01 +  2.8638e+01 +  6.4634e+02,   1.8638e-02)
%                                         (   1.1627e+01 +  8.8935e+01 +  1.5819e+02 +  2.5700e+03,   1.4342e-02)
%                                         (   1.1606e+01 +  8.8734e+01 +  6.9307e+02 +  1.0577e+04,   8.5702e-03)
%                                                                     }; % \addlegendentry{TT2};  % -rr;  N1 = 3e-2*N0^2*eps

  \addplot+[] coordinates{(   1.1634e+01 +  9.0110e+01 +  4.3024e-01 +  2.6397e+00,   1.8702e-01)
                          (   1.1566e+01 +  8.9982e+01 +  5.2149e-01 +  4.9376e+00,   1.6947e-01)
                          (   1.1648e+01 +  9.0148e+01 +  6.9160e-01 +  9.4356e+00,   6.7395e-02)
                          (   1.1584e+01 +  9.0007e+01 +  1.0868e+00 +  1.8594e+01,   4.0595e-02)
                          (   1.1590e+01 +  8.9825e+01 +  1.8374e+00 +  3.6858e+01,   2.9320e-02)
                          (   1.1592e+01 +  9.0399e+01 +  3.3070e+00 +  7.2963e+01,   1.8531e-02)
                          (   1.1574e+01 +  9.0318e+01 +  7.8752e+00 +  1.4709e+02,   1.2079e-02)
                          (   1.1622e+01 +  9.0287e+01 +  2.0535e+01 +  2.9860e+02,   1.1776e-02)
                          (   1.1619e+01 +  9.0118e+01 +  4.2098e+01 +  5.9589e+02,   7.0201e-03)
                          (   1.1592e+01 +  9.0020e+01 +  8.3422e+01 +  1.1932e+03,   6.6003e-03)
                          (   1.1572e+01 +  9.0350e+01 +  1.6839e+02 +  2.3937e+03,   4.4602e-03)
                          (   1.1573e+01 +  9.0237e+01 +  3.3436e+02 +  4.7783e+03,   2.8745e-03)
                          (   1.3864e+01 +  9.0075e+01 +  6.7327e+02 +  9.5649e+03,   1.3615e-03)
                         }; % \addlegendentry{TTqIW};

  \addplot+[purple,mark=triangle*,mark options={purple}] coordinates{(1.5806e+02, 5.5297e-01)
                                (5.1087e+02, 2.0208e-01)
                                (1.6139e+03, 5.5398e-02)
                                (3.0717e+03, 3.7944e-02)
                                (5.9127e+03, 2.2592e-02)
                                }; % \addlegendentry{DRAM};

  \addplot+[black] coordinates{
                          (3.9494e+01, 6.5612e-01)
                          (1.4621e+02, 2.2114e-01)
                          (6.2069e+02, 8.8401e-02)
                          (2.4901e+03, 5.8528e-02)
                          (4.9564e+03, 2.0935e-02)
                          (9.9083e+03, 1.3746e-02)
                         }; % \addlegendentry{QMC};

  \addplot+[no marks,dashed,black, domain=1e1:1e4] {1.5970e-03}  node[pos=0.0,anchor=north west] {discr. error};
  \end{axis}
\end{tikzpicture}
}
\end{figure*}

Finally, we profile the computational cost of all the various components in the TT approaches with respect to the total error (truncation, spatial discretization and quadrature).
We vary the spatial mesh size $h$ from $2^{-5}$ to $2^{-7}$ and estimate the convergence rate of the discretization error (Fig. \ref{fig:ff-d}, left).
Then, we choose the other approximation parameters in order to equilibrate the errors.
In particular, the number of random variables $d$ and the number of samples $N$ are chosen such that the KL truncation error in \eqref{eq:kle_art} and the quadrature error of the TT-qIW method are equal to the discretization error, respectively (see Fig. \ref{fig:ff-d}, left).

The solid lines in Fig. \ref{fig:ff-d} (right) give the computational times necessary for the various components of our algorithm (with all errors equilibrated),
as a function of $d$ (and thus also as a function of $h^{-1}$ and $N$):
the ALS-Cross algorithm to build the TT surrogate of $u_h$,
the TT cross algorithm to build the TT surrogate of $\pi$,
the TT-CD sampling procedure for the $N$ samples $x^\ell$, $\ell=1,\ldots,N$
and the evaluation of $\pi$ at the $N$ samples.
Clearly the $N$ PDE solves in the evaluation of $\pi$ are the dominant part and the complexity of these evaluations grows fairly rapidly due to the spatial mesh refinement and the increase in $N$.
The TT cross algorithm for building $\tilde\pi$ (once a TT surrogate of the forward solution is available) and the cost of the TT-CD sampler depend on the dimension $d$ and on the TT ranks of $\tilde\pi$ (which grow very mildly with $d$ and $h^{-1}$).

In addition, we also ran all the experiments with $h=2^{-6}$ and $N=2^{14}$ fixed, varying only $d$ to explicitly see the growth with $d$.
The timings for these experiments are plotted using dashed lines.
The cost for the ALS-Cross algorithm to build $\tilde{u}_h$ grows cubically in $d$,
while the cost to build the TT surrogate $\tilde\pi$ and the cost of the TT-CD sampling procedure grow linearly with $d$.
Since the evaluation of $\pi$ is dominated by the cost of the PDE solve, its cost does not grow with dimension.
This shows that the TT-CD sampler is an extremely effective surrogate for high dimensions when the model admits a natural extension in $d$ (e.g. it converges as $d\rightarrow \infty$, or the variables remain locally correlated).

\begin{figure*}
\centering
\caption{Inverse diffusion problem: Dimension ($d$) dependence of discretization error and numbers of samples (left) and CPU times of the various algorithmic components in TT-qIW (right); solid lines with equilibrated errors, dashed lines with $h=2^{-6}$ and $N=2^{14}$ fixed.}
\label{fig:ff-d}
\resizebox{0.415\linewidth}{!}{
\begin{tikzpicture}
  \begin{axis}[%
  xmode=log,
  ymode=log,
  xlabel=$d$ (log scale),
  xmin=7,xmax=25,
  xtick={8,12,16,24},
  xticklabels={8,12,16,24},
  ylabel={relative discretisation error},
  legend style={at={(0.01,0.6)},anchor=west},
  y label style={at={(-0.15,1.0)},rotate=0,color={blue},anchor=south west,at={(0.0,1.0)}},every y tick label/.style={blue},
  grid=none,
  ]
  \addplot+[] coordinates{( 9, 6.0460e-04+4.796e-04)
                          (15, 9.3359e-05+1.199e-04)
                          (23, 2.0939e-05+2.998e-05)
                         };
  \end{axis}
  \begin{axis}[%
  xmode=log,
  ymode=log,
  ylabel={\# samples $N$},
  xmin=7,xmax=25,
  xtick={8,12,16,24},
  xticklabels={8,12,16,24},
  legend style={at={(0.99,0.99)},anchor=north east},
  axis y line*=right, y label style={at={(1.1,1.0)},anchor=south east,rotate=0,color={red}},every y tick label/.style={red},
  axis x line=none,
  grid=none,
  ]
  \addplot+[red, mark options={red},mark=+] coordinates{
                                              ( 9, 2^10)
                                              (15, 2^14)
                                              (23, 2^17)
                                            };
  \end{axis}
\end{tikzpicture}
}
\resizebox{0.39\linewidth}{!}{
\begin{tikzpicture}
  \begin{axis}[%
  xmode=log,
  ymode=log,
  xlabel=$d$ (log scale),
  ylabel={CPU time},
  ymax=2e6,
  legend columns=2,
  legend style={/tikz/column 2/.style={column sep=10pt},at={(0.99,0.99)},anchor=north east}
  ]
  \addplot+[blue,solid,mark=*] coordinates{( 9, 5.5055)
                          (15, 26.6886)
                          (23, 142.7671)
                         }; \addlegendentry{ALS-Cross($\tilde u$)};
  \addplot+[red,solid,mark=square*] coordinates{( 9, 2.6447)
                          (15, 4.5219)
                          (23, 7.1680)
                         }; \addlegendentry{TT-Cross($\tilde\pi$)};
  \addplot+[orange,solid,mark=diamond*] coordinates{( 9, 0.1301)
                          (15, 1.3359)
                          (23, 14.9513)
                         }; \addlegendentry{TT-CD($x^\ell$)};
  \addplot+[black,solid,mark=triangle*] coordinates{( 9, 4.3790)
                          (15, 282.8199)
                          (23, 1.0189e+04)
                         }; \addlegendentry{Exact $\pi(x^\ell)$};

  % non-equilibrated experiments
  \addplot+[blue,dashed,mark=o,line width=1pt] coordinates{
                                            ( 9, 8.9671)                     % Errors
                                            (15, 26.8375)                    % 7.6083e-05   2.4604e-03
                                            (22, 107.847)                    % 1.1676e-04   4.1610e-03
                                            (23, 114.701)                    % 9.3372e-05   3.3230e-03
                                            (33, 360.156)                    % 9.3761e-05   3.8641e-03
                                            (36, 497.731)                    % 6.4926e-05   3.4871e-03
                                            (47, 967.644)                    % 6.8732e-05   3.2072e-03
                                            (70, 2456.96)                    % 7.5468e-05   3.7098e-03
                                            (100,5215.22)                    % 8.9585e-05   4.2277e-03
                         }; % als-cross
  \addplot+[red,dashed,mark=square,line width=1pt] coordinates{
                                                ( 9, 2.6773)
                                                (15, 4.5624)
                                                (22, 4.49238)
                                                (23, 4.67722)
                                                (33, 6.5398)
                                                (36, 7.09813)
                                                (47, 8.88098)
                                                (70, 11.5215)
                                                (100,16.1314)
                         }; % tt-cross
  \addplot+[orange,dashed,mark=diamond,line width=1pt] coordinates{
                                                    ( 9, 0.7387)
                                                    (15, 1.3300)
                                                    (22, 1.53788)
                                                    (23, 1.60291)
                                                    (33, 2.04577)
                                                    (36, 2.22716)
                                                    (47, 2.6007)
                                                    (70, 3.22144)
                                                    (100,4.11477)
                         }; % tt-cd
  \addplot+[black,dashed,mark=triangle,line width=1pt] coordinates{
                                           ( 9, 299.9616)
                                           (15, 282.0095)
                                           (22, 282.399)
                                           (23, 281.291)
                                           (33, 284.154)
                                           (36, 283.647)
                                           (47, 283.593)
                                           (70, 278.719)
                                           (100,283.908)
                         }; % \pi
  \end{axis}

% ml    d   tol         lvl     ttimes_solve    ttimes_cross    ttimes_cd   ttimes_pi   err_quad_flux   err_quad_prob
% 2     9   4.796e-04   10      7.0871          28.788          1.9951      70.349      6.0460e-04      2.6481e-02
% 3    15   1.199e-04   14      2.6621e+01      4.4212e+01      1.6955e+01      4.5463e+03  9.3359e-05      2.1354e-03
% 4    23   2.998e-05   17      1.4232e+02      6.4716e+01      1.8452e+02      1.6483e+05  2.0939e-05      1.0492e-03
% 5    36   7.494e-06   20

\end{tikzpicture}
}
\end{figure*}
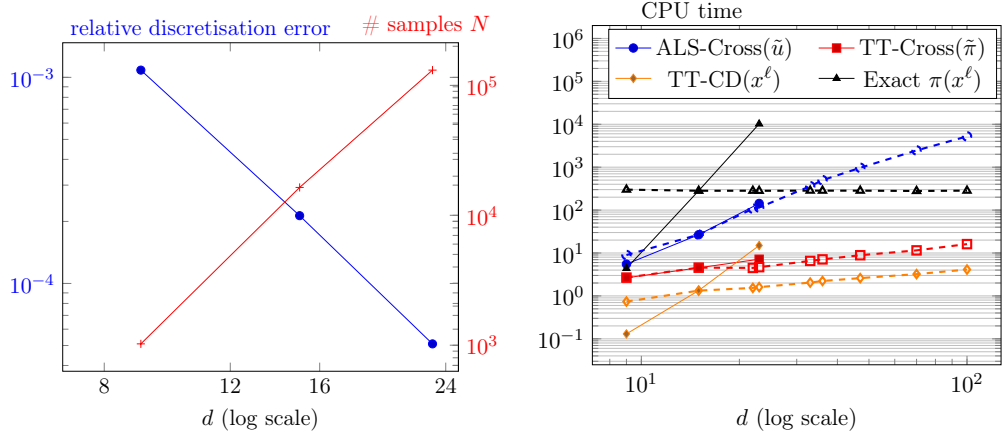

% \section{Software}
% We implemented Algorithm \ref{alg:samp} in Matlab and C+Python, using the TT-Toolbox in Matlab \cite{tt-toolbox} and Python\footnote{Available at \url{http://github.com/oseledets/ttpy}}, respectively.
% The code is available at \url{http://github.com/dolgov/ttirt}
% This research made use of the Balena High Performance Computing (HPC) Service at the University of Bath.
% Numerical experiments were carried out in Matlab R2016b on one core of the Balena cluster per run, an Intel Xeon E5-2650 CPU.

\section{Conclusion}
We presented a method for computational inference based on function approximation of the target PDF. That task has traditionally been viewed as infeasible for general multivariate distributions due to the exponential growth in cost for grid-based representations. The advent of the tensor train representation, amongst other hierarchical representations, is a significant development that circumvents that `curse of dimensionality'. Our main contributions here have been showing that the conditional distribution method can be implemented efficiently for PDFs represented in (interpolated) TT format, and that quasi-Monte Carlo quadrature is both feasible and efficient with bias correction through a control-variate structure or via importance weighting. The latter scheme was most efficient across all computed examples and parameter choices.

We adapted existing tools for \emph{tensors}, i.e., multi-dimensional arrays, in particular the TT cross approximation scheme, and tools for basic linear algebra. We expect that substantial improvement could be achieved with algorithms tailored for the  specific tasks required, such as function approximation, and the setting of coordinates and bounding region. Nevertheless, the algorithms presented are already very promising, providing sample-based inference that is more computationally efficient than a benchmark MCMC, the DRAM MCMC. We demonstrated the algorithms in three stylized examples: a time-to-failure model; an inverse problem; and sampling from a non-Gaussian PDF. Extensive computations showed that in each example the methods performed as theory predicts, and that scaling with dimension is linear.

We view the methods developed here as a promising development in Markov chain Monte Carlo methods. It is noteworthy, however, that our most efficient algorithm (TT-qIW), implements neither a Markov chain for the basic sampler, nor uses standard Monte Carlo quadrature. Instead, points from a randomized quasi-Monte Carlo (QMC) lattice are mapped into state space by the inverse Rosenblatt transform, implemented in the TT-CD algorithm, with unbiased estimates available via importance-weighted QMC quadrature.  Nevertheless, the basic structure remains a proposal mechanism that is modified to produce a sequence of points that is ergodic for the target distribution.%, as in existing MCMC.

%\section{Software}
Numerical experiments were carried out in Matlab R2016b on an Intel Xeon E5-2650 CPU at the Balena High Performance Computing Service at the University of Bath, using one core per run.
We implemented Algorithm~\ref{alg:samp} in Matlab and C+Python, using the TT-Toolbox in Matlab \cite{tt-toolbox} and Python (available at
% \emph{[Blinded as requested by the Editor]}
\url{http://github.com/oseledets/ttpy}
), respectively. The code is available at
\url{http://github.com/dolgov/tt-irt};
% \emph{[Blinded as requested by the Editor]};
we welcome suggestions or feedback from users.\vspace{2ex}

\noindent
\textbf{Acknowledgments.}
% \emph{[Blinded as requested by the Editor]}
SD is grateful for the support from the Engineering and Physical Sciences Research Council (EPSRC) through Fellowship EP/M019004/1. This research was started while CF was Global Chair in the Institute for Mathematical Innovation (IMI) at the University of Bath. %; CF is grateful to the IMI for associated funding.

\bibliographystyle{siam}
\bibliography{bayes-tt}
\end{document}